\documentclass{amsart}
\usepackage{amssymb}

\usepackage[T1]{fontenc}
\usepackage{mathpazo,courier}
\usepackage[scaled]{helvet}

\usepackage[usenames,dvipsnames]{xcolor}
\definecolor{darkblue}{rgb}{0.0, 0.0, 0.55}
\usepackage[pagebackref,colorlinks,linkcolor=BrickRed,citecolor=OliveGreen,urlcolor=darkblue,hypertexnames=true]{hyperref}

\usepackage{enumerate}

\newcommand\N{\mathbb N}

\newcommand\R{\mathbb R}
\newcommand\C{\mathbb C}

\usepackage{bbm} 
\newcommand\ii{\mathbbm i}

\newcommand\al\alpha
\newcommand\be\beta
\newcommand\ga\gamma
\newcommand\la\lambda
\newcommand\de\delta
\newcommand\ep\varepsilon
\newcommand\ph\varphi
\newcommand\si\sigma
\newcommand\ta\tau
\newcommand\rh\varrho
\newcommand\io\iota
\newcommand\La\Lambda

\DeclareMathOperator\trunc{trunc}

\DeclareMathOperator\tr{tr}
\DeclareMathOperator\hur{hur}

\DeclareMathOperator\re{Re}

\theoremstyle{definition}
\newtheorem{thm}{Theorem}[section]

\newtheorem{cor}[thm]{Corollary}
\newtheorem{pro}[thm]{Proposition}
\newtheorem{rem}[thm]{Remark}
\newtheorem{ex}[thm]{Example}
\newtheorem{lem}[thm]{Lemma}
\newtheorem{df}[thm]{Definition}

\newtheorem{conjecture}[thm]{Conjecture}
\newtheorem{problem}[thm]{Problem}

\title{Spectrahedral relaxations of hyperbolicity cones}

\author[M. Schweighofer]{Markus Schweighofer}
\address{Fachbereich Mathematik und Statistik, Universität Konstanz, 78457 Konstanz, Germany}
\email{markus.schweighofer@uni-konstanz.de}


\date{July 25, 2023}


\begin{document}
\begin{abstract}
Let $p$ be a real zero polynomial in $n$ variables. Then $p$ defines a rigidly convex set $C(p)$.
We construct a linear matrix inequality of size $n+1$ in the same $n$ variables that depends only on the cubic part of $p$
and defines a spectrahedron $S(p)$ containing $C(p)$. The proof of the containment uses the characterization of real zero polynomials in
two variables by Helton and Vinnikov. We exhibit many cases where $C(p)=S(p)$.

In terms of optimization theory, we introduce a small semidefinite relaxation of a potentially huge hyperbolic program. If the hyperbolic program is a
linear program, we introduce even a finitely convergent hierachy of semidefinite relaxations.
With some extra work, we discuss the homogeneous setup where real zero polynomials correspond to homogeneous polynomials and
rigidly convex sets correspond to hyperbolicity cones.

The main aim of our construction is to attack the generalized Lax conjecture saying that
 $C(p)$ is always a spectrahedron. We show that the ``weak real zero amalgamation conjecture'' of Sawall and the author would imply the
 following partial result towards the generalized Lax
conjecture: Given finitely many planes in $\R^n$, there is a spectrahedron containing $C(p)$ that coincides with $C(p)$ on each of these planes.
This uses again the result of Helton and Vinnikov.
\end{abstract}

\maketitle

\newpage
\tableofcontents


\section{Instead of an introduction}

I started to write these notes during my sabbatical winter term 2018/2019 but still did not finish them due
to lack of time. The notes are thus still incomplete and probably contain numerous errors. I do not yet know
what is the best way to publish them. All readers are kindly invited to report any errors
(from typographic to fatal), e.g., by electronic mail to:
\[\texttt{markus.schweighofer@uni-konstanz.de}\]
Suggestions and remarks are also highly welcome.
I will try to correct and expand these notes in the future. This version is still far from being suitable for
publication, probably parts of it will be caught up in ongoing projects, also with other collaborators.
This third version of the notes does no longer contain much of the material about real zero amalgamation which
has now become part of the preprint \cite{ss} with David Sawall. That is why the preprint got shorter compared to its second version
[\url{https://arxiv.org/abs/1907.13611v2}] that contains still a version of the ``real zero amalgamation conjecture'' which has been disproved in the meanwhile \cite{ss}.

The structure of the notes or even the title
will change in the future.
If you want to reference some content in the future,
please check for newer or even published versions of the text.
I nevertheless decided to make the notes available on the \texttt{arXiv} preprint server
since I plan to give some talks about the material.

This article lacks a proper introduction. We consider rigidly convex sets in the affine setting and
hyperbolicity cones in the homogeneous setting \cite{vin2}. While we will formally
introduce these notions, the reader who is not yet familiar with them should
first have a look at some survey articles \cite{wag,pem,fis}. We will assume that
the reader knows about the definition and the most basic properties of spectrahedra
\cite{lau}.

An extremely important theorem, the Helton-Vinnikov theorem \cite{hv,vin1,han}
and some weaker versions or predecessors of it \cite{gkvw,pv} (see \cite[§8]{hv} for the complicated history)
show that in dimension two each rigidly convex set is a spectrahedron.
In the homogeneous setting, it shows that in dimension three each hyperbolicity cone is
spectrahedral. To my knowledge,
all existing proofs of this theorem are very deep \cite{hv,vin1,han} and not easily accessible.
Although it complicates a bit our approach, we therefore take care to use whenever possible only
a weaker version of the theorem
(with hermitian instead of symmetric matrices) whose proof is considerably
easier. Perhaps the most elementary proof of it can be found in \cite{gkvw}.

The open question whether, regardless of the dimension,
every rigidly convex set is a spectrahedron, or in the homogeneous setting
whether each hyperbolicity cone is spectrahedral, is now widely known as the
generalized Lax conjecture (GLC). This conjecture that has motivated our work
has first been formulated in \cite[Subsection 6.1]{hv} together with
a stronger conjecture which turned out to be false \cite[Proposition 7]{vin2} as Brändén showed
\cite{bra1}.

There are a number of very interesting partial results towards the generalized
Lax conjecture, see \cite{hv,bk,bra2,ami,kum2,kum3,sau} and
the references therein. Some of the constructions in these references lead to huge linear matrix inequalities
describing hyperbolicity cones \cite{bra2,ami}. In the light of the recent result \cite{rrsw}, the huge size of the matrices is no longer surprising although \cite{rrsw} does says little for these concrete examples. Under the complexity-theoretic assumption $\mathbf{VP}\ne\mathbf{VNP}$ which is an algebraic analog of $\mathbf{P}\ne\mathbf{NP}$, Oliveira
\cite{oli} showed loosely speaking that the hyperbolicity cones Amini's multivariate matching polynomials \cite{ami} cannot in general not
be defined by small linear matrix inequalities.

Our basic idea here is that we produce very small size natural spectrahedral outer approximations
to rigidly convex sets and hyperbolicity cones. To the best of our knowledge, these are new although they
turn out to be more or less known for quadratic real zero and hyperbolic polynomials and they are related
to a construction of \cite{sau} in the case of the ``general'' spectrahedron.
The fact that our constructions yield actually a relaxation seems to be non-trivial and relies on
the Helton-Vinnikov theorem or some weaker versions of it. In view of the lower complexity
bounds proved in \cite{rrsw,oli} for the size of a describing linear matrix inequality, our approach seems
at first sight hopeless since we produce linear matrix inequalities of very small size.

However, there is an important additional twist which allows us to prove a partial result and leaves
room for speculations whether GLC could be attacked in the same way.
The basic idea is to try to apply our construction to a new real zero polynomial in the original variables and
many more new variables. The new real zero polynomial should yield back the original one when all
new variables are set to zero. Due to the big number of new variables, the linear matrix inequality
we now produce has huge size in perfect accordance with the results of \cite{rrsw,oli}.

In \cite{ss}, the authors present the ``real zero amalgamation'' conjecture (as said above, a stronger version has been
conjectured in the last version of this preprint [\url{https://arxiv.org/abs/1907.13611v2}] but has been disproved in \cite{ss}).
It would allow to ``amalgamate'' two real zero polynomials with only two shared variables
that agree whenever you set non-common variables to zero. By ``amalgamating''
we mean to find a third real zero polynomial that yields the first or second polynomial, respectively,
when you set all the non-corresponding variables to zero.

Our main result is that the mentioned amalgamation conjecture would have implications on GLC.
We show that it would imply that each rigidly convex set can be ``wrapped'' into a spectrahedron.
By ``wrapping'' we mean that the rigidly convex set can be packed into a spectrahedron which
is tied to the rigidly convex set with finitely many cords. A bit less vaguely,
the cords are one-dimensional curves lying on the boundary of both, the rigidly convex set and
the spectrahedron. Mathematically speaking, for
each rigidly convex set and finitely given many given planes (two-dimensional subspaces)
there would always exist a spectrahedron containing the rigidly convex set and coinciding with it on the
given planes. In proving this, we use once more the Helton-Vinnikov theorem or a weaker version of it.

The prototype of all rigidly convex sets are polyhedra. Curiously, our work can even be applied to
yield a spectrahedral relaxation of a polyhedron. This makes sense if the polyhedron has a large
number of facets since we produce a small size linear matrix inequality. In this case, we can even
make a hierachy of relaxations out of our construction which converges finitely. Just like in the
Lasserre hierarchy of semidefinite relaxations for polynomial optimization problems \cite{lau},
moment and localization matrices play a big role. Here, we use moment and localization matrices filled
with real numbers to define our relaxation. In the Lasserre hierarchy, the matrices are filled with unknowns
which one hopes to become moments when solving the semidefinite program. Moment matrices play thus a
completely different role here than in the Lasserre hierarchy.

What seems to be strange at first sight is that in general (in the non-LP-case)
our relaxations are based on ``moments'' (actually a kind of pseudo-moments which is different
from the notion of pseudo-moments in the Lasserre context) up to degree three only.
This makes however a lot of sense for people that are acquainted with tensor decomposition methods
like Jennrich's algorithm \cite{har} many of which are also based on degree three moments.

Our ``pseudo-moments'' correspond to real zero polynomials. When the real zero polynomial is a
product of linear polynomials, these are actual moments of a sum of Dirac measures. When the
real zero polynomial has more generally a certain symmetric determinantal representation, we deal
with certain ``tracial moments''. 

Some of our results, like Proposition \ref{rotate} or Proposition \ref{transform} are related to the question
if there is a generalization of certain operations from moments to pseudo-moments. Namely, one
can rotate and shift a point-configuration and the behavior of the
moments of the corresponding sum of Dirac measures is mimicked. This will be another line of future
research whose applications are yet unclear.

Let us say a word why we work mainly with rigidly convex sets before going to hyperbolicity cones
(i.e., we prefer, at least in a first stage, to work in the affine setting rather than in the homogeneous
one): Comparing Definitions \ref{dfpencil} and \ref{hbdfpencil}, it becomes clear that
in the homogeneous setup, there are more technical steps in our construction which makes also
all the proofs a bit more confusing. Moreover, GLC can be perfectly
formulated in the setting of rigidly convex sets. Also, rigidly convex sets of very small dimension
can more easily be visualized. In general, we feel the more natural setup for our method is the
affine one although the homogeneous one is also important and cannot be simply deduced from the
affine one.

\subsection{Numbers}
We always write $\N$ and $\N_0$ for the sets of positive and nonnegative integers, $\R$ and $\C$ for the fields of real and complex numbers, respectively. We denote the complex imaginary unit by $\ii$ so that
the letter $i$ can be used for other purposes such as summations.

\subsection{Polynomials and power series}
Let $K$ be a field. We will always denote by $x=(x_1,\ldots,x_n)$ an $n$-tuple of distinct variables so that
\[K[x]:=K[x_1,\dots,x_n]\qquad\text{and}\qquad K[[x]]:=K[[x_1,\dots,x_n]]\] denote the rings of polynomials and (formal) power series in $n$ variables over $K$,
respectively. The number $n$ of variables will thus often be fixed implicitly although suppressed from the notation. Sometimes, we need
an additional variable $x_0$ so that $K[x_0,x]=K[x_0,\ldots,x_n]$ denotes the ring of polynomials in $n+1$ variables over $K$.
We allow of course the case $n=0$ since it is helpful to avoid case distinctions, for example in inductive 
proofs. We also use the letter $t$ to denote another variable so that $K[t]$ and $K[[t]]$
denote the rings of univariate polynomials and formal power series over $K$, respectively.
For $n\in\N_0$ and $\al\in\N_0^n$, we denote
\[|\al|:=\al_1+\ldots+\al_n\]
and call
\[x^\al:=x_1^{\al_1}\dotsm x_n^{\al_n}\]
a monomial. Hence
\[K[[x]]=\left\{\sum_{\al\in\N_0^n}a_\al x^\al\mid a_\al\in K\right\}.\]
The degree $\deg p$ of a power series $p=\sum_{\al\in\N_0^n}a_\al x^\al$ with all $a_\al\in K$ is defined as a supremum in the (naturally) ordered set
$\{-\infty\}\cup\N_0\cup\{\infty\}$ by
\[\deg p:=\sup\{|\al|\mid a_\al\ne0\}\in\{-\infty\}\cup\N_0\cup\{\infty\}\]
which entails $\deg 0=-\infty$ and
\[K[x]=\{p\in K[[x]]\mid\deg p<\infty\}.\]
We call a polynomial constant, linear, quadratic, cubic and so on if its degree is \emph{less than or equal to} $0$, $1$, $2$, $3$ and so on.

\subsection{Matrices}
For any set $S$, we denote by $S^{m\times n}$ the set of all matrices with $m$ columns and $n$ rows
over $S$. It is helpful to allow the \emph{empty matrix} which is for each $m,n\in\N_0$
the only element of $S^{m\times 0}=S^{0\times n}$. For any matrix $A\in S^{m\times n}$, we denote
by $S^T\in S^{n\times m}$ its transpose.
For a matrix $A\in\C^{m\times n}$, we denote by $A^*\in\C^{n\times m}$ its
\emph{complex conjugate transpose}.

We tend to view $n$-tuples
over a set $S$, i.e., elements of $S^n$, as \emph{column vectors}, i.e., elements of $S^{n\times 1}$.
In this sense, we can transpose column vectors and get \emph{row vectors} or vice versa. Analogously,
we can apply the complex conjugate transpose to complex column vectors to get row vectors and vice 
versa.

If $R$ is a commutative ring (where a ring in our sense is always associative with $1$),
$d\in\N_0$ and $A\in R^{d\times d}$, then the determinant $\det A$ is declared in the usual way.
For each $n\in\N_0$,
we denote by $I_n\in R^{d\times d}$ the unity matrix of size $d$ for which $\det I_d=1$ (even when
$d=0$, i.e., the determinant of the empty matrix is one).

A matrix $A\in\C^{d\times d}$ is called
\emph{hermitian} if $A=A^*$ which is equivalent to $v^*Av\in\R$ for all $v\in\C^d$.
A matrix $U\in\C^{d\times d}$ is called \emph{unitary} if $U^*U=I_d$ which is equivalent to $UU^*=I_d$ and also to $\|Uv\|=\|v\|$ for all $v\in\C^d$. We will use many times and often without
mentioning the spectral theorem for hermitian matrices that says that
for any hermitian $A\in\C^{d\times d}$ there exists a unitary $U\in\C^{d\times d}$ such that
$U^*AU$ is a diagonal matrix. In this situation, the diagonal matrix is obviously real and
the diagonal elements are the eigenvalues of $A$.

A hermitian matrix all of whose (real) eigenvalues are
nonnegative is called \emph{positive semidefinite (psd)}. If all its eigenvalues are even positive,
it is called \emph{positive definite (pd)}. If all its eigenvalues are nonpositive or negative
then it is called \emph{negative semidefinite (nsd)} or \emph{negative definite (nd)}, respectively.
A \emph{definite matrix} is one that is pd or nd. It is easy to see that a matrix $A\in\C^{d\times d}$
is psd if and only if $v^*Av\ge0$ (i.e., $v^*Av$ is real and nonnegative) for all $v\in\C^n$. The analogous
characterizations for pd, nsd and nd matrices should be clear. Using the intermediate value theorem,
it follows that $A\in\C^{d\times d}$ is definite if and only if $v^*Av\in\R\setminus\{0\}$ for all
$v\in\C^d\setminus\{0\}$.

For matrices $A,B\in\C^{d\times d}$, we write
$A\preceq B$ (or equivalently $B\succeq A$) if $B-A$ is psd and
$A\prec B$ (or equivalently $B\succ A$) if $B-A$ is pd.

A real matrix $A\in\R^{d\times d}$ is of course symmetric if and only if it is hermitian.
It is called \emph{skew-symmetric} if $A=-A^T$. A real unitary matrix
is called \emph{orthogonal}. A matrix $U\in\R^{d\times d}$ is thus obviously orthogonal if $\|Uv\|=\|v\|$
for all $v\in\R^d$. The spectral theorem for symmetric matrices that we will often use silently says that
for any symmetric $A\in\R^{d\times d}$ there exists an orthogonal $U\in\R^{d\times d}$ such that
$U^*AU$ is a diagonal matrix. A real matrix is obviously psd if and only if
it is symmetric and its (real) eigenvalues are nonnegative.
It is thus easy to see that a matrix $A\in\R^{d\times d}$
is psd if and only if $A$ is symmetric and $v^*Av\ge0$ for all $v\in\R^n$. The analogous
characterizations for pd, nsd and nd matrices are clear. Moreover,
$A\in\R^{d\times d}$ is definite if and only if $A$ is symmetric and $v^*Av\ne0$ for all
$v\in\R^d\setminus\{0\}$. In particular, the notation $A\succeq0$ means for $A\in\R^{d\times d}$
that $A$ is symmetric and $v^*Av\ge0$ for all $v\in\R^d$.

If $A,B\in\R^{d\times d}$, then $C:=A+\ii B\in\C^{d\times d}$ is hermitian if and only if
$A$ is symmetric and $B$ is skew-symmetric, or in other words, if the matrix
\[R:=\begin{pmatrix}A&-B\\B&A\end{pmatrix}\in\R^{(2d)\times(2d)}\]
is symmetric.
Now if $C$ is hermitian, one easily checks that \[(v+\ii w)^*C(v+\ii w)=
\begin{pmatrix}v\\w\end{pmatrix}^TR\begin{pmatrix}v\\w\end{pmatrix}\]
for all $v,w\in\R^n$.
Hence it is clear that $C$ is psd if and only if $R$ is psd.

\subsection{Matrix polynomials}

Matrices whose entries are polynomials are often called \emph{matrix polynomials}.
As said above, we allow the empty matrix and therefore the empty matrix polynomial.
We define the \emph{degree} of the empty matrix polynomial to be $\infty$ and the degree
of a non-empty matrix polynomial to be the maximum of the degrees of its entries.
Hence the degree of a matrix polynomial is always from $\N_0$ except for the zero matrix polynomial
which has degree $\infty$. Exactly as for polynomial, we say that a matrix polynomial is
constant, linear, quadratic, cubic and so on if its degree is \emph{less than or equal to} $0$, $1$, $2$, $3$
and so on. In numerical linear algebra, linear matrix polynomials are often called matrix pencils, especially
if they are univariate, i.e., only one variables is involved. Here it will be convenient to reserve the term
\emph{pencil} for \emph{symmetric} linear real matrix polynomials in one or several variables. A pencil
of size $d$ in $n$ variables $x=(x_1,\ldots,x_n)$ is thus of the form
\[A_0+x_1A_1+\ldots+x_nA_n\]
where $A_0,A_1,\ldots,A_n\in\R^{d\times d}$ are symmetric. The determinant of such a pencil is of
course a polynomial of degree at most $d$.

\subsection{Convex sets and cones} A subset $C$ of a real vector space $V$
is called \emph{convex} if it contains with any two of its points $v,w\in C$ also the line segment
\[\{\la v+(1-\la)w\mid\la\in[0,1]\}\]
joining them. We call it a \emph{cone} what many authors call a ``convex cone'', namely a subset
$C$ of a real vector space $V$ that contains the origin, is closed under addition and under multiplication
with nonnegative scalars, i.e., $0\in C$, $v+w\in C$ for all $v,w\in C$ and $\la v\in C$ for all $v\in C$
and $\la\ge0$. For example, the set \[\{A\in\R^{d\times d}\mid A\succeq0\}\]
of psd matrices is a cone inside the vector space $\R^{d\times d}$ of matrices of size $d$.
Most of the convex sets and cones we will consider will however live in the vector space $\R^d$.

\subsection{Affine half spaces, polyhedra and spectrahedra}
A subset of $\R^n$ of the form $\{a\in\R^n\mid\ell(a)\ge0\}$ where $\ell\in\R[x]=\R[x_1,\ldots,x_n]$ is
a non-constant linear polynomial is called an \emph{affine half space}. A finite intersection
of such half spaces is called a \emph{polyhedron}. A subset of $\R^n$ of the form
$\{a\in\R^n\mid L(a)\succeq0\}$ where $L\in\R[x]^{d\times d}$ is a pencil of size $d$
for some $d\in\N_0$ is called a \emph{spectrahedron}. Hence polyhedra are exactly the spectrahedra
that can be defined by diagonal pencils. Spectrahedra are more flexible than polyhedra and unlike
polyhedra allow for round shapes in their geometry. On the other hand, they still share many good properties with polyhedra. Restating the definition more explicitly, 
$S\subseteq\R^n$ is a spectrahedron if and only
if there exists $d\in\N_0$ and symmetric matrices $A_0,A_1,\ldots,A_n\in\R^{d\times d}$ such that
\[S=\{a\in\R^n\mid A_0+a_1A_1+\ldots+a_nA_n\succeq0\}.\]
Here one could equivalently require $A_0,A_1,\ldots,A_n$ to be complex hermitian matrices instead
of symmetric real matrices as can be seen easily by the above translation of a psd condition for a
complex matrix into a psd condition of a real matrix of double size.
Using block diagonal matrices, one sees immediately that finite intersections of spectrahedra are
again spectrahedra.

\subsection{Elementary combinatorics} We use some standard notation from elementary 
combinatorics. For $\ell\in\N_0$, the factorial of $\ell$
\[\ell!:=\ell(\ell-1)\dotsm1\]
denotes the number of permutations of $\ell$ objects (in particular $0!=1$).
For $k,\ell\in\N_0$, the binomial coefficient
\[\binom\ell k:=\frac{\ell!}{(\ell-k)!k!}\]
denotes the number of choices of $k$ objects among $\ell$. For $\al\in\N_0^n$, the
multinomial coefficient
\[\binom{|\al|}\al:=\binom{\al_1+\ldots+\al_n}{\al_1\ldots\al_n}:=\frac{|\al|!}{\al_1!\ldots\al_n!}\]
denotes the number of ways of depositing $|\al|$ distinct objects into n distinct bins, with $\al_i$
objects in the $i$-th bin for each $i\in\{1,\ldots,n\}$. For $n=1$ this notation agrees with
the one for binomial coefficients which fortunately does not lead to a conflict.


\section{Real zero polynomials and spectrahedra}

\subsection{Definition and examples}

The following definition stems from \cite[§2.1]{hv}.

\begin{df}\label{dfrz}
We say that $p\in\R[x]$ is a \emph{real zero polynomial} if for all $a\in\R^n$ and $\la\in\C$,
\[p(\la a)=0\implies\la\in\R.\]
\end{df}

\begin{rem}\label{nonzeroatorigin}
If $p\in\R[x]$ is a real zero polynomial, then $p(0)\ne0$.
\end{rem}

\begin{pro}\label{splitrz}
Let $p\in\R[x]$. Then $p$ is a real zero polynomial if and only if for each $a\in\R^n$, the univariate polynomial
\[p(ta)=p(ta_1,\ldots,ta_n)\in\R[t]\]
splits (i.e., is a product of non-zero linear polynomials) in $\R[t]$.
\end{pro}

\begin{proof}
The ``if'' direction is easy and the ``only if'' direction follows from the fundamental theorem of algebra. 
\end{proof}

\begin{ex}\label{quadraticrealzero}
Let $p\in\R[x]$ be a quadratic real zero polynomial with $p(0)=1$. Then $p$ can be uniquely written as
\[p=x^TAx+b^Tx+1\] with a symmetric matrix $A\in\R^{n\times n}$ and a vector $b\in\R^n$. For $a\in\R^n$ the univariate
quadratic polynomial $p(ta)=a^TAat^2+b^Tat+1$ splits in $\R[t]$ if and only if its discriminant $(b^Ta)^2-4a^TAa=a^T(bb^T-4A)a$ is
is nonnegative. Hence $p$ is a real zero polynomial if and only if \[bb^T-4A\succeq0.\]
\end{ex}

\begin{lem}\label{genev}
Let $A,B\in\C^{d\times d}$ be hermitian and suppose $A$ is definite. Then for all $\la\in\C$,
\[\det(A+\la B)=0\implies\la\in\R.\]
\end{lem}

\begin{proof}
The case $A\prec0$ can be reduced to the case $A\succ0$ by scaling $p$ with $(-1)^d$.
WLOG $A\succ0$. Since there is a unitary matrix $U\in\C^{d\times d}$ such that $U^*AU$ is diagonal, we can assume that $A$ is diagonal.
The diagonal entries $d_1,\dots,d_n$ of $A$ are positive. Multiplying both $A$ and $B$ from the left and the right by the diagonal matrix whose
diagonal entries are the inverted square roots of $d_1,\dots,d_n$ changes the determinant of $A+\la B$ but preserves the condition
$\det(A+\la B)=0$. Hence WLOG $A=I_d$. Now $\det(B-(-\frac1\la)I_d)=0$. But then $-\frac1\la$ is an
eigenvalue of the hermitian matrix $B$ and thus real. Hence $\la$ is real as well.
\end{proof}

The most obvious example of real zero polynomials are products of linear polynomials that do not vanish at the origin. But this is just the special case where
all $A_i$ are diagonal matrices of the following more general example:

\begin{pro}\label{detexample}
Let $A_0,A_1,\ldots,A_n\in\C^{d\times d}$ be hermitian matrices such that $A_0$ is definite and
\[p=\det(A_0+x_1A_1+\ldots+x_nA_n),\]
then $p$ is a real zero polynomial.
\end{pro}

\begin{proof}
If $a\in\R^n$ and $\la\in\C$ with $p(\la a)=0$, then $\det(A_0+\la B)=0$ where $B:=a_1A_1+\ldots+a_nA_n\in\C^{d\times d}$ and Lemma \ref{genev}
implies $\la\in\R$.
\end{proof}

\subsection{The Helton-Vinnikov theorem}

The following celebrated partial converse to Proposition \ref{detexample}
has been obtained in 2006 by Helton and Vinnikov \cite[Theorem 2.2, §4]{hv}.
For a short account of the long history of partial results, we refer to \cite[§8]{hv}.
Recently, a purely algebraic proof of this theorem has been given by Hanselka \cite[Theorem 2]{han}.

\begin{thm}[Helton and Vinnikov]\label{vinnikov}
If $p\in\R[x_1,x_2]$ is a real zero polynomial of degree $d$ with $p(0)=1$, then there exist symmetric $A_1,A_2\in\R^{d\times d}$ such that
\[p=\det(I_d+x_1A_1+x_2A_2).\]
\end{thm}

A weaker version of this theorem will be enough for most of our purposes.
Since this weaker version seems to be considerably easier to prove
(see mainly \cite{gkvw}, also \cite{pv} and \cite[§7]{han}), we state it here.
this weaker version as a corollary. In the following, we will always prefer to use the corollary instead of the theorem.

\begin{cor}[Helton and Vinnikov]\label{vcor}
If $p\in\R[x_1,x_2]$ is a real zero polynomial of degree $d$ with $p(0)=1$, then there exist hermitian matrices $A_1,A_2\in\C^{d\times d}$ such that
\[p=\det(I_d+x_1A_1+x_2A_2).\]
\end{cor}

\begin{ex}\label{hv2}
Consider a quadratic real zero polynomial $p\in\R[x_1,x_2]$ with $p(0)=1$. Write
\[p=x^TAx+b^Tx+1\] with a symmetric matrix $A\in\R^{n\times n}$ and a vector $b\in\R^n$.
Write moreover \[A=\begin{pmatrix}a_{11}&a_{12}\\a_{12}&a_{22}\end{pmatrix}\qquad\text{and}\qquad
b=\begin{pmatrix}b_1\\b_2\end{pmatrix}.\]
Then $bb^T-4A\succeq0$ by Example \ref{quadraticrealzero}. Hence the leading prinicipal minors $r:=b_1^2-4a_{11}$
and $s:=\det(bb^T-4A)$ are nonnegative. Consider the real symmetric matrices
\[A_1:=\frac12\det\begin{pmatrix}b_1-\sqrt r&0\\0&b_1+\sqrt r\end{pmatrix}\]
and
\[A_2:=\frac1{2r}
\begin{pmatrix}b_1^2b_2-b_1b_2\sqrt{r}-4 a_{11}b_2+4a_{12}\sqrt{r}&\sqrt{rs}\\
   \sqrt{rs}&b_1^2 b_2+b_1 b_2 \sqrt{r}-4 a_{11}b_2-4a_{12}\sqrt{r}\end{pmatrix}.
\]
Then one can easily verify that $p=\det(I_2+x_1A_1+x_2A_2)$.
\end{ex}

\subsection{Rigidly convex sets}

\begin{df}\label{dfrzs}
Let $p\in\R[x]$ be a polynomial.
Then we call
\[Z(p):=\{a\in\R^n\mid p(a)=0\}\]
the (real) \emph{zero set} defined by $p$.
\end{df}

\begin{df}\label{dfrcs}
Let $p\in\R[x]$ be a real zero polynomial.
Then we call
\[C(p):=\{a\in\R^n\mid\forall\la\in[0,1):p(\la a)\ne0\}\]
the \emph{rigidly convex set} defined by $p$.
\end{df}

A priori, it is not even clear that rigidly convex sets are convex. This was however already known to
G\aa rding, see Theorem \ref{rc-is-c} below. In the case where $p$ has a determinantal
representation of the kind considered in Proposition \ref{detexample} above, it is however easy to show that $C(p)$ is not only convex but even a spectrahedron:

\begin{pro}\label{rcsdet}
Suppose $d\in\N_0$, $A_0,A_1,\ldots,A_n\in\C^{d\times d}$ be hermitian,
$A_0\succ0$ and \[p=\det(A_0+x_1A_1+\ldots+x_nA_n).\] Then
\[C(p)=\{a\in\R^n\mid A_0+a_1A_1+\ldots+a_nA_n\succeq0\}\]
and
\[C(p)\setminus Z(p)=\{a\in\R^n\mid A_0+a_1A_1+\ldots+a_nA_n\succ0\}.\]
\end{pro}

\begin{proof}
The second statement follows easily from the first. To prove the first,
let $a\in\R^n$ and set $B:=a_1A_1+\ldots+a_nA_n$. We have to show
\[(\forall\la\in[0,1):\det(A_0+\la B)\ne0)\iff A_0+B\succeq0.\]
Since $A_0$ is positive definite, there exists a (unique) pd matrix $\sqrt{A_0}$
whose square is $A_0$. Rewriting both the left and right hand side of our claim, it becomes
\[(\forall\la\in[0,1):\det(I_d+\la C)\ne0)\iff I_d+C\succeq0.\]
where $C:=\sqrt{A_0}^{-1}B\sqrt{A_0}^{-1}$.
Since $C$ is hermitian, we find a unitary matrix
$U\in\C^{n\times n}$ such that $U^*CU$ is a diagonal matrix with diagonal entries $d_1,\dots,d_n\in\R$.
Our claim can be rewritten
\[\left(\forall\la\in[0,1):\prod_{i=1}^d(1+\la d_i)\ne0\right)\iff\forall i\in\{1,\ldots,n\}:1+d_i\ge0\]
which is easily checked.
\end{proof}

\begin{rem}\label{planar}
For $n\in\{0,1\}$, it is trivial that each rigidly convex set in $\R^n$ is a spectrahedron.
For $n=2$ this follows from Helton-Vinnikov Corollary \ref{vcor} together with Proposition \ref{rcsdet}.
Whether this continues to holds for $n>2$ is unknown and is the topic of Section \ref{sec:glc} below.
\end{rem}

If $S\subseteq\R^n$ and $a\in\R^n$, we write $S+a:=\{b+a\mid b\in S\}$.
G\aa rding proved the following result in a more elementary way \cite{gar}. For convenience of
the reader, we include here a proof but allow ourselves the luxury to base it on the Helton-Vinnikov Corollary \ref{vcor} although this is an overkill.

\begin{thm}[G\aa rding]\label{rzshift}
Let $p\in\R[x]$ be a real zero polynomial and $a\in C(p)\setminus Z(p)$. Then the shifted polynomial $p(x+a)$ is a real zero polynomial as well
and $C(p)=C(p(x+a))+a$.
\end{thm}

\begin{proof}
We first show that $p(x+a)$ is a real zero polynomial. To this end, let $b\in\R^n$ and $\la\in\C$ such that $p(\la b+a)=0$.
We have to show $\la\in\R$.
We have $b\ne0$ since $a\notin Z(p)$.
If $a=\mu b$ for some $\mu\in\R$, then $p((\la+\mu)a)=0$ implies $\la+\mu\in\R$ and thus $\la\in\R$.
Hence we can now suppose that $a$ and $b$ are linearly independent. By an affine transformation, we can even suppose that
$a$ and $b$ are the first two unit vectors in $\R^n$. Without loss of generality, we can thus assume that the
number of variables is $n=2$. Also WLOG $p(0)=1$. By the Helton-Vinnikov Corollary \ref{vcor}, we can write
\[p=\det(I_d+x_1A_1+x_2A_2)\]
with hermitian $A_1,A_2\in\C^{d\times d}$ where $d:=\deg p$.
The hypothesis $a\in C(p)\setminus Z(p)$ now translates into $A:=I_d+A_1\succ0$ by Proposition \ref{rcsdet}.
From $\det(A+\la A_2)=0$ and Lemma \ref{genev}, we get $\la\in\R$.

To prove the second statement, let $b\in\R^n$. We show that
\[(*)\qquad b\in C(p)\iff b-a\in C(p(x+a)).\]
The case where $a$ and $b$ are linearly
dependent is an easy exercise. Suppose therefore that $a$ and $b$ are linearly independent. After an affine transformation, we can
even assume that $a$ and $b$ are the first two unit vectors. Hence we can reduce to the case where the number of variables $n$ equals $2$.
By the Helton-Vinnikov Corollary, we can choose hermitian matrices $A_1,A_2\in\C^{d\times d}$ such that
\[p=\det(I_d+x_1A_1+x_2A_2)\]
so that
\[C(p)=\{c\in\R^2\mid I_d+c_1A_1+c_2A_2\succeq0\}\]
by Proposition \ref{rcsdet}. Then 
\[p(x+a)=\det((I_d+A_1)+x_1A_1+x_2A_2)\]
and $I_d+A_1\succ A_1\succ0$ so that
\[C(p(x+a))=\{c\in\R^2\mid I_d+A_1+c_1A_1+x_2A_2\succeq0\}.\]
again by Proposition \ref{rcsdet}.
Our claim $(*)$ now translates into
\[I_d+A_2\succeq0\iff I_d+A_1-A_1+A_2\succeq0.\]
which holds trivially.
\end{proof}

\begin{thm}[G\aa rding]\label{rc-is-c}
If $p$ is a real zero polynomial, then $C(p)\setminus Z(p)$ and
$C(p)$ are convex.
\end{thm}

\begin{proof}
Call for the moment a subset $S\subseteq\R^n$ \emph{star-shaped} if for all $x\in S$, we have $\la x\in S$ for each $\la\in[0,1]$.
Clearly, a subset $S\subseteq\R^n$ is convex if and only if $S-a$ is star-shaped for each $a\in S$.
To show that $C(p)\setminus Z(p)$ is convex, we therefore fix $a\in C(p)\setminus Z(p)$ and show that
$(C(p)\setminus Z(p))-a$ is star-shaped. By Theorem \ref{rzshift},
$(C(p)\setminus Z(p))-a$ equals $C(q)\setminus Z(q)$ for some real zero polynomial $q\in\R[x]$ and therefore is obviously star-shaped by Definition \ref{dfrcs}.

Finally, to prove that $C(p)$ is convex, note that
\begin{multline*}
C(p)=\{a\in\R^n\mid\forall\la\in(0,1):\la a\in C(p)\setminus Z(p)\}\\
=\bigcap_{\la\in(0,1)}\{a\in\R^n\mid\la a\in C(p)\setminus Z(p)\}
\end{multline*}
is an intersection of convex sets and therefore convex.
\end{proof}


\section{The relaxation}

\subsection{The linear form associated to a polynomial}

All power series are formal. We refer to \cite[§3]{god}, \cite[§6.1]{rob} and \cite[§2]{rui} for an introduction to power series.

\begin{df}\label{dftrunc}
Suppose $a_\al\in\R$ for all $\al\in\N_0^n$ and consider the power series
\[p=\sum_{\al\in\N_0^n}a_\al x^\al\in\R[[x]].\] Then we call for $d\in\N_0$, the polynomial
\[\trunc_dp:=\sum_{\substack{\al\in\N_0^n\\|\al|\le d}}a_\al x^\al\in\R[x]\] the \emph{truncation of $p$ at degree $d$}.
\end{df}

\begin{df}\label{dflogexp}
Let $p\in\R[[x]]$ be a power series.
\begin{enumerate}[(a)]
\item
If $p$ has constant coefficient $0$, then the power series
\[\exp p:=\sum_{k=0}^\infty\frac{p^k}{k!}\in\R[[x]]\]
is well-defined because all monomials appearing in $p^k$ have degree at least $k$ so that
\[\trunc_d \exp p=\trunc_d\sum_{k=0}^d\frac{p^k}{k!}\]
for all $d\in\N_0$. We call it the \emph{exponential} of $q$.
\item
If $p$ has constant coefficient $1$, then the power series
\[\log p:=\sum_{k=1}^\infty\frac{(-1)^{k+1}}k(p-1)^k\in\R[[x]]\]
is well-defined because all monomials appearing in $(p-1)^k$ have degree at least $k$ so that
\[\trunc_d\log p=\trunc_d\sum_{k=1}^d\frac{(-1)^{k+1}}k(p-1)^k\]
for all $d\in\N_0$. We call it the \emph{logarithm} of $p$.
\end{enumerate}
\end{df}

The following can be found for example in \cite[Lemma 4.1]{god},
\cite[§5.4.2, Proposition 2]{rob} or \cite[§6.1.3]{rob}

\begin{pro}\label{logexp}
Consider the sets
\begin{align*}
A&:=\{p\in\R[[x]]\mid\trunc_0p=0\}\qquad\text{and}\\
B&:=\{p\in\R[[x]]\mid\trunc_0p=1\}.
\end{align*}
Then the following hold:
\begin{enumerate}[(a)]
\item $\exp\colon A\to B$ and $\log\colon B\to A$ are inverse to each other.
\item $\exp(p+q)=(\exp p)(\exp q)$ for all $p,q\in A$
\item $\log(pq)=(\log p)+(\log q)$ for all $p,q\in B$
\end{enumerate}
\end{pro}

\begin{proof}
(a) can be proven in two different ways: One way is to play it back to known facts about converging power series from calculus
\cite[§5.4.2, Proposition 2]{rob}. Note that the argument given in \cite[Lemma 4.1]{god} looks innocent but in reality needs
good knowledge of multivariate power series \cite[§I.1, §I.2]{rui}. The other way to prove it is by using formal composition and derivation of power series:
Deduce the result from the univariate case $n=1$ by
formally deriving the formal composition (in either way) of the univariate logarithmic and exponential power series \cite[§6.1.3]{rob}.

(b) is an easy calculation:
\begin{multline*}
\exp(p+q)=\sum_{k=0}^\infty\frac{(p+q)^k}{k!}=\sum_{k=0}^\infty\frac1{k!}\sum_{i=0}^k\binom kip^iq^{k-i}\\
=\sum_{k=0}^\infty\sum_{i=0}^k\frac{p^i}{i!}\frac{q^{k-i}}{(k-i)!}=(\exp p)(\exp q)
\end{multline*}

(c) follows easily from (a) and (b).
\end{proof}

\begin{df}\label{dfriesz}
Let $p\in\R[[x]]$ satisfy $p(0)\ne0$ and let $d\in\N_0$. We define the linear form $L_{p,d}$ on $\R[x]$ \emph{associated to $p$
with respect to the virtual degree $d$}
by specifying it on the monomial basis of $\R[x]$, namely by setting
\[L_{p,d}(1)=d\]
and by requiring the identity of formal power series
\[-\log\frac{p(-x)}{p(0)}=\sum_{\substack{\al\in\N_0^n\\\al\ne0}}\frac1{|\al|}\binom{|\al|}\al L_{p,d}(x^\al)x^\al\]
to hold. If $p\in\R[x]$, then we call $L_p:=L_{p,\deg p}$ the linear form \emph{associated to $p$}.
\end{df}

\begin{ex}\label{moments3}
Suppose $p\in\R[[x]]$ such that
\[\trunc_3p=1+\sum_{i\in\{1,\ldots,n\}}a_ix_i+\sum_{\substack{i,j\in\{1,\ldots,n\}\\i\le j}}
a_{ij}x_ix_j+\sum_{\substack{i,j,k\in\{1,\ldots,n\}\\i\le j\le k}}a_{ijk}x_ix_jx_k\]
where $a_i,a_{ij},a_{ijk}\in\R$. Then
\begin{multline*}
\trunc_3(-\log p(-x))=
\trunc_3\left(\sum_{\ell=1}^3\frac{(-1)^\ell}\ell(p(-x)-1)^\ell\right)\\
=
\sum_{i\in\{1,\ldots,n\}}a_ix_i-
\sum_{\substack{i,j\in\{1,\ldots,n\}\\i\le j}}a_{ij}x_ix_j+\sum_{\substack{i,j,k\in\{1,\ldots,n\}\\i\le j\le k}}a_{ijk}x_ix_jx_k\\
+\frac12\left(\sum_{i\in\{1,\ldots,n\}}a_ix_i\right)^2-\left(\sum_{i\in\{1,\ldots,n\}}a_ix_i\right)\left(\sum_{\substack{i,j\in\{1,\ldots,n\}\\i\le j}}a_{ij}x_ix_j\right)
+\frac13\left(\sum_{i\in\{1,\ldots,n\}}a_ix_i\right)^3\\
\end{multline*}
It follows that
\begin{align*}
L_{p,d}(x_i)&=a_i,\\
\frac12L_{p,d}(x_i^2)&=-a_{ii}+\frac12a_i^2,\\
\frac13L_{p,d}(x_i^3)&=a_{iii}-a_ia_{ii}+\frac13a_i^3
\end{align*}
for all $i\in\{1,\ldots,n\}$,
\begin{align*}
L_{p,d}(x_ix_j)&=-a_{ij}+a_ia_j,\\
L_{p,d}(x_i^2x_j)&=a_{iij}-a_ia_{ij}-a_ja_{ii}+a_i^2a_j
\end{align*}
for all $i,j\in\{1,\ldots,n\}$ with $i<j$, and 
\begin{align*}
2L_{p,d}(x_ix_jx_k)&=a_{ijk}-a_ia_{jk}-a_ja_{ik}-a_ka_{ij}+2a_ia_ja_k
\end{align*}
for all $i,j,k\in\{1,\ldots,n\}$ with $i<j<k$.
\end{ex}

\begin{pro}\label{logprop}
\begin{enumerate}[(a)]
\item If $p,q\in\R[[x]]$ satisfy $p(0)\ne0\ne q(0)$ and $d,e\in\N_0$, then \[L_{pq,d+e}=L_{p,d}+L_{q,e}.\]
\item If $p,q\in\R[x]$ satisfy $p(0)\ne0\ne q(0)$, then \[L_{pq}=L_p+L_q.\]
\end{enumerate}
\end{pro}

\begin{proof}
Part (b) follows from (a) by observing that $\deg(pq)=\deg p+\deg q$ for all $p,q\in\R[x]$.
To prove (a), we suppose WLOG $p(0)=1=q(0)$ and thus $(pq)(0)=1$. 
By Definition \ref{dfriesz},
we then have to show that the following identity of formal power series holds:
\[-\log((pq)(-x))=-\log(p(-x))-\log(q(-x)).\]
This follows from Proposition \ref{logexp}(c).
\end{proof}

\begin{lem}\label{thesame}
Let $L$ be a linear form on the vector subspace $V$ of $\R[x]$ generated by the monomials of degree $k$.
Then
\[\sum_{\substack{\al\in\N_0^n\\|\al|=k}}\frac1k\binom k\al L((Ux)^\al)(Ux)^\al\in V\]
is the same polynomial for all orthogonal matrices $U\in\R^{n\times n}$.
\end{lem}

\begin{proof}
Denote the $i$-th line of $U$ by $u_i$ (so that $u_i$ is a row vector) and the $j$-th entry of $u_i$ by $u_{ij}$ for all $i,j\in\{1,\ldots,n\}$. In the following, we often form the product of a row vector $u\in\R^n\subseteq\R[x]^n$ with the column vector $x\in\R[x]^n$ which is of course an element of $\R[x]$.
Then
\begin{multline*}
\sum_{|\al|=k}\binom k\al L((Ux)^\al)(Ux)^\al=\sum_{i_1,\ldots,i_k=1}^nL(u_{i_1}x\dotsm u_{i_k}x)u_{i_1}x\dotsm u_{i_k}x\\
=\sum_{i_1,\ldots,i_k=1}^n\sum_{j_1,\ldots,j_k=1}^nu_{i_1j_1}\dotsm u_{i_kj_k}L(x_{j_1}\dotsm x_{j_k})\sum_{\ell_1,\ldots,\ell_k=1}^n
u_{i_1\ell_1}\dotsm u_{i_k\ell_k}x_{\ell_1}\dotsm x_{\ell_k}\\
=\sum_{\ell_1,\ldots,\ell_k=1}^n\sum_{j_1,\ldots,j_k=1}^n\left(\sum_{i_1,\ldots,i_k=1}^nu_{i_1j_1}\dotsm u_{i_kj_k}
u_{i_1\ell_1}\dotsm u_{i_k\ell_k}\right)L(x_{j_1}\dotsm x_{j_k})x_{\ell_1}\dotsm x_{\ell_k}\\
=\sum_{\ell_1,\ldots,\ell_k=1}^n\sum_{j_1,\ldots,j_k=1}^n
\Bigg(\underbrace{\sum_{i_1=1}^nu_{i_1j_1}u_{i_1\ell_1}}_{\tiny=\begin{cases}1&\text{if }j_1=\ell_1\\0&\text{otherwise}\end{cases}}\Bigg)
\dotsm
\Bigg(\underbrace{\sum_{i_k=1}^nu_{i_kj_k}u_{i_k\ell_k}}_{\tiny=\begin{cases}1&\text{if }j_k=\ell_k\\0&\text{otherwise}\end{cases}}\Bigg)
L(x_{j_1}\dotsm x_{j_k})x_{\ell_1}\dotsm x_{\ell_k}\\
=\sum_{\ell_1,\ldots,\ell_k=1}^n
L(x_{\ell_1}\dotsm x_{\ell_k})x_{\ell_1}\dotsm x_{\ell_k}=\sum_{|\al|=k}\binom k\al L(x^\al)x^\al
\end{multline*}
for all orthogonal matrices $U\in\R^{n\times n}$. Multiplying with $\frac1k$ gives the result.
\end{proof}

\begin{pro}\label{rotate}
Let $U\in\R^{n\times n}$ be an orthogonal matrix.
\begin{enumerate}[(a)]
\item If $p\in\R[[x]]$ with $p(0)\ne0$ and $d\in\N_0$, then
\[L_{p(Ux),d}(q(Ux))=L_{p,d}(q).\]
\item If $p\in\R[x]$ with $p(0)\ne0$, then
\[L_{p(Ux)}(q(Ux))=L_p(q).\]
\end{enumerate}
\end{pro}

\begin{proof}
Part (b) follows easily from (a) since $\deg(p(Ux))=\deg p$ for all $p\in\R[x]$. To prove (a), fix
$p\in\R[[x]]$ with $p(0)\ne0$ and $d\in\N_0$. It suffices by linearity to prove
$L_{p(Ux),d}((Ux)^\al)=L_{p,d}(x^\al)$ for all $\al\in\N_0^n$ with $\al\ne0$.
WLOG $p(0)=1$. From Definition \ref{dflogexp}(b), one gets easily
\[\log(p(U(-x)))=(\log p)(-Ux).\]
This means by Definition \ref{dfriesz} that 
\[\sum_{\substack{\al\in\N_0^n\\\al\ne0}}\frac1{|\al|}\binom{|\al|}\al L_{p(Ux),d}(x^\al)x^\al=\sum_{\substack{\al\in\N_0^n\\\al\ne0}}\frac1{|\al|}\binom{|\al|}\al L_{p,d}(x^\al)(Ux)^\al.\]
We rewrite the left hand side by means of Lemma \ref{thesame} to obtain
\[\sum_{\substack{\al\in\N_0^n\\\al\ne0}}\frac1{|\al|}\binom{|\al|}\al L_{p(Ux),d}((Ux)^\al)(Ux)^\al=\sum_{\substack{\al\in\N_0^n\\\al\ne0}}\frac1{|\al|}\binom{|\al|}\al L_{p(x^\al),d}(Ux)^\al.\]
Substituting $U^Tx$ for $x$, we finally get
\[\sum_{\substack{\al\in\N_0^n\\\al\ne0}}\frac1{|\al|}\binom{|\al|}\al L_{p(Ux),d}((Ux)^\al)x^\al=\sum_{\substack{\al\in\N_0^n\\\al\ne0}}\frac1{|\al|}\binom{|\al|}\al L_{p,d}(x^\al)x^\al.\]
Comparing coefficients, we get the result.
\end{proof}

\begin{df}\label{defhomogeneous}
A polynomial is called \emph{homogeneous} if all its monomials are of the same degree. If $p\in\R[x]$ is a polynomial of
degree $d\in\N_0$, then the homogeneous polynomial
\[p^*:=x_0^d\;p\hspace{-0.3em}\left(\frac{x_1}{x_0},\ldots,\frac{x_n}{x_0}\right)\in\R[x_0,x]\]
is called its \emph{homogenization} (with respect to $x_0$). In addition, we set
\[p^*:=0\in\R[x_0,x].\]
\end{df}

The ``shifted homogenization'' of a polynomial used in Part (b) of the following lemma appears for real zero polynomials already in Brändén
\cite{bra1} and \cite{nt}.

\begin{lem}\label{extend1}
\begin{enumerate}[(a)]
\item For all $p\in\R[[x]]$ with $p(0)\ne0$, the power series
\[q:=(1+t)^d\,p\hspace{-0.3em}\left(\frac{x_1}{1+t},\ldots,\frac{x_n}{1+t}\right)\in\R[[t,x]]\]
satisfies $q(0)\ne0$ and
\[L_{q,d}(f)=L_{p,d}(f(1,x))\]
for all $d\in\N_0$ and $f\in\R[t,x]$.
\item For all $p\in\R[x]$ with $p(0)\ne0$, the polynomial
\[q:=p^*(1+t,x_1,\ldots,x_n)\in\R[t,x]\]
satisfies $q(0)\ne0$ and
\[L_{q}(f)=L_{p}(f(1,x))\]
for all $f\in\R[t,x]$.
\end{enumerate}
\end{lem}

\begin{proof}
Part (b) follows from (a) by setting $d:=\deg p$ and observing that $\deg p=\deg q$. To prove (a), we let $d\in\N_0$ and
suppose WLOG $p(0)=1$ so that $q(0)=1$. By Definition \ref{dfriesz}, it remains to show the identity 
\[-\log(q(-t,-x_1,\ldots,-x_n))=\sum_{\substack{k\in\N_0\\\al\in\N_0^n\\(k,\al)\ne0}}\frac1{k+|\al|}\binom{k+|\al|}{k\ \,\al_1\ \ldots\ \al_n}L_p(x^\al)t^kx^\al\]
of formal power series. By Proposition \ref{logexp}(c), it suffices to prove the identities
\[-\log((1-t)^d)=\sum_{k=1}^\infty\frac1kL_{p,d}(1)t^k\]
and
\[-\log\left(p\hspace{-0.3em}\left(\frac{-x_1}{1-t},\ldots,\frac{-x_n}{1-t}\right)\right)=
\sum_{\substack{k\in\N_0\\\al\in\N_0^n\\\al\ne0}}\frac1{k+|\al|}\binom{k+|\al|}{k\ \,\al_1\ \ldots\ \al_n}L_{p,d}(x^\al)t^kx^\al\]
where one should note that the argument of the logarithm in the second identity is a power series since
\[\frac1{1-t}=\sum_{i=0}^\infty t^i.\]
Again by Proposition \ref{logexp}(c), we get $-\log((1-t)^d)=d\log(1-t)$. Together with $L_{p,d}(1)=d$ and Definition \ref{dflogexp}(b),
this yields the first identity. To prove the second identity, we substitute \[\frac{x_i}{1-t}=x_i\sum_{i=0}^\infty t^i\] for $x_i$
in the defining identity of $L_{p,d}$ from Definition \ref{dfriesz} for each $i\in\{1,\ldots,n\}$ to see that its left hand side equals
\[\sum_{\substack{\al\in\N_0^n\\\al\ne0}}\frac1{|\al|}\binom{|\al|}\al L_{p,d}(x^\al)\frac{x^\al}{(1-t)^{|\al|}}.\]
So it remains to show that for $\al\in\N_0^n$ with $\al\ne0$, we have
\[\frac{\binom{|\al|}\al}{|\al|(1-t)^{|\al|}}=\sum_{k=0}^\infty\frac1{k+|\al|}\binom{k+|\al|}{k\ \,\al_1\ \ldots\ \al_n}t^k.\]
Fix $\al\in\N_0^n$ with $\al\ne0$. The multinomial coefficient on the right hand side divided by the one on the left
hand side yields equals the binomial coefficient $\binom{k+|\al|}k$ which becomes
$\binom{k+|\al|-1}{|\al|-1}$ when multiplied with $\frac{|\al|}{k+|\al|}$. So it remains to show that
\[\left(\sum_{i=0}^\infty t^i\right)^{|\al|}=\sum_{k=0}^\infty\binom{k+|\al|-1}{|\al|-1}t^k.\]
This is clear since the number of tuples of length $|\al|$ of nonnegative integers that sum up to $k$ is the binomial
coefficient on the right hand side. Indeed, choosing such a tuple amounts to partition $\{1,\ldots,k\}$ into $|\al|$ discrete intervals.
This in turn is equivalent to choosing $|\al|-1$ elements as separating landmarks in the set $\{1,\ldots,k+|\al|-1\}$.
\end{proof}

\begin{df}
Let $p\in\R[x]$ and $a\in\R^n$. Then we define the \emph{$a$-transform} $p[a]$ of $p$ by
\[p[a]:=p^*(1+a^Tx,x)\in\R[x].\]
In other words, $0[a]=0$ and if $p$ is a polynomial of degree $d\in\N_0$, then
\[p[a]=(1+a^Tx)^d\,p\hspace{-0.3em}\left(\frac{x_1}{1+a^Tx},\ldots,\frac{x_n}{1+a^Tx}\right).\]
\end{df}

\begin{rem}\label{degrem}
Let $p\in\R[x]$ and $a\in\R^n$.
\begin{enumerate}[(a)]
\item
Of course, we have always $\deg(p[a])\le\deg p$. Unfortunately, the degree of $p[a]$ might
sometimes be strictly smaller than the one of $p$, for example if $p=1+x_1\in\R[x_1]$ and $a=-1\in\R=\R^1$ where
$p[a]=(1-x_1)+x_1=1\in\R[x_1]$.
But the reader easily verifies that the degrees of $p$ and $p[a]$ coincide if and only if the homogeneous polynomial
$p^*(a^Tx,x)$ is not the zero polynomial.
\item It is an easy exercise to show that
\[p[a][b]=p[a+b]\]
in the case where $\deg p=\deg(p[a])$.
\end{enumerate}
\end{rem}

\begin{pro}\label{transformrz}
Suppose $p\in\R[x]$ is a real zero polynomial and $a\in\R^n$. Then $p[a]$ is again a real zero polynomial.
\end{pro}

\begin{proof}
Suppose $b\in\R^n$ and $\la\in\C$ such that $p[a](\la b)=0$. We have to show $\la\in\R$.
If $1+\la a^Tb=0$, then $a^Tb\ne0$ and thus $\la=-\frac1{a^Tb}\in\R$. Suppose therefore $1+\la a^Tb\ne0$.
Then $p\left(\frac\la{1+\la a^Tb}b\right)=0$ and thus $c:=\frac\la{1+\la a^Tb}\in\R$ since $p$ is a real zero polynomial.
If $c=0$, then $\la=0\in\R$ and we are done. Suppose therefore $c\ne0$. Then $\la=c+\la ca^Tb$ and hence
$\la(1-ca^Tb)=c\ne0$ which again implies $\la\in\R$.
\end{proof}

\begin{pro}\label{transform}
Let $a\in\R^n$, $p\in\R[x]$ with $p(0)\ne0$ and $d\in\N_0$. Then
\[L_{p[a],d}(f(x))=L_{p,d}(f(x+a))\]
for all $f\in\R[x]$.
\end{pro}

\begin{proof}
WLOG $p(0)=1$ and thus $p[a](0)=1$. For the duration of this proof, we denote by $\preceq$ the
the partial order on $\N_0^n$ which stand for the componentwise natural order, i.e.,
\[\al\preceq\be\,:\iff\forall i\in\{1,\ldots,n\}:\al_i\le\be_i\]
for $\al,\be\in\N_0^n$. From Lemma \ref{extend1}, we know that
\[-\log(p^*(1-t,-x_1,\ldots,-x_n))=\sum_{\substack{k\in\N_0\\\al\in\N_0^n\\(k,\al)\ne0}}\frac1{k+|\al|}\binom{k+|\al|}{k\ \,\al_1\ \ldots\ \al_n}L_{p,d}(x^\al)t^kx^\al.\]
Substituting $a^Tx$ for $t$ in this identity, yields
\begin{align*}
-\log(p[a](-x))
&=\sum_{\substack{k\in\N_0\\\al\in\N_0^n\\(k,\al)\ne0}}\frac1{k+|\al|}\binom{k+|\al|}{k\ \,\al_1\ \ldots\ \al_n}L_{p,d}(x^\al)(a^Tx)^kx^\al\\
&=\sum_{\substack{k\in\N_0\\\al\in\N_0^n\\(k,\al)\ne0}}\frac1{k+|\al|}\binom{k+|\al|}{k\ \,\al_1\ \ldots\ \al_n}L_{p,d}(x^\al)
\left(\sum_{\substack{\be\in\N_0^n\\|\be|=k}}\binom k\be a^\be x^\be\right)x^\al\\
&=\sum_{\substack{\al,\be\in\N_0^n\\(\al,\be)\ne0}}\frac1{|\al|+|\be|}\binom{|\al|+|\be|}{|\be|\ \al_1\ \ldots\ \al_n}L_{p,d}(x^\al)
\binom{|\be|}\be a^\be x^{\al+\be}\\
&=\sum_{\substack{\al,\be\in\N_0^n\\(\al,\be)\ne0}}\frac1{|\al|+|\be|}\binom{|\al|+|\be|}{\al\quad\be}L_{p,d}(x^\al)
a^\be x^{\al+\be}\\
&=\sum_{\substack{\ga\in\N_0^n\\\ga\ne0}}\sum_{\substack{\al\in\N_0^n\\\al\preceq\ga}}\frac1{|\ga|}\binom{|\ga|}{\al\quad\ga-\al} L_{p,d}(x^\al)
a^{\ga-\al}x^\ga\\
&=\sum_{\substack{\ga\in\N_0^n\\\ga\ne0}}\frac1{|\ga|}\binom{|\ga|}\ga
\left(\sum_{\substack{\al\in\N_0^n\\\al\preceq\ga}}\binom{\ga_1}{\al_1}\dotsm\binom{\ga_n}{\al_n}
L_{p,d}(x^\al)a^{\ga-\al}\right)x^\ga\\
&=\sum_{\substack{\ga\in\N_0^n\\\ga\ne0}}\frac1{|\ga|}\binom{|\ga|}\ga
L_{p,d}\left(\prod_{i=1}^n\sum_{\al_i=1}^{\ga_i}\binom{\ga_i}{\al_i}
{x_i}^{\al_i}{a_i}^{\ga_i-\al_i}\right)x^\ga\\
&=\sum_{\substack{\ga\in\N_0^n\\\ga\ne0}}\frac1{|\ga|}\binom{|\ga|}\ga
L_{p,d}((x+a)^\ga)x^\ga.
\end{align*}
This implies $L_{p[a],d}(x^\al)=L_{p,d}((x+a)^\al)$ for all $\al\in\N_0^n$ with $\al\ne0$.
By linearity, this shows the claim. 
\end{proof}

\subsection{Linear forms and traces}

\begin{lem}\label{convergence}
Let \[\sum_{k=1}^\infty a_kt^k\in\C[[t]]\qquad(a_1,a_2,\ldots\in\C)\] be a univariate power series with positive radius of convergence. Let
$p\in\C[x]$ be a polynomial with $p(0)=0$.
Then there is some $\ep>0$ such that for each $z\in\C$ with $|z|<\ep$, the series
\[\sum_{k=1}^\infty a_kp(z)^k\]
is absolutely convergent even after fully expanding $a_kp(z)^k$ into the obvious sum of $m^k$ many terms where $m$ is the number of monomials in p.
\end{lem}

\begin{proof}
We have to bound the finite partial sums of the absolute values of the individual terms from above (see for example \cite[Definition 8.2.4]{t1} or \cite[§I.1]{rui}).
Write $p=t_1+\ldots+t_m$ where each $t_i$ involves only one monomial of $p$. 
We have to find $\ep>0$ and $C\in\R$ such that for each $z\in\C$ with $|z|<\ep$ and for each
$\ell\in\N$, we have
\[\sum_{k=1}^\ell\sum_{i_1=1}^m\ldots\sum_{i_k=1}^m|a_kt_{i_1}(z)\dotsm t_{i_k}(z)|\le C.\]
WLOG $m>0$. Let $r>0$ denote the radius of convergence of the univariate power series.
Choose $\rh\in\R$ with $0<\rh<r$. As $p(0)=0$, each $t_i$ vanishes at the origin.
By continuity, we can choose $\ep>0$ such that \[|t_i(z)|\le\frac\rh m\] for all $i\in\{1,\ldots,m\}$ and $z\in\C$ with $|z|<\ep$.
Since a univariate power series converges absolutely inside the radius of convergence
\cite[Theorem 4.1.6(b)]{t2}, we can set $C:=\sum_{k=1}^\infty |a_k|\rh^k<\infty$.
For each $\ell\in\N$ and $z\in\C$ with $|z|<\ep$, we have
\begin{align*}
\sum_{k=1}^\ell\sum_{i_1=1}^m\ldots\sum_{i_k=1}^m|a_kt_{i_1}\dotsm t_{i_k}|&\le
\sum_{k=1}^\ell m^k|a_k|\left(\frac\rh m\right)^k=\sum_{k=1}^\ell |a_k|\rh^k\le C.
\end{align*}
\end{proof}

\begin{df}\label{hurwitz}
Let $A_1,\dots,A_n\in\C^{d\times d}$ and $\al\in\N_0^n$. The $\al$-\emph{Hurwitz product} of $A_1,\ldots,A_n$ is the matrix
that arises as follows: First, form all words in $n$ letters where the $i$-th letter appears exactly $\al_i$ times. Then turn
each word into a product of matrices by substituting $A_i$ for the $i$-th letter. Finally, sum up all matrices that arise in this way.
Formally, we can define it as
\[\hur_\al(A_1,\ldots,A_n):=\sum_{\substack{f\colon\{1,\ldots,|\al|\}\to\{1,\ldots,n\}\\\forall i\in\{1,\ldots,n\}:\#f^{-1}(i)=\al_i}}A_{f(1)}\dotsm A_{f(|\al|)}\in\C^{d\times d}.\]
In particular, $\hur_0(A_1,\ldots,A_n)=I_d$.
\end{df}

\begin{pro}\label{hur}
Suppose $d\in\N_0$ and $A_1,\ldots,A_n\in\C^{d\times d}$ are hermitian. Then
\[p:=\det(I_d+x_1A_1+\ldots+x_nA_n)\in\R[x]\]
and
\[L_{p,d}(x^\al)=\frac1{\binom{|\al|}\al}\tr(\hur_\al(A_1,\ldots,A_n))\]
for all $\al\in\N_0^n$.
\end{pro}

\begin{proof}
For $q\in\C[x]$, we denote by $q^*\in\C[x]$ the polynomial which arises from $q$ by applying the complex conjugation to the coefficients. We have
\begin{align*}
p^*&=(\det(I_d+x_1A_1+\ldots+x_nA_n))^*=(\det((I_d+x_1A_1+\ldots+x_nA_n)^T))^*\\
&=(\det(I_d+x_1A_1^T+\ldots+x_nA_n^T))^*=\det(I_d+x_1A_1^*+\ldots+x_nA_n^*)\\
&=\det(I_d+x_1A_1+\ldots+x_nA_n)=p
\end{align*}
and therefore $p\in\R[x]$. It is easy to see that Hurwitz products of hermitian matrices are again hermitian
and therefore have real diagonal entries and henceforth real trace.

It is clear that $L_{p,d}(1)=d=\tr(I_d)=\tr(\hur_0(A_1,\ldots,A_n))$.
By Definition \ref{dfriesz}, it remains to show that
\[(*)\qquad-\log(p(-x))=\sum_{\substack{\al\in\N_0^n\\\al\ne0}}\frac1{|\al|}
\tr(\hur_\al(A_1,\ldots,A_n))x^\al.\]
The real multivariate power series on both sides converge absolutely in a neighborhood of the origin in $\R^n$. For the left hand side this follows from
Lemma \ref{convergence} by recollecting terms belonging to the same monomial. For the right hand side, we argue as follows: The number of words of
length $k$ in $n$ letters is $n^k$.
If the entries of each $A_i$ are bounded in absolute value by $c>0$, then the entries of a product of the $A_i$
with $k$ many factors are bounded in absolute value by $d^{k-1}c^k$. Then the trace of such a product is bounded by $(dc)^k$. Hence we get
\[\sum_{\substack{\al\in\N_0^n\\|\al|=k}}\left|\frac1{|\al|}\hur_\al(A_1,\ldots,A_n)a^\al\right|\le(cdn)^k\|a\|_\infty^k\le\left(\frac12\right)^k
\]
for all $a\in\R^n$ with $\|a\|_\infty\le\frac1{2cdn}$.

By the identity theorem for multivariate real power series \cite[Proposition 2.9]{rui}, it suffices to show that
both series in $(*)$ converge absolutely to the same value in a neighborhood of the origin in $\C^n$. It is a subtle issue that uses Lemma \ref{convergence}
and rearrangement of absolutely convergent series (cf. Proposition \cite[Proposition 1.6]{rui}) to show that for all $a$ in a neighborhood of the origin in $\R^n$, the left hand side of $(*)$
evaluated at $a$ (i.e., $(-\log(p(-x)))(a)$ equals $-\log(p(-a))$ where the first $\log$ stands for the
operation on power series defined in Definition \ref{dflogexp} and the second one for the usual real logarithm.
On the other hand, the right hand side of $(*)$ evaluates at $a$ from a small neighborhood of the
origin to
\[\sum_{\substack{\al\in\N_0^n\\\al\ne0}}\frac1{|\al|}
\tr(\hur_\al(a_1A_1,\ldots,a_nA_n))=\sum_{k=1}^\infty\frac1k\tr((a_1A_1+\ldots+a_nA_n)^k).\]
It now suffices to fix $a\in\R^n$ such that the hermitian matrix
\[B:=a_1A_1+\ldots+a_nA_n\in\C^{d\times d}\] is of operator norm strictly less than $1$ (or equivalently
has all eigenvalues in the open real interval $(-1,1)$) and to show that
\[-\log(\det(I_n-B))=\sum_{k=1}^\infty\frac1k\tr(B^k).
\]
Since the operator norm is sub-multiplicative, the matrix
$C:=\sum_{k=1}^\infty\frac1k B^k\in\C^{d\times d}$ exists. Obviously, $C$ is hermitian and its eigenvalues, listed according to their algebraic multiplicity,  arise from the eigenvalues of $I_n-B$ by taking minus the
real logarithm. 
Since determinant and trace are the product and sum, respectively, of the eigenvalues counted with algebraic multiplicity, we thus get the result.
\end{proof}

The traces of Hurwitz products appearing in Proposition \ref{hur} are in general hard to deal with.
It is an easy but absolutely crucial observation that this is different for Hurwitz products with up to
three factors.

\begin{cor}\label{cubictraces}
Suppose $d\in\N_0$ and $A_1,\ldots,A_n\in\C^{d\times d}$ are hermitian. Set
\[p:=\det(I_d+x_1A_1+\ldots+x_nA_n)\in\R[x].\]
Then
\begin{align*}
L_{p,d}(1)&=\tr(I_d),\\
L_{p,d}(x_i)&=\tr(A_i),\\
L_{p,d}(x_ix_j)&=\tr(A_iA_j)=\tr(A_jA_i)\qquad\text{and}\\
L_{p,d}(x_ix_jx_k)&=\re\tr(A_iA_jA_k)=\re\tr(A_iA_kA_j)=\re\tr(A_jA_iA_k)\\
&=\re\tr(A_jA_kA_i)=\re\tr(A_kA_iA_j)=\re\tr(A_kA_jA_i)
\end{align*}
for all $i,j,k\in\{1,\ldots,n\}$.
\end{cor}

\begin{proof} The first three statements are trivial. For the last statement, note that
\[(\tr(ABC))^*=(\tr((ABC)^T))^*=(\tr(C^TB^TA^T))^*=\tr(C^*B^*A^*)=\tr(CBA)\]
for all hermitian $A,B,C\in\C^{d\times d}$.
\end{proof}

\subsection{Relaxing hyperbolic programs}

\begin{df}\label{dfpencil}
Let $p\in\R[[x]]$ be a power series with $p(0)\ne0$ and $d\in\N_0$. Consider the symmetric matrices
\[
A_0:=
\begin{pmatrix}
L_{p,d}(1)&L_{p,d}(x_1)&\ldots&L_{p,d}(x_n)\\
L_{p,d}(x_1)&L_{p,d}(x_1^2)&\ldots&L_{p,d}(x_1x_n)\\
\vdots&\vdots&&\vdots\\
L_{p,d}(x_n)&L_{p,d}(x_1x_n)&\ldots&L_{p,d}(x_n^2)
\end{pmatrix}\in\R^{(n+1)\times(n+1)}\]
and
\[
A_i:=
\begin{pmatrix}
L_{p,d}(x_i)&L_{p,d}(x_ix_1)&\ldots&L_{p,d}(x_ix_n)\\
L_{p,d}(x_ix_1)&L_{p,d}(x_ix_1^2)&\ldots&L_{p,d}(x_ix_1x_n)\\
\vdots&\vdots&&\vdots\\
L_{p,d}(x_ix_n)&L_{p,d}(x_ix_1x_n)&\ldots&L_{p,d}(x_ix_n^2)
\end{pmatrix}\in\R^{(n+1)\times(n+1)}
\]
for $i\in\{1,\ldots,n\}$.
\begin{enumerate}[(a)]
\item We call the linear matrix polynomial
\[M_{p,d}:=A_0+x_1A_1+\ldots+x_nA_n\in\R[x]^{(n+1)\times(n+1)}\]
the \emph{pencil associated to $p$ with respect to the virtual degree $d$} and
\[S_d(p):=\{a\in\R^n\mid M_{p,d}(a)\succeq0\}\]
the \emph{spectrahedron associated to $p$ with respect to the virtual degree $d$}.
\item In the case where $p$ is a polynomial, we call
\[M_p:=M_{p,\deg p}\] the \emph{pencil associated to $p$} and \[S(p):=S_{\deg p}(p)\]
the \emph{spectrahedron associated to $p$}.
\item
We call the linear matrix polynomial
\[M_{p,\infty}\in\R[x]^{n\times n}\] that arises from $M_{p,d}$ (for no matter what $d\in\N_0$)
by deleting the first row and column the \emph{pencil associated to $p$ with respect to infinite virtual degree}
and \[S_\infty(p):=\{a\in\R^n\mid M_{p,\infty}(a)\succeq0\}\]
the \emph{spectrahedron associated to $p$ with respect to infinite virtual degree}.
\end{enumerate}
\end{df}

\begin{rem}\label{whenvirtualdegreerises}
Let $p\in\R[[x]]$ be a power series with $p(0)\ne0$. Then \[S_0(p)\subseteq S_1(p)\subseteq S_2(p)\subseteq S_3(p)\subseteq S_4(p)\subseteq\ldots\subseteq S_\infty(p).\]
\end{rem}

\begin{rem}\label{dependsonlyoncubicpart}
Let $p\in\R[x]$ be a polynomial with $p(0)\ne0$.
Note that $M_p$ and therefore $S(p)$ depend only on the cubic part $\trunc_3p$ of $p$. Indeed, if
one assumes moreover that $p(0)=1$ then this is a  polynomial dependance on the corresponding
coefficients of $p$, more exactly a cubic one
which could be written down explicitly by the expressions of Example \ref{moments3} for the values
of $L_p$ on the monomials of degree at most $3$.
\end{rem}

\begin{lem}\label{pencileval}
Let $p\in\R[x]$ be a power series with $p(0)\ne0$, $d\in\N_0$, $a\in\R^n$ and \[v=
\begin{pmatrix}v_0&v_1&\ldots&v_n\end{pmatrix}^T\in\R^{n+1}.\]
Then $v^TM_{p,d}(a)v=L_{p,d}((v_0+v_1x_1+\ldots+v_nx_n)^2(1+a_1x_1+\ldots+a_nx_n))$.
\end{lem}

\begin{proof}
For the moment denote $x_0:=1$ an $a_0:=1$. Then
\begin{multline*}
v^TM_{p,d}(a)v=\sum_{i=0}^n\sum_{j=0}^nv_iv_j\sum_{k=0}^na_kL_{p,d}(x_ix_jx_k)\\
=L_{p,d}\left(\left(\sum_{i=0}^nv_ix_i\right)\left(\sum_{j=0}^nv_jx_j\right)\left(\sum_{k=0}^na_kx_k\right)\right).
\end{multline*}
\end{proof}

\begin{lem}\label{pencilrotate}
Suppose $U\in\R^{n\times n}$ is an orthogonal matrix and consider the orthogonal matrix
\[\widetilde U:=\begin{pmatrix}1&0\\0&U\end{pmatrix}\in\R^{(n+1)\times(n+1)}\]
\begin{enumerate}[(a)]
\item If $p\in\R[[x]]$ is a power series with $p(0)\ne0$ and $d\in\N_0$, then
\[M_{p(Ux),d}=\widetilde U^TM_{p,d}(Ux)\widetilde U.\]
\item If $p\in\R[x]$ is a polynomial with $p(0)\ne0$, then
\[M_{p(Ux)}=\widetilde U^TM_p(Ux)\widetilde U.\]
\item If $p\in\R[[x]]$ is a power series with $p(0)\ne0$, then
\[M_{p(Ux),\infty}=U^TM_{p,\infty}(Ux)U.\]
\end{enumerate}
\end{lem}

\begin{proof}
Part (c) is immediate from (a).
Part (b) follows from (a) by observing that $\deg(p(Ux))=\deg p$ for all polynomials $p\in\R[x]$.
To prove (a), we let $p\in\R[[x]]$ be a power series with $p(0)\ne0$ and $d\in\N_0$.
We can rewrite the claim as
\[\widetilde UM_{p(Ux),d}\widetilde U^T=M_{p,d}(Ux)\]
which in turn is equivalent to
\[\widetilde UM_{p(Ux),d}(U^Tx)\widetilde U^T=M_{p,d}\]
by the automorphisms of the power series ring $\R[x]$ given by $x\mapsto Ux$ and $x\mapsto U^Tx$.
For each $v\in\R^n$, we denote by $\widetilde v\in\R^{n+1}$ the vector that arises from $v$ by prepending $1$.
By continuity, homogeneity and the identity theorem for multivariate polynomials,
it suffices to show that
\[\widetilde v^T\widetilde UM_{p(Ux),d}(U^Ta)\widetilde U^T\widetilde v=\widetilde v^TM_{p,d}(a)\widetilde v\] for all $a,v\in\R^n$.
By Lemma \ref{pencileval}, this is equivalent to
\[L_{p(Ux),d}((1+(U^Tv)^Tx)^2(1+(U^Ta)^Tx))=L_{p,d}((1+v^Tx)^2(1+a^Tx))\]
for all $a,v\in\R^n$ which follows easily from Proposition \ref{rotate}(a) after rewriting the left hand side as
$L_{p(Ux),d}((1+v^TUx)^2(1+a^TUx))$.
\end{proof}

\begin{pro}\label{rotatespectrahedron}
Let $p\in\R[x]$ with $p(0)\ne0$ and $U\in\R^{n\times n}$ an orthogonal matrix. Then
\begin{align*}
C(p(Ux))&=\{U^Ta\mid a\in C(p)\},\\
S(p(Ux))&=\{U^Ta\mid a\in S(p)\}\qquad\text{and}\\
S_d(p(Ux))&=\{U^Ta\mid a\in S_d(p)\}
\end{align*}
for all $d\in\N_0\cup\{\infty\}$.
\end{pro}

\begin{proof}
We have
\begin{align*}
C(p(Ux))&=\{a\in\R^n\mid\forall\la\in[0,1):p(U(\la a))\ne0\}\\
&=\{a\in\R^n\mid\forall\la\in[0,1):p(\la Ua)\ne0\}\\
&=\{U^Ta\in\R^n\mid\forall\la\in[0,1):p(\la a)\ne0\}=\{U^Tx\mid x\in C(p)\}
\end{align*}
and using Lemma \ref{pencilrotate}(b),
\begin{align*}
S(p(Ux))&=\{a\in\R^n\mid M_{p(Ux)}(a)\succeq0\}\\
&=\{a\in\R^n\mid M_p(Ua)\succeq0\}\\
&=\{U^Ta\in\R^n\mid M_p(a)\succeq0\}\\
&=\{U^Ta\in\R^n\mid a\in S(p)\}.
\end{align*}
The last statement follows in a similar way from Lemma \ref{pencilrotate}(a).
\end{proof}

\begin{lem}\label{penciltransform}
Let $p\in\R[x]$ be a polynomial with $p(0)\ne0$ and set $d:=\deg p$. Then
\[P:=\begin{pmatrix}1&a^T\\0&I_n\end{pmatrix}\in\R^{(n+1)\times(n+1)}\]
is invertible and
\[M_{p[a],d}=P^T(M_p+a^TxM_p(0))P.\]
\end{lem}

\begin{proof}
For each $v\in\R^n$, we denote by $\widetilde v\in\R^{n+1}$ the vector that arises from $v$ by prepending $1$.
By continuity, homogeneity and the identity theorem for multivariate polynomials,
it suffices to show that
\[\widetilde v^TM_{p[a],d}(b)\widetilde v=\widetilde v^TP^T(M_p(b)+a^TbM_p(0))P\widetilde v\]
for all $b,v\in\R^n$. Fixing $b,v\in\R^n$ and setting $w:=P\widetilde v=\begin{pmatrix}1+a^Tv\\v\end{pmatrix}\in\R^{n+1}$, this amounts to show
\[\widetilde v^TM_{p[a],d}(b)\widetilde v=w^TM_p(b)w+(a^Tb)w^TM_p(0)w.\]
 Applying Lemma \ref{pencileval}, this is equivalent to
\begin{multline*}
L_{p[a],d}((1+v^Tx)^2(1+b^Tx))=\\
L_p(((1+a^Tv)+v^Tx)^2(1+b^Tx))+a^TbL_p(((1+a^Tv)+v^Tx)^2)
\end{multline*}
for all $b,v\in\R^n$ which follows easily from Proposition \ref{transform} after rewriting the left hand side as
$L_p((1+v^T(x+a))^2(1+b^T(x+a)))$.
\end{proof}

\begin{lem}\label{restriction}
Suppose $m,n\in\N_0$ with $m\le n$ and $q\in\R[x]$ with $q(0)\ne0$. Set
\[r:=q(x_1,\ldots,x_m,0,\ldots,0)\in\R[x_1,\ldots,x_m].\]
Then
\begin{enumerate}[(a)]
\item $L_{q,d}(p)=L_{r,d}(p)$ for all $p\in\R[x_1,\ldots,x_m]$ and $d\in\N_0$
\item $\{a\in\R^m\mid(a,0,\ldots,0)\in C(q)\}=C(r)$
\item $\{a\in\R^m\mid(a,0,\ldots,0)\in S_d(q)\}\subseteq S_d(r)$ for all $d\in\N_0\cup\{\infty\}$
\end{enumerate}
\end{lem}

\begin{proof} (a) By linearity, it suffices to consider the case where $p$ is a monomial. If $p=1$, then
$L_{q,d}(p)=d=L_{r,d}(p)$. It remains to show that $L_{q,d}(x^\al)=L_{r,d}(x^\al)$ for all $a\in\N_0^m$ with
$\al\ne0$. But this follows from Definitions \ref{dfriesz} and \ref{dflogexp}(b) since
the power series $\log r$ arises from the power series $\log q$ by substituting the variables $x_{m+1},\ldots,x_n$ with $0$.

(b) is clear.

(c) follows from (a) together with Lemma \ref{pencileval}.
\end{proof}

\begin{pro}\label{traceisreal}
Fix $d\in\N_0$. Then
\[(A,B)\mapsto\tr(AB)\]
is a scalar product on the real vector space of hermitian matrices in $\C^{d\times d}$.
In particular, $\tr(AB)\in\R$ for all hermitian $A,B\in\C^{d\times d}$.
\end{pro}

\begin{proof}
Identifying each matrix of size $d$ with a ``long'' vector of size $d^2$ by reading its entries in the usual way, the scalar product is induced by the usual complex
scalar product on $\C^d$. Since all diagonal entries of a hermitian matrix are real and all other entries have the opposite imaginary part of its mirror entry, the claim
easily follows.
\end{proof}

\begin{lem}\label{pencildeteval}
Suppose $d\in\N_0$ and $A_1,\ldots,A_n\in\C^{d\times d}$ are hermitian. Set
\[p:=\det(I_d+x_1A_1+\ldots+x_nA_n)\in\R[x].\]
For all $a\in\R^n$ and
\[v=\begin{pmatrix}v_0&v_1&\ldots&v_n\end{pmatrix}^T\in\R^{n+1},\]
we then have
\[
v^TM_{p,d}(a)v=\tr((v_0I_d+v_1A_1+\ldots+v_nA_n)^2
(I_d+a_1A_1+\ldots+a_nA_n)).
\]
\end{lem}

\begin{proof}
Corollary \ref{cubictraces}, Lemma \ref{pencileval} and Proposition \ref{traceisreal}.
\end{proof}

Since $M^2$ is hermitian for each hermitian $M\in\C^{d\times d}$, Proposition \ref{traceisreal} shows that the traces occurring in the next definition are real. 
Moreover, since $M^2$ is even psd for each hermitian $M\in\C^{d\times d}$ and the trace of a product of two psd matrices nonnegative, we see that the two occurrences of ``${\implies}$'' could equivalently be replaced by ``${\iff}$'' in the next definition.

\begin{df}\label{couple}
We call $U$ \emph{perfect} if it is a subset of
$\{A\in\C^{d\times d}\mid A\text{ hermitian}\}$ that satisfies
\[\forall A\in U:((\forall M\in U:\tr(M^2A)\ge0)\implies\text{$A\succeq0$}).\]
We call $(U,V)$ an \emph{admissible couple} if $U\subseteq V\subseteq\{A\in\C^{d\times d}\mid A\text{ hermitian}\}$ and
\[\forall A\in U:((\forall M\in V:\tr(M^2A)\ge0)\implies\text{$A\succeq0$}).\]
\end{df}

\begin{rem}\label{perfetto}
\begin{enumerate}[(a)]
\item Let $U\subseteq\C^{d\times d}$ be perfect and $k\in\N_0$. Then
\[\left\{\begin{pmatrix}A&0&\hdots&0\\0&\ddots&\ddots&\vdots\\
\vdots&\ddots&\ddots&0\\
0&\hdots&0&A\end{pmatrix}\in\C^{(kd)\times(kd)}~\middle|~A\in U\right\}\]
is again perfect.
\item Let $U\subseteq\C^{d\times d}$ and $V\subseteq\C^{e\times e}$ be perfect and suppose $0\in U$ and $0\in V$.
Then
\[\left\{\begin{pmatrix}A&0\\0&B\end{pmatrix}\in\C^{(d+e)\times(d+e)}~\middle|~A\in U,B\in V\right\}\]
is again perfect.
\item Let $U\subseteq\C^{d\times d}$ be perfect and $Q\in\C^{d\times d}$ be a unitary matrix (e.g., a permutation matrix). Then
\[\{Q^*AQ\mid A\in U\}\]
is again perfect.
\end{enumerate}
\end{rem}

\begin{rem}\label{perfectex}
The following is an easy exercise that we leave to the reader:
\begin{enumerate}[\normalfont(a)]
\item The following sets are perfect:
\begin{itemize}
\item $\{\la I_d\mid\la\in\R\}$
\item $\{A\in\R^{d\times d}\mid A\text{ diagonal}\}$
\item $\{A\in\R^{d\times d}\mid A\text{ symmetric}\}$
\item $\{A\in\C^{d\times d}\mid A\text{ hermitian}\}$
\end{itemize}
\item If $V$ is a perfect set and $U$ is contained in it, then $(U,V)$ is an admissible couple.
\end{enumerate}
\end{rem}

\begin{ex}
Let $A\in\C^{d\times d}$ be hermitian. We claim that the real span of \[I_d,A,A^2,\ldots,A^{d-1}\] is perfect. Indeed,
by conjugating $A$ with a suitable unitary matrix, one easily reduces to the case where $A$ is a diagonal matrix
with diagonal entries $a=(a_1,\ldots,a_d)\in\R^n$. By conjugating it once more with a permutation matrix, we can moreover
suppose that the first $n$ entries $a_1,\ldots,a_n$ of $a$ are pairwise distinct and all other entries $a_{n+1},\ldots,a_d$ are
repetitions of entries of $a_1,\ldots,a_n$.
Consider now the Vandermonde matrix
\[H:=\begin{pmatrix}1&a_1&a_1^2&\hdots&a_1^{d-1}\\
1&a_2&a_2^2&\hdots&a_2^{d-1}\\
\vdots&\vdots&\vdots&&\vdots\\
1&a_d&a_d^2&\hdots&a_d^{d-1}
\end{pmatrix}\in\R^{d\times d}\]
whose columns are the diagonals of the diagonal matrices $I_d,A,\ldots,A^{d-1}$.
The top left $d\times d$ submatrix is invertible since it is again a Vandermonde matrix with pairwise different rows.
Hence the projection of the column space of $H$ on the first $d$ components is $\R^d$. The rows $n+1$ to $d$ of $H$
are repetitions of the first $n$ rows of $H$. The entries of each element in the column space follow the same pattern.
Hence it suffices to apply Remark \ref{perfetto}(a) $n$-times to the perfect set $\R=\R^{1\times1}$ (each time with a possibly different
appropriate number of repetitions $k$), then use several times Remark \ref{perfetto}(b) and finally apply Remark \ref{perfetto}(c) with a suitable
permutation matrix.
\end{ex}

\begin{pro}\label{rcsspecdet}
Suppose $d\in\N_0$, $A_1,\ldots,A_n\in\C^{d\times d}$,
\begin{align*}
U&:=\{v_0I_d+v_1A_1+\ldots+v_nA_n\mid v_0,v_1,\ldots,v_n\in\R\},\\
U_\infty&:=\{v_1A_1+\ldots+v_nA_n\mid v_1,\ldots,v_n\in\R\}
\end{align*}
and $(U,V)$ is an admissible couple (in particular, each $A_i$ is hermitian). Set
\[p:=\det(I_d+x_1A_1+\ldots+x_nA_n)\in\R[x].\]
\begin{enumerate}[(a)]
\item We have
\begin{align*}
C(p)&=\{a\in\R^n\mid\forall M\in V:\tr(M^2(I_d+a_1A_1+\ldots+a_nA_n))\ge0\},\\
S_d(p)&=\{a\in\R^n\mid\forall M\in U:\tr(M^2(I_d+a_1A_1+\ldots+a_nA_n))\ge0\}\qquad\text{and}\\
S_\infty(p)&=\{a\in\R^n\mid\forall M\in U_\infty:\tr(M^2(I_d+a_1A_1+\ldots+a_nA_n))\ge0\}.
\end{align*}
\item $C(p)\subseteq S_d(p)\subseteq S_\infty(p)$
\item If $U$ is perfect, then $C(p)=S_d(p)$.
\item If $U_\infty$ is perfect, then $C(p)=S_d(p)=S_\infty(p)$.
\end{enumerate}
\end{pro}

\begin{proof}
The first statement in (a) follows directly from Proposition \ref{rcsdet} together with Definition \ref{couple}. The remaining statements of (a) follow
easily from Lemma \ref{pencildeteval} and Definition \ref{dfpencil}. Statement (b) is a direct consequence of (a) since
$U_\infty\subseteq U\subseteq V$. Part (c) and (d) now follow directly from Definition \ref{couple}.
\end{proof}

\begin{ex}
Let $A_1,\ldots,A_n\in\R^{d\times d}$ by symmetric and consider
\[p:=\det(I_d+x_1A_1+\ldots+x_nA_n)\in\R[x]\]
which has degree at most $d$.
\begin{enumerate}[(a)]
\item Suppose the $A_i$ together with $I_d$ generate the vector space of all real symmetric matrices of size $d$, then 
$C(p)=S_d(p)$.
\item Suppose the $A_i$ themselves generate the vector space of all real symmetric matrices of size $d$, then
even $C(p)=S_\infty(p)$.
\end{enumerate}
\end{ex}

For real zero polynomials having a determinantal representation as in Proposition \ref{rcsspecdet}
whose size equals
their degree, the following is an immediate consequence of Proposition \ref{rcsspecdet}.
However, in the general case we have to argue in a much more subtle way. Actually, this is the first place
in this article where we wouldn't know how to avoid the Helton-Vinnikov theorem
(in form of Corollary \ref{vcor}).

\begin{thm}\label{relaxation}
Let $p$ be a real zero polynomial. Then $C(p)\subseteq S(p)$.
\end{thm}

\begin{proof}
WLOG $p(0)=1$. For $n\le2$, the claim follows from Proposition \ref{rcsspecdet}(b)
where we use the Helton-Vinnikov Corollary \ref{vcor} for $n=2$.
We now suppose $n>2$ and reduce it to the already proven case $n=2$. 
Let $a\in C(p)$. We have to show $M_p(a)\succeq0$. By continuity and homogeneity, it suffices to show
\[\begin{pmatrix}1&v^T\end{pmatrix}M_p(a)\begin{pmatrix}1\\v\end{pmatrix}\ge0\]
for all  $v\in\R^n$. By Lemma \ref{pencileval}, this is equivalent to
\[L_p((1+v^Tx)^2(1+a^Tx))\ge0.\]
To prove this, choose an orthogonal matrix
$U\in\R^{n\times n}$ such that $w:=U^Tv$ and $b:=U^Ta$ lie both in $\R^2\times\{0\}\subseteq\R^n$. By Proposition \ref{rotate}(b), it suffices to show
\[L_q((1+w^Tx)^2(1+b^Tx))\ge0\] where $q:=p(Ux)\in\R[x]$. Here $q$ and henceforth
$r:=q(x_1,x_2,0,\ldots,0)\in\R[x_1,x_2]$ are of course also
real zero polynomials. By Lemma \ref{restriction}(a), it is enough to show
\[L_r((1+w^Tx)^2(1+b^Tx))\ge0.\]
Now observe that $a\in C(p)$ implies $b\in C(q)$ and thus $(b_1,b_2)\in C(r)$.
But $C(r)\subseteq S(r)$ by the already proven case $n=2$. Hence $M_r(b_1,b_2)\succeq0$ and we can conclude by
Lemma \ref{pencileval}.
\end{proof}

\subsection{Relaxing linear programs}

We now come back to the most basic example of hyperbolic polynomials, namely products of linear polynomials non-vanishing at the origin.
The rigidly convex
sets they define are exactly the polyhedra containing the origin in their interior. The complexity behavior of optimization of linear functions over polyhedra is
mainly governed by the number of linear inequalities they are defined by \cite{mg}.
For polyhedra with a huge number of facets, it is therefore reasonable to try to
find reasonable outer approximations defined by a small linear matrix inequality. This fits of course into the above more general framework.
However, in this special case, we get new interpretations, simplifications and extensions of the construction presented above:

First, we will be able to
interpret the matrix coefficients $A_i$ of the pencil from Definition \ref{dfpencil} as moment matrices and localization matrices \cite{lau}. This might remind
the reader of Lasserre's moment relaxations. But Lasserre's relaxation is a lift-and-project method where
``moment matrices'' are actually matrices filled with unknowns having the structure of moment matrices. Our method is not a lift-and-project method
and the matrices are actual moment matrices filled with real numbers.

Second, the proofs for the case of linear programming will simplify dramatically. In particular, we will not need any version of the Helton-Vinnikov Theorem
\ref{vinnikov}.

Third, we will not restrict to moments of degree three in this case but will go to arbitrarily high moments and thus present a hierarchy of relaxations
for which we can prove finite convergence at level $d-1$ if $d$ is the number of linear inequalities.

\begin{pro}\label{dirac}
Let $d\in\N_0$ and $a_1,\dots,a_d\in\R^n$ and $p:=\prod_{i=1}^d(1+a_i^Tx)$. Then
\[L_p(q)=\sum_{i=1}^dq(a_i)\]
for all $q\in\R[x]$, i.e., $L_p$ is integration with respect to the sum of the Dirac measures in the points $a_i$.
\end{pro}

\begin{proof}
In the case $d=0$, we have $p=1$ and thus $L_p=0$. The case $d\ge2$ reduces to the case $d=1$ by means of Proposition
\ref{logprop}(b).
Hence we suppose now that $d=1$ and write $a:=a_1$. By linearity, it suffices to treat the case where $q$ is a monomial. For the constant monomial
$q=1$, we have that $L_p(q)=\deg p=1=q(a)$. For the other monomials, we have to show
\[-\sum_{k=1}^\infty\frac{(a_1^Tx)^k}k=\sum_{k=1}^\infty\frac1k\sum_{\substack{\al\in\N_0 ^n\\|\al|=k}}\binom{|\al|}\al a_1^\al x^\al\]
in view of Definitions \ref{dfriesz} and \ref{dflogexp}(b). But this follows from the multinomial theorem.
\end{proof}

The following generalizes Definition \ref{dfpencil}(b).

\begin{df}\label{dfpencilhierarchy}
Let $p\in\R[x]$ with $p(0)\ne0$ and $d\in\N_0$. Set $s:=\binom{d+n}d=\binom{d+n}n=\#E$ where \[E=\{x^\al\mid \al\in\N_0^n,|\al|\le d\}\]
is the set of monomials in $n$ variables of degree at most $d$.
Fix an order on these monomials, i.e., write $E=\{m_1,\ldots,m_s\}$.
Consider the symmetric matrices
\[A_0:=(L_p(m_im_j))_{\substack{1\le i\le s\\1\le j\le s}}\qquad\text{and}\qquad A_k:=(L_p(x_km_im_j))_{\substack{1\le i\le s\\1\le j\le s}}\]
for $k\in\{1,\ldots,n\}$. Then we call the linear matrix polynomial
\[M_p^{(d)}:=A_0+x_1A_1+\ldots+x_nA_n\in\R[x]^{s\times s}\]
the \emph{$d$-th pencil associated to $p$} and
\[S^{(d)}(p):=\{a\in\R^n\mid M_p^{(d)}(a)\succeq0\}\]
the \emph{$d$-th spectrahedron associated to $p$}.
\end{df}

\begin{lem}\label{pencilevalhierarchy}
Let $p\in\R[x]$ be a polynomial with $p(0)\ne0$, $a\in\R^n$ and $d\in\N_0$. Let $m_1,\ldots,m_s$ be the pairwise distinct monomials
in $n$ variables of degree at most $d$ in the order that has been fixed in Definition \ref{dfpencilhierarchy}. Let $v\in\R^s$.
Then \[v^TM_p^{(d)}(a)v=L_p((v_1m_1+\ldots+v_sm_s)^2(1+a_1x_1+\ldots+a_nx_n)).\]
\end{lem}

\begin{lem}\label{interpol}
Let $a_1,\ldots,a_d\in\R^n$ and $i\in\{1,\ldots,d\}$. Then there exists a polynomial $q\in\R[x]\setminus\{0\}$ with $\deg q<d$ such that $q(a_i)\ne0$ and
$q(a_j)=0$ for all $j\in\{1,\ldots,d\}\setminus\{i\}$.
\end{lem}

\begin{proof}
The polynomial $\prod_{j\in\{1,\ldots,d\}\setminus\{i\}}(x^T(a_j-a_i))\ne0$ cannot vanish on the whole of $\R^n$. 
So we can choose $v\in\R^n$ with $v^Ta_j\ne v^Ta_i$ for all $j\in\{1,\ldots,d\}\setminus\{i\}$. Now set
$q:=\prod_{j\in\{1,\ldots,d\}\setminus\{i\}}(v^Ta_j-v^Tx)$
\end{proof}

\begin{thm}\label{relaxationhierarchy}
Let $p\in\R[x]$ be a product of linear polynomials with $p(0)\ne0$. Then the following hold:
\begin{enumerate}[(a)]
\item $C(p)\subseteq S^{(d)}(p)$ for each $d\in\N_0$\qquad(``relaxation'')
\item $S^{(0)}(p)\supseteq S^{(1)}(p)\supseteq S^{(2)}(p)\supseteq S^{(3)}(p)\supseteq\ldots$\qquad(``hierarchy'')
\item If $d:=\deg p\ge1$, then $C(p)=S^{(d-1)}(p)$\qquad(``finite convergence'')
\end{enumerate}
\end{thm}

\begin{proof}
WLOG $p(0)=1$. Write $p=\prod_{i=1}^d(1+a_i^Tx)$ with $a_1,\ldots,a_d\in\R^n$. Then $d=\deg p$.
For example by Proposition \ref{rcsdet}
(interpreting the product representation of $p$ as a diagonal determinantal representation), we have
\[C(p)=\{a\in\R^n\mid1+a_1^Tx\ge0,\ldots,1+a_d^Tx\ge0\}.\]

(a) Let $a\in C(p)$. We have to show $M_{p,e}(a)\succeq0$ for all $e\in\N_0$. By Lemma \ref{pencilevalhierarchy}, this is equivalent to
$L_p(q^2(1+a^Tx))\ge0$ for all $q\in\R[x]$. This means by Proposition \ref{dirac} that
\[\sum_{i=1}^dq(a_i)^2(1+a^Ta_i)\ge0\]
for all $q\in\R[x]$. But even more is true: For each $i\in\{1,\dots,d\}$,
$1+a^Ta_i=1+a_i^Ta\ge0$
since $a\in C(p)$ and therefore $q(a_i)^2(1+a^Ta_i)\ge0$.

(b) is clear from Definition \ref{dfpencilhierarchy}.

(c) One inclusion has been proven already in (a). For the other one, let $a\in\R^n\setminus C(p)$. We show that $a\notin S^{(d-1)}(p)$.
By Lemma \ref{pencilevalhierarchy} and Proposition \ref{dirac}, this means we have to show that there exists a polynomial $q\in\R[x]\setminus\{0\}$ with
$\deg q\le d-1$ such that \[\sum_{j=1}^dq(a_j)^2(1+a_j^Ta)<0.\]
Choose $i\in\{1,\ldots,d\}$ such that $1+a_i^Ta<0$. By Lemma \ref{interpol}, we can choose a polynomial $q\in\R[x]\setminus\{0\}$ with $\deg q<d$ such
that $q(a_i)\ne0$ and $q(a_j)=0$ for all $j\in\{1,\ldots,d\}\setminus\{i\}$. Then
\[\sum_{j=1}^dq(a_j)^2(1+a_j^Ta)=q(a_i)^2(1+a_i^Ta)<0.\]
\end{proof}


\section{Tightening the relaxation}

Let $p\in\R[x]$ be real zero polynomial. By Theorem \ref{relaxation}, $S(p)$ is an outer spectrahedral
approximation of $C(p)$, i.e., $C(p)\subseteq S(p)$. If $p$ has high degree, then we cannot expect
in general that this is a good approximation since $S(p)$ is defined by a very small linear matrix inequality.
In this section, we analyze qualitatively a very simple idea of how to improve the spectrahedral outer approximation. The price we will have to pay is of course that we will need more linear matrix
inequalities (of the same size however).
Roughly, the idea is to ``move the origin''. More precisely, choose a point
$a\in C(p)\setminus Z(p)$ that is different from the origin. By Theorem \ref{rzshift}, the polynomial
$p(x+a)$ is again a real zero polynomial and we have $C(p(x+a))+a=C(p)$. In general, we do however
not have that $S(p(x+a))+a=S(p)$. This seems to a lacking theoretical property at first sight but
turns out to be a fortunate fact that we can take advantage of. Namely, we have
$C(p(x+a))\subseteq S(p(x+a))$ by Theorem \ref{relaxation} and therefore
$C(p)=C(p(x+a))+a\subseteq S(p(x+a))+a$ so that $S(p(x+a))+a$ is another outer spectrahedral
relaxation of $C(p)$ that will in general be different from $S(p)$.
Hence the intersection \[S(p)\cap(S(p(x+a))+a)\] will in general be an improved outer approximation
of $C(p)$. It is defined by two linear matrix inequalities of size $n+1$ each
which could of course be combined into a single one of size $2n+2$.
Instead of choosing two points inside $C(p)\setminus Z(p)$, namely the origin and $a$, we could now more
generally choose
finitely many points $a_1,\ldots,a_k\in C(p)\setminus Z(p)$ (the origin must not necessarily be among
them) and consider the spectrahedron
\[\bigcap_{i=1}^k(S(p(x+a_i))+a_i)\supseteq C(p)\]
defined by a linear matrix inequality of size $k(n+1)$. In practice, it seems like $S(p(x+a))+a$ tightly
approximates $C(p)$ in a neighborhood of $a\in C(p)\setminus Z(p)$. If this is right, then one would
get a very tight outer approximation by choosing $a_1,\ldots,a_k\in C(p)\setminus Z(p)$ in such way
that each point in $C(p)\cap Z(p)$ is close to one of the $a_i$. This might of course have a very
large price, namely that the number of points $k$ might have to be very large.

Now we want to support the just presented view on how to make the approximation tighter. To this end,
we prove a rather theoretical result that says that if we take \emph{all} points of $C(p)\setminus Z(p)$
instead of just finitely many, then the corresponding intersection equals $C(p)$. This is Corollary
\ref{intersectioncor} below. It does of course not imply that $C(p)$ is a spectrahedron since we deal
now with an \emph{infinite} intersection. In fact, we prove a more precise theorem that provides
an evidence for our idea that it might be sufficient to choose
the points close to the boundary of $C(p)$ provided each boundary point is close to one of the $a_i$.
This is Theorem \ref{intersection} where we intersect not over all points of $C(p)\setminus Z(p)$ but
only over those lying outside of a fixed closed subset $D$ of $C(p)\setminus Z(p)$. A good imagination
is that $C(p)$ is a potato, $Z(p)$ ist skin and $D$ the peeled potato (where the removed part inevitably
is a bit more than the skin).

Our theoretical theorem actually works even for a certain polyhedral instead of spectrahedral
outer approximation that we will now introduce. This is a poor man's version of the spectrahedron introduced in Definition \ref{dfpencil}(a). When no intersection comes into play, then we use actually
just an affine half space or in exceptional cases the full space, namely the one
that is defined by the linear inequality that corresponds to the top left entry of the pencil that
defines the spectrahedron.

\begin{df}\label{halfspace}
Let $p\in\R[[x]]$ be a power series with $p(0)\ne0$ and $d\in\N_0$.
We call
\begin{align*}
P_d(p)&:=\{a\in\R^n\mid L_{p,d}(1)+L_{p,d}(x_1)b_1+\ldots+L_{p,d}(x_n)b_n\ge0\}
\end{align*}
the \emph{polyhedron associated to $p$ with respect to the virtual degree $d$}.
\end{df}

In Remark \ref{hom-pol-rel-rem} below, we will give an interesting interpretation of this half-space for which we do not yet have the necessary notions.

\begin{rem}\label{polyhedral-relaxation-remark}
\begin{enumerate}[(a)]
\item If $p\in\R[[x]]$ satisfies $\trunc_1p=1+a_1x_1+\ldots+a_nx_n$, then
\[P_d(p)=\{b\in\R^n\mid d+a_1b_1+\ldots+a_nb_n\ge0\}\]
by Example \ref{moments3} and Definition \ref{dfriesz}(a).
\item Let $p\in\R[[x]]$ with $p(0)\ne0$, then $P_d(p)$ is either an affine half space or the full space depending on whether $\trunc_1p$ is a constant polynomial or not.
\item If $p\in\R[x]$ is a real zero polynomial of degree at most $d$, then we have \[C(p)\subseteq S_d(p)\subseteq P_d(p)\] where the first
inclusion follows from Theorem \ref{relaxation} and the second is trivial.
\item Let $p\in\R[x]$ is a real zero polynomial of degree at most $d$. The inclusion \[C(p)\subseteq P_d(p)\]
is trivial contrary to the finer statement from (c). To show it, suppose WLOG $p(0)=1$, $\deg p=d$ and write $\trunc_1p=1+a_1x_1+\ldots+a_nx_n$
with $a_1,\ldots,a_n\in\R$. We have to show that $d+a_1b_1+\ldots+a_nb_n\ge0$ for all $b\in C(p)$. Fixing $b\in C(p)$,
write \[(*)\qquad p(tb)=\prod_{i=1}^e(1+d_it)\] for some $e\in\{0,\ldots,d\}$ and $d_1,\ldots,d_e\in\R^\times$. By Definition \ref{dfrcs}, we have that
$p(tb)$ has no roots in the interval $[0,1)$. In other words, we have $d_i\ge-1$ for all $i\in\{1,\ldots,e\}$. Extracting the coefficient of $t$ on both
sides of $(*)$, we get $a_1b_1+\ldots+a_nb_n=d_1+\ldots+d_e$ and therefore
\[d+a_1b_1+\ldots+a_nb_n=d+d_1+\ldots+d_e\ge d-e\ge0.\]
\end{enumerate}
\end{rem}

\begin{thm}\label{intersection}
Suppose $d\in\N_0$, $p\in\R[x]$ is a real zero polynomial of degree at most $d$ and $D$ a closed subset of
$C(p)\setminus Z(p)$. Then
\[C(p)=\bigcap_{a\in C(p)\setminus(Z(p)\cup D)}S_d(p(x+a))+a=\bigcap_{a\in C(p)\setminus(Z(p)\cup D)}P_d(p(x+a))+a\]
where the empty intersection is interpreted as $\R^n$.
\end{thm}

\begin{proof}
Recall that $p(x+a)$ is a real zero polynomial for each $a\in C(p)\setminus Z(p)$ by Theorem \ref{rzshift}.

Both inclusions from left to right follow essentially from Theorem \ref{relaxation}:
For each $a\in C(p)\setminus Z(p)$, we have $C(p)=C(p(x+a))+a$ and
\[C(p(x+a))\subseteq S(p(x+a))\subseteq S_d(p(x+a))\] by Remark \ref{polyhedral-relaxation-remark}(c) so that
$C(p)\subseteq S(p(x+a))+a\subseteq P(p(x+a))+a$.

It remains to show that
\[C(p)\supseteq\bigcap_{a\in C(p)\setminus(Z(p)\cup D)}P_d(p(x+a))+a\]
We even show that for each half-line $H$ emanating from the origin,
\[H\cap C(p)\supseteq H\cap\bigcap_{a\in H\cap(C(p)\setminus(Z(p)\cup D))}(P_d(p(x+a))+a).\]
where the empty intersection stands of course again for $\R^n$.
Using Proposition \ref{rotate}(a), one easily reduces to the case where $H$ is the positive first axis
\[H=\R_{\ge0}\times\{0\}\subseteq\R^n.\]
By Lemma \ref{restriction}(a), we can reduce to the case
$n=1$ since $p(x_1,0,\ldots,0)\in\R[x_1]$ has again degree at most $d$. 
So suppose from now on that
we have just one variable $x=x_1$. Then $H=\R_{\ge0}$. WLOG $p(0)=1$. Write
\[p=\prod_{i=1}^d(1+a_ix)\]
with $a_1,\ldots,a_d\in\R$. The roots of $p$ are then the $-\frac1{a_1},\ldots,-\frac1{a_d}$.
If none of these roots is positive, then $C(p)\cap H=\R_{\ge0}$ and there is nothing to show.
For now on we can therefore
suppose that $d\ge1$ and \[r:=-\frac1{a_1}\] is the smallest positive root of $p$ so that
$H\cap C(p)=[0,r]$. In particular, $r>0$ and $a_1<0$. Since $D$ is closed and $r\notin D$, we can choose $\ep>0$ such that $r-\ep>0$ and
$H\cap D\subseteq[0,r-\ep]$. Then $(r-\ep,r)\subseteq H\cap(C(p)\setminus(Z(p)\cup D))$.
It therefore is enough to show that
\[(*)\qquad \bigcap_{a\in(r-\ep,r)}P_d(p(x+a))+a\subseteq(-\infty,r].\]
For $a\in\R\setminus Z(p)\supseteq(r-\ep,r)$, the polynomial
\[(**)\qquad p_a:=\frac{p(x+a)}{p(a)}=\prod_{i=1}^d\frac{1+a_i(x+a)}{1+a_ia}=\prod_{i=1}^d\left(1+\frac{a_i}{1+a_ia}x\right)\]
has constant coefficient $1$. Proposition \ref{dirac} hence implies that
$L_{p_a}$ is for each $a\in\R\setminus Z(p)\supseteq(r-\ep,r)$ integration with respect to the
sum of the Dirac measures in the points $\frac{a_1}{1+a_1a},\ldots,\frac{a_d}{1+a_da}$.
Now suppose that $b$ lies in the left hand side of $(*)$.
Then we have
\[d+(b-a)\sum_{i=1}^d\frac{a_i}{1+a_ia}=L_{p_a,d}(1)+(b-a)L_{p_a,d}(x)\ge0\]
for all $a\in(r-\ep,r)$. Now let $a$ converge to $r$ from below and consider what happens in $(*)$.
We have that $1+a_1a$ converges to $0$ from above.
Hence the first term in the sum in $(**)$ converges to $-\infty$. The $i$-th term of the sum shows the same
behavior in the case where $a_i=a_1$. All other terms of the sum converge to some real number.
Hence the whole sum converges to $-\infty$. The term $b-a$ converges to $b-r$ from above. If $b-r$ wer positive, then the left hand side of $(*)$ would converge to $-\infty$ while being nonnegative all the time.
Hence $b-r\le0$, i.e., $b$ lies in the right hand side of $(*)$.
\end{proof}

\begin{cor}\label{intersectioncor}
Let $p\in\R[x]$ be a real zero polynomial.
Then
\[C(p)=\bigcap_{a\in C(p)\setminus Z(p)}S(p(x+a))+a.\]
\end{cor}


\section{Exactness for quadratic real zero polynomials}

It is quite trivial that $C(p)=S(p)$ for linear real zero polynomials $p\in\R[x]$. One way of seeing this is
via Remark \ref{polyhedral-relaxation-remark} which shows that $C(p)=P_1(p)=S(p)$ for all
real zero polynomials $p\in\R[x]$ of degree $1$.
Since $S(p)$ depends only on the cubic part
$\trunc_3p$ of $p$ by Remark \ref{dependsonlyoncubicpart}, there seems to be no way
that $C(p)=S(p)$ in general if the degree of of the real zero polynomial $p$ is bigger than three.
But worse than that, one lacks in general exactness also for cubic real zero polynomials as the reader
will easily find. The next theorem shows that our relaxation at least is exact for quadratic real zero polynomials. 

\begin{thm}\label{quadratic}
Let $p$ be a quadratic real zero polynomial. Then $C(p)=S(p)$.
\end{thm}

\begin{proof}
From Theorem \ref{relaxation}, we know already $C(p)\subseteq S(p)$. To show the reverse $S(p)\subseteq C(p)$, we show that it holds
after intersecting with an arbitrary line through the origin. Instead of an arbitrary line, we can without loss of generality consider the
first axis in $\R^n$ by Proposition \ref{rotatespectrahedron}.
By Lemma \ref{restriction}(c), it is enough to show $S_{\deg p}(q)\subseteq C(q)$ for
$q:=p(x_1,0,\ldots,0)\in\R[x_1]$. If $\deg q=0$, then $C(q)=\R$ and there is nothing to show.
If $\deg q=1$, say $q=1+ax_1$ with a non-zero $a\in\R$, then the bottom right entry of $M_{q,\deg p}$ is $a^2+x_1a^3$ by Proposition
\ref{dirac} and hence $S_{\deg p}(q)\subseteq\{b\in\R\mid a^2+a^3b\ge0\}=\{b\in\R\mid 1+ab\ge0\}$. Finally, if
$\deg q=2$, then $S_{\deg p}(q)=S(q)=S^{(1)}(q)=C(q)$ by Theorem \ref{relaxationhierarchy}.
\end{proof}

The preceding proof makes use of the Helton-Vinnikov Theorem \ref{vinnikov} or Corollary \ref{vcor} indirectly through Theorem \ref{relaxation}. 
However, we need it only for quadratic polynomials. So we can use Example \ref{hv2} instead and therefore our proof is self-contained.

The next corollary was explicitly mentioned by Netzer and Thom \cite[Corollary 5.4]{nt} but most likely
was known before. Netzer and Thom use a hermitian linear matrix inequality of size
$2^{\lfloor\frac n2\rfloor}$ to describe a rigidly convex set given by a quadratic polynomial in $\R^n$ whereas
we need just a symmetric linear matrix inequality of size only $n+1$. While the proof of Netzer and Thom
is still very interesting for other reasons,
the result could actually also be proven by reducing it, via projective space, to the case of the unit ball
(cf. \cite[Page 958]{gar}). The well-known description of the unit ball in $\R^n$ by a
symmetric linear matrix inequality of size $n$ (see for example \cite[Page 591]{kum1}) shows that one can
for $n\ge2$ even get down to size $n$ instead of $n+1$. If $n=2^k+1$ for some $k\in\N_0$, then Kummer
shows that $n$ is the minimal size of a real symmetric linear matrix inequality describing the unit ball.
In general, he shows that $\frac n2$ is a lower bound \cite[Theorem 1]{kum1}.

\begin{cor}\label{quspec}
The rigidly convex set defined by a quadratic real zero polynomial is always a spectrahedron.
\end{cor}

If $p\in\R[x]$ is a quadratic real zero polynomial, then it follows easily from this corollary and from Proposition \ref{geoalg} below that there exists $q\in\R[x]$, $d\in\N_0$ and symmetric matrices $A_1,\ldots,A_n\in\R^{d\times d}$ such that
\[pq=\det(I_d+x_1A_1+\ldots+x_nA_n)\]
and $C(p)\subseteq C(q)$.
Our aim is now to prove this without using Proposition \ref{geoalg} by studying $\det(M_p)$ and
see which cofactor $q$ we get. We begin with a technical lemma that in particular gives a very nice
determinantal representation of the real zero polynomial $1-x_1^2-\ldots-x_n^2$.

\begin{lem}\label{circlelemma}
Let $d_1,\ldots,d_n\in\R$. Consider the matrix
\[M:=\begin{pmatrix}x_0&-d_1x_1&\ldots&-d_nx_n\\
-d_1x_1&-d_1x_0\\
\vdots&&\ddots\\
-d_nx_n&&&-d_nx_0
\end{pmatrix}\in\R[x_0,x]^{(n+1)\times(n+1)}.
\]
where the empty space is filled by zeros.
Then \[\det M=x_0^{n-1}(x_0^2+d_1x_1^2+\ldots+d_nx_n^2)\det(M(1,0)).\]
\end{lem}

\begin{proof}
Setting $r:=x_0^2+d_1x_1^2+\ldots+d_nx_n^2$, we show that
$x_0\det M=x_0^nr\det(M(1,0))$. Of course, $x_0\det M$ is the determinant
of the matrix $N$ that arises from $M$ by multiplying the first row with $x_0$. To compute the determinant of $N$,
we subtracting $x_i$ times its $(i+1)$-th row from its first row for each $i\in\{1,\ldots,n\}$. This results in an upper triangular matrix
with diagonal entries $r,-d_1x_0,\ldots,-d_nx_0$. The determinant of $N$ is the product of the
diagonal entries, i.e., $x_0^nr(-d_1)\dotsm(-d_n)=x_0^nr\det M(1,0)$.
\end{proof}

Using this lemma, we will now be able to compute the determinant of $M_p$ for quadratic $p\in\R[x]$.

\begin{thm}
For all quadratic $p\in\R[x]$ with $p(0)=1$,
\[\det(M_p)=(\det(M_p(0)))\left(\frac{1+\trunc_1p}2\right)^{n-1}p.\]
\end{thm}

\begin{proof}
For the constant polynomial $p=1\in\R[x]$, we have $L_p=0$ by Definition \ref{dfriesz} or Example
\ref{moments3} and hence $M_p=M_p(0)=0\in\R^{(n+1)\times(n+1)}$ by
Definition \ref{dfpencil}.
For a polynomial
$p\in\R[x]$ of degree one, i.e., $p=b^Tx+1$ for some $b\in\R^n$ with $b\ne0$, Proposition \ref{dirac} or Example \ref{moments3} yields
$M_p=(1+b^Tx)M_p(0)$ and \[M_p(0)=\begin{pmatrix}1&b^T\\b^T&bb^T\end{pmatrix}\in\R^{(n+1)\times(n+1)}\]
has rank one and therefore vanishing determinant.

Hence it suffices to show the claim for all polynomials $p\in\R[x]$ of degree two. But for fixed degree polynomials $p$,
both sides of the claimed equation depend continuously (in fact even polynomially) on the coefficients of $p$ (the top left entry of $M_p(0)$
is then the constant degree and therefore makes no trouble). Hence it is enough to show the equations for all polynomials of the form
\[p=x^TAx+b^Tx+1\]
for some symmetric matrix $A\in\R^{n\times n}$ and some vector $b\in\R^n$ such that $A\ne0$ and $4A\ne bb^T$. Fix some
$p$ of this form.
Set
\begin{multline*}
q:=p\left[-\frac b2\right]=x^TAx+\left(1-\frac{b^T}2x\right)b^Tx+\left(1-\frac{b^T}2x\right)^2\\
=x^TAx+\left(-\frac12+\frac14\right)(b^Tx)^2+1=x^T\left(A-\frac14bb^T\right)x+1.
\end{multline*}
Choose an orthogonal matrix $U$ such that
\[D:=U^T\left(A-\frac14bb^T\right)U\]
is diagonal and consider the polynomial \[r:=q(Ux)=x^TDx+1\]
of degree two. By Example \ref{moments3}, we have
\[M_r=\begin{pmatrix}2&-2x^TD\\-2Dx&-2D\end{pmatrix}=M(1,x)\]
where
\[M:=\begin{pmatrix}2x_0&-2x^TD\\-2Dx&-2x_0D\end{pmatrix}\in\R[x_0,x]^{(n+1)\times(n+1)}\]
By Lemma \ref{circlelemma}, we have
\[\det M=x_0^{n-1}r^*\det(M_r(0))\]
where $r^*\in\R[x_0,x]$ denotes the homogenization of $r$ defined in Definition \ref{defhomogeneous}.
By Remark \ref{degrem}, we have $p=q\hspace{-0.3em}\left[\frac b2\right]$ so that Lemma \ref{penciltransform} yields
\[\det(M_p)=\det\left(M_q+\frac{b^Tx}2M_q(0)\right).\]
Rewriting the right hand side in view of $q=r(U^Tx)$ by means of Lemma \ref{pencilrotate}, we get
\begin{align*}
\det(M_p)&=\det\left(M_r(U^Tx)+\frac{b^Tx}2M_r(0)\right)
=\det\left(M\left(1+\frac{b^Tx}2,U^Tx\right)\right)\\
&=\left(1+\frac{b^Tx}2\right)^{n-1}(\det(M_r(0))\,r^*\hspace{-0.3em}\left(1+\frac{b^Tx}2,U^Tx\right)\\
\end{align*}
Now we use $r^*(x_0,U^Tx)=q^*$ to see that
\[\det(M_p)=(\det(M_r(0))\left(\frac{1+\trunc_1p}2\right)^{n-1}q^*\left(1+\frac{b^Tx}2,x\right).\]
We result follows thus from
\[q^*\left(1+\frac{b^Tx}2,x\right)=q\left[\frac b2\right]=p.\]
\end{proof}

The following theorem follows easily from the previous one in the case where the matrix
$M_p(0)\in\R^{(n+1)\times(n+1)}$ is invertible. In the general case, its proof is clearly inspired
by the proof of the previous theorem.

\begin{thm}\label{lincofactor}
Let $p\in\R[x]$ be a quadratic real zero polynomial with $p(0)=1$. Then there exist symmetric matrices
$A_1,\ldots,A_n\in\R^{(n+1)\times(n+1)}$ such that 
\[p\left(\frac{1+\trunc_1p}2\right)^{n-1}=\det(I_{n+1}+x_1A_1+\ldots+x_nA_n).\]
\end{thm}

\begin{proof} Write
\[p=x^TAx+b^Tx+1\]
where $A\in\R^{n\times n}$ is symmetric and $b\in\R^n$. Set
\[q:=x^T\left(A-\frac14bb^T\right)x+1\in\R[x].\]
Choose an orthogonal matrix $U$ such that
\[D:=U^T\left(A-\frac14bb^T\right)U\]
is diagonal with diagonal entries
\[\underbrace{d_1,\ldots,d_m}_{\substack{\text{all different}\\\text{from zero}}},
\underbrace{0,\ldots,0}_{n-m\text{ zeros}}.\]
By Example \ref{quadraticrealzero}, we have that each $d_i$ is negative.

We first treat the special case where $m=0$. In this case $A=\frac14bb^T$ and thus $p=\left(1+\frac{b^Tx}2\right)^2$.
Then we choose $A_i$ as a diagonal matrix which arises from the identity matrix by replacing the first two diagonal entries
by $\frac12b_i$.

From now on, we consider the case $m>0$. Then we
consider the polynomial \[r:=q(Ux)=d_1x_1^2+\ldots+d_mx_m^2+1\in\R[x_1,\ldots,x_m]\]
of degree two and we set
\[M:=\begin{pmatrix}x_0&-d_1x_1&\ldots&-d_mx_m\\
-d_1x_1&-d_1x_0\\
\vdots&&\ddots\\
-d_mx_m&&&-d_mx_0\\
&&&&x_0\\
&&&&&\ddots\\
&&&&&&x_0\\
\end{pmatrix}\in\R[x_0,x]^{(n+1)\times(n+1)}
\]
where the empty space is filled by zeros. Using Lemma \ref{circlelemma}, we see easily that
\[\det M=x_0^{n+1}r^*\det(M(1,0)).\]
Evaluating $M$ in $(1+\frac{b^Tx}2,U^Tx)$, we get another matrix \[N\in\R[x]^{(n+1)\times(n+1)}\]
of the form $N=B_0+x_1B_1+\ldots+x_nB_n$ with symmetric $B_0,\ldots,B_n\in\R^{(n+1)\times(n+1)}$ such that
\begin{align*}
\det N&=\left(1+\frac{b^Tx}2\right)^{n+1}\left(q-1+\left(1+\frac{b^Tx}2\right)^2\right)\det(M(1,0))\\
&=\left(1+\frac{b^Tx}2\right)^{n+1}p\det(M(1,0))=c\left(\frac{1+\trunc_1p}2\right)^{n+1}p
\end{align*}
for some $c\in\R^\times$. Note that $B_0=N(0)=M(1,0)$ is a diagonal matrix. Recalling that each $d_i$ is negative, it is actually
a diagonal matrix with only positive diagonal entries. Hence, it is easy to find an invertible
diagonal matrix $D\in\R^{(n+1)\times(n+1)}$ such that $DB_0D=I_{n+1}$.
Setting $A_i:=DB_iD=D^2B_i$ for each $i\in\{0,\dots,n\}$, we have $A_0=I_{n+1}$ and
\[\det(A_0+x_1A_1+\ldots+x_nA_n)=c'\left(\frac{1+\trunc_1p}2\right)^kp\]
for some $c'\in\R$. Evaluating in $0$, we see that $c'=1$.
\end{proof}

As promised before Lemma \ref{circlelemma}, we have no found our
cofactor \[q:=\left(\frac{1+\trunc_1p}2\right)^{n-1}\] such that $pq$ has the
desired determinantal representation. Note that the required inclusion $C(p)\subseteq C(q)$ holds since
\[C(p)\subseteq P_2(p)=C\left(\frac{1+\trunc_1p}2\right)=C(q)\]
where the inclusion follows from Remark \ref{polyhedral-relaxation-remark}(c)(d).

Unfortunately, for real zero polynomials $p$ of higher degree it is not true that one can always choose
the cofactor $q$ from Proposition \ref{geoalg} below to be a power of linear polynomial. This is discussed in detail in \cite[Section 7]{ab} where even other
shapes for the cofactor are excluded. Indeed, from
\cite[Theorem 7.3]{ab} (setting there $n:=8$ and $k:=2$) one can deduce the existence of
a quartic real zero polynomial in 11 variables that it will never be of the shape
$\det(I_N+x_1A_1+\ldots+x_nA_n)$ with symmetric matrices $A\in\R^{N\times N}$ after being multiplied
with a power of a linear polynomial. We are not aware of any cubic example like that. But at least
the arguments in \cite[Example 12]{kum1} show that
\[p:=10-3x_1^2-6x_2-x_1^2x_2-3x_2^2+x_2^3-3x_3^2+x_2x_3^2\in\R[x_1,x_2,x_3]\]
is a real zero polynomial such that that there is no linear polynomial $\ell\in\R[x_1,x_2,x_3]$
and $k\in\N$ such that $p\ell^k$ has a determinantal representation of the form
\[p\ell^k=\det(I_{3+k}+x_1A_1+\ldots+x_nA_n)\]
with symmetric matrices $A_1,\ldots,A_n\in\R^{(3+k)\times(3+k)}$
(but they do not seem to exclude such a representation with matrices of size bigger than $3+k$). 
Since $p$ can easily be checked to be irreducible in $\R[x_1,x_2,x_3]$, \cite[Proposition 8]{kum1}
implies then together with Lemma \ref{sylvester} that there are a cubic
$q\in\R[x_1,x_2,x_3]$ and symmetric matrices $A_1,A_2,A_3\in\R^{6\times 6}$
such that $pq=\det(I_6+x_1A_1+x_2A_2+x_3A_3)$.


\section{Hyperbolicity cones and spectrahedral cones}

Recall Definition \ref{defhomogeneous} of homogeneous polynomials.

\subsection{Hyperbolic polynomials}

The following is the homogeneous analog of Definition \ref{dfrz}.

\begin{df}\label{dfhb}
Let $p\in\R[x]$ be a homogeneous. We call a vector $e\in\R^n$ with $e\ne0$
a \emph{hyperbolicity direction} of $p$ if for all
$a\in\R^n$ and $\la\in\C$,
\[p(a-\la e)=0\implies\la\in\R.\]
In this case, we call $p$ \emph{hyperbolic with respect to $e$} or \emph{hyperbolic in direction $e$}.
We call $p$ \emph{hyperbolic} if it is hyperbolic with respect to some direction.
\end{df}

\begin{rem}\label{nonzeroate}
\begin{enumerate}[(a)]
\item If $p\in\R[x]$ is hyperbolic with respect to $e\in\R^n$, then $p(e)\ne0$.
\item If $p\in\R[x]$ is hyperbolic, then the number of variables $n$ is of course greater or equal to $1$ since $\R^0$ contains
only the zero vector which can never be a hyperbolicity direction by definition.
\end{enumerate}
\end{rem}

The following is the analog of Proposition \ref{splitrz}.

\begin{pro}\label{splithb}
Let $p\in\R[x]$ and $e\in\R^n$ with $e\ne0$. Then
$p$ is hyperbolic in direction $e$ if and only if for each $a\in\R^n$, the univariate polynomial
\[p(a-te)\in\R[t]\]
splits (i.e., is a product of non-zero linear polynomials) in $\R[t]$. 
\end{pro}

\begin{proof}
The ``if'' direction is easy and the ``only if'' direction follows from the fundamental theorem of algebra. 
\end{proof}

\begin{df}\label{eigtr}
Let $p\in\R[x]$ be hyperbolic in direction $e$ and let $a\in\R^n$.
\begin{enumerate}[(a)]
\item The zeros of $p(a-te)\in\R[t]$ are called the \emph{eigenvalues}
of $a$ (with respect to $p$ in direction $e$). When we speak of their \emph{multiplicity}, we mean their multiplicity as roots of $p$.
\item The weighted sum of the eigenvalues of $a$ where the weights are the multiplicities is called the \emph{trace} of $a$
(with respect to $p$ in direction $e$). We denote it by $\tr_{p,e}(a)$.
\end{enumerate}
\end{df}

The following is the analog of Proposition \ref{detexample}.

\begin{pro}\label{detexamplehyperbolic}
Let $A_1,\ldots,A_n\in\C^{d\times d}$ be hermitian matrices and $e\in\R^n$ with $e\ne0$
such that $e_1A_1+\ldots+e_nA_n$ is definite. Then
\[p:=\det(x_1A_1+\ldots+x_nA_n)\in\R[x]\]
is hyperbolic (of degree $d$) in direction $e$.
\end{pro}

\begin{proof}
If $a\in\R^n$ and $\la\in\C$ with $p(a-\la e)=0$, then $\det(A-\la B)=0$ where $A:=a_1A_1+\ldots+a_nA_n\in\C^{d\times d}$
and $B:=e_1A_1+\ldots+e_nA_n\in\C^{d\times d}$ is definite. We have to show $\la\in\R$. WLOG $\la\ne0$. Then
$\det(B+\frac1\la(-A))=0$ and thus $\frac1\la\in\R$ by Lemma \ref{genev}. Hence $\la\in\R$.
\end{proof}

\subsection{Hyperbolicity cones versus rigidly convex sets}

The following is the analog of Definition \ref{dfrcs}.

\begin{df}\label{dfhbc}
Let $p\in\R[x]$ be hyperbolic in direction $e$.
Then we call
\[C(p,e):=\{a\in\R^n\mid\forall\la\in\R:(p(a-\la e)=0\implies\la\ge0)\}\]
the \emph{hyperbolicity cone} of $p$ with respect to $e$.
\end{df}

A priori, it is not clear that hyperbolicity cones are cones. We will see this in Theorem \ref{hbshift}
below but again it was already known to G\aa rding \cite[Theorem 2]{gar}.

\begin{pro}\label{hbrz}
Let $p\in\R[x]$. Then $p$ is hyperbolic in direction of the first unit vector $u$ of $\R^n$
if and only if its dehomogenization \[q:=p(1,x_2,\ldots,x_n)\in\R[x_2,\ldots,x_n]\]
is a real zero polynomial.
\end{pro}

\begin{proof}
First, suppose that $p$ is hyperbolic in direction of the first unit vector $u$. We show that
$q$ is a real zero polynomial. To this end, consider
$a\in\R^{n-1}$ and $\la\in\C$ such that $q(\la a)=0$. We have to show that $\la\in\R$. WLOG $\la\ne0$.
We show that $\mu:=\frac1\la\in\R$.
We have $p(u+\la(0,a))=q(\la a)=0$. By homogeneity, $p((0,a)+\mu u)=0$. Since $p$ is hyperbolic in direction
$u$, we have indeed $\mu\in\R$.

Conversely, suppose that $q$ is a real zero polynomial. To show that $p$ is hyperbolic in direction $u$,
we fix $a\in\R^n$ and $\la\in\C$ such that $p(a-\la u)=0$. We have to show that $\la\in\R$. If $\la=a_1$,
then this is trivial. Hence suppose $\la\ne a_1$ and set $\mu:=\frac1{a_1-\la}$. It is enough to show
that $\mu\in\R$. By homogeneity, we have $q(\mu a_2,\ldots,\mu a_n)=0$ and thus $\mu\in\R$.

For the rest of the proof suppose
that $q$ is a real zero polynomial (in particular, $q(0)\ne0$ as mentioned in Remark \ref{nonzeroatorigin})
and $a\in\R^{n-1}$.

(a)\quad We have
\begin{align*}
(1,a)\in C(p,u)&\overset{\text{Definition \ref{dfhbc}}}{\underset{\hphantom{\text{homogeneity}}}\iff}\forall\la\in\R:(p((1,a)-\la u)=0\implies\la\ge0)\\
&\overset{\hphantom{\text{homogeneity}}}\iff\forall\la\in\R:(p(1-\la,a)=0\implies\la\ge0)\\
&\overset{\hphantom{\text{homogeneity}}}\iff\forall\la\in\R_{<0}:p(1-\la,a)\ne0\\
&\overset{\hphantom{\text{homogeneity}}}\iff\forall\la\in\R_{>0}:p(1+\la,a)\ne0\\
&\overset{\text{homogeneity}}\iff\forall\la\in\R_{>0}:q\left(\frac a{1+\la}\right)\ne0\\
&\overset{\hphantom{\text{homogeneity}}}\iff\forall\la\in(0,1):q(\la a)\ne0\\
&\overset{q(0)\ne0}{\underset{\hphantom{\text{homogeneity}}}\iff}\forall\la\in[0,1):q(\la a)\ne0\\
&\overset{\text{Definition \ref{dfrzs}}}{\underset{\hphantom{\text{homogeneity}}}\iff}a\in C(q).
\end{align*}

(b)\quad We observe
\begin{align*}
(0,a)\in C(p,u)&\overset{\text{Definition \ref{dfhbc}}}{\underset{\hphantom{\text{homogeneity}}}\iff}\forall\la\in\R:(p((0,a)-\la u)=0\implies\la\ge0)\\
&\overset{\hphantom{\text{homogeneity}}}\iff\forall\la\in\R:(p(-\la,a)=0\implies\la\ge0)\\
&\overset{\hphantom{\text{homogeneity}}}\iff\forall\la\in\R_{<0}:p(-\la,a)\ne0\\
&\overset{\hphantom{\text{homogeneity}}}\iff\forall\la\in\R_{>0}:p(\la,a)\ne0\\
&\overset{\text{homogeneity}}\iff\forall\la\in\R_{>0}:q\left(\frac a\la\right)\ne0\\
&\overset{\hphantom{\text{homogeneity}}}\iff\forall\la\in\R_{>0}:q(\la a)\ne0\\
&\overset{q(0)\ne0}{\underset{\hphantom{\text{homogeneity}}}\iff}\forall\mu\in\R_{\ge0}:\forall\la\in[0,1):q(\la\mu a)\ne0\\
&\overset{\text{Definition \ref{dfrzs}}}{\underset{\hphantom{\text{homogeneity}}}\iff}
\forall\mu\in\R_{\ge0}:\mu a\in C(q).
\end{align*}
\end{proof}

\subsection{The homogeneous Helton-Vinnikov theorem}

\begin{rem}\label{modulematrix}
Let $R$ be a ring. We can multiply finitely many matrices over this ring provided that for each factor (but the last one) its number of columns matches the number of rows of the next factor.
Row and column vectors of elements can of course also be factors of such a product since they can be seen as matrices with a single row or column, respectively. Because of associativity of the matrix product, one does not have to specify
parentheses in such a product of matrices over $R$. Now if $M$ is an $R$-module, we can declare the product $AB\in M^{k\times m}$
of a matrix $A\in R^{k\times\ell}$ and a matrix $B\in M^{\ell\times m}$ in the obvious way. In this way, we can now also declare
products of finitely many matrices exactly as above even if the last factor is not a matrix over $R$ but over $M$. One has again the obvious
associativity laws that allow to omit parentheses. We will use this in the proofs of Theorem \ref{hbvinnikov} and
Proposition \ref{hbspecdet} below for the case $R=\R[x]$
and $M=\R[x]^{d\times d}$. For example, we will write $x^TA$ for $x_1A_1+\ldots+x_nA_n$ if $A$ is the column vector with entries
$A_1,\ldots,A_n\in\C^{d\times d}$.
\end{rem}

The following is the homogeneous version of the Helton-Vinnikov Theorem \ref{vinnikov}.

\begin{thm}[Helton and Vinnikov]\label{hbvinnikov}
If $p\in\R[x_1,x_2,x_3]$ is hyperbolic of degree $d$ in direction $e\in\R^3$ such that $p(e)=1$, then there exist symmetric $A_1,A_2,A_3\in\R^{d\times d}$
such that $e_1A_1+e_2A_2+e_3A_3=I_d$ and
\[p=\det(x_1A_1+x_2A_2+x_3A_3).\]
\end{thm}

\begin{proof}
We reduce it to the Helton-Vinnikov theorem for real zero polynomials \ref{vinnikov}. Denote the first unit vector in
$\R^3$ by $u$. Choose an orthogonal matrix $U\in\R^{3\times 3}$ such that $Ue=u$ and set $q:=p(U^Tx)$. Then
$q\in\R[x_1,x_2,x_3]$ is also homogeneous of degree $d$ but is hyperbolic in direction $u$. 
By Proposition \ref{hbrz},
the polynomial $r:=q(1,x_2,x_3)\in\R[x_2,x_3]$ is a real zero polynomial of some degree $d'\in\{0,\ldots,d\}$. By
the Helton-Vinnikov theorem, we can choose symmetric $B_2',B_3'\in\R^{d'\times d'}$ such that
$r=\det(I_{d'}+x_2B_2'+x_3B_3')$ (note that $r(0)=q(u)=q(Ue)=p(e)=1$).
Appending $d-d'$ zero columns and lines to $B_2'$ and $B_3'$, we get symmetric matrices
$B_2,B_3\in\R^{d\times d}$ such that $r=\det(I_{d}+x_2B_2+x_3B_3)$. Setting $B_1:=I_d$, we see that
$q=\det(x_1B_1+x_2B_2+x_3B_3)=\det(x^TB)$ where $B$ designates the row vector with entries $B_1,B_2,B_3$ and we use
the notation introduced in Remark \ref{modulematrix}. Hence $p=q(Ux)=\det((Ux)^TB)=\det(x^TU^TB)=\det(x^TA)$ where
$A:=U^TB$ is a row vector whose entries are symmetric matrices $A_1,A_2,A_3\in\R^{d\times d}$. It remains to check that
$e^TA=I_d$. Indeed, $e^TA=e^TU^TB=(Ue)^TB=u^TB=B_1=I_d$.
\end{proof}

The following is the homogeneous version Corollary \ref{hbvcor} which is a weaker version of the Helton-Vinnikov Theorem \ref{hbvinnikov}.
It can be derived from Corollary \ref{vcor} in exactly the same manner as we derived Theorem \ref{hbvinnikov} from Theorem \ref{vinnikov}.

\begin{cor}[Helton and Vinnikov]\label{hbvcor}
If $p\in\R[x_1,x_2,x_3]$ is hyperbolic of degree $d$ in direction $e\in\R^3$ such that $p(e)=1$, then there exist hermitian $A_1,A_2,A_3\in\C^{d\times d}$
such that $e_1A_1+e_2A_2+e_3A_3=I_d$ and
\[p=\det(x_1A_1+x_2A_2+x_3A_3).\]
\end{cor}

\subsection{Basics on hyperbolicity cones}

The following is the analog of Proposition \ref{rcsdet}.

\begin{pro}\label{hbdet}
Let $d\in\N_0$, $A_1,\ldots,A_n\in\C^{d\times d}$ be hermitian, $e\in\R^n$, $e\ne0$
\[e_1A_1+\ldots+e_nA_n\succ0\]
and
\[p=\det(x_1A_1+\ldots+x_nA_n).\]
Then
\[C(p,e)=\{a\in\R^n\mid a_1A_1+\ldots+a_nA_n\succeq0\}\]
and
\[C(p,e)\setminus Z(p)=\{a\in\R^n\mid a_1A_1+\ldots+a_nA_n\succ0\}.\]
\end{pro}

\begin{proof}
The second statement follows easily from the first. To prove the first, set
$A_0:=e_1A_1+\ldots+e_nA_n$,
let $a\in\R^n$ and set $B:=a_1A_1+\ldots+a_nA_n$. We have to show
\[(\forall\la\in\R:(\det(B-\la A_0)=0\implies\la\ge0))\iff B\succeq0.\]
Since $A_0$ is positive definite, there exists a (unique) positive definite matrix $\sqrt{A_0}$
matrix whose square is $A_0$. Rewriting both the left and right hand side of our claim, it becomes
\[(\forall\la\in\R:(\det(C-\la I_d)=0\implies\la\ge0))\iff C\succeq0\]
where $C:=\sqrt{A_0}^{-1}B\sqrt{A_0}^{-1}$.
This is clear.
\end{proof}

The following is the analog of Theorem \ref{rzshift}.

\begin{thm}[G\aa rding]\label{hbshift}
Let $p\in\R[x]$ be hyperbolic in direction $e$ and
$a\in C(p,e)\setminus Z(p)$ with $a\ne0$.
Then $a$ is also a hyperbolicity direction of $p$ and $C(p,e)=C(p,a)$.
\end{thm}

\begin{proof}
We suppose that
$e$ and $a$ are linearly independent since otherwise all statements are trivial.

We first show that $a$ is also a hyperbolicity direction of $p$.
Now let $b\in\R^n$ and $\la\in\C$ such that $p(b-\la a)=0$. We have to show $\la\in\R$.
The case where $b$ is a linear combination of $e$ and $a$ ist again easy and we leave
it to the reader. Hence we can now suppose that $e$, $a$ and $b$ are linearly independent.
By an affine transformation, we can even suppose that these are the first three unit vectors in $\R^n$.
Without loss of generality, we can thus assume that the
number of variables is $n=3$. Also WLOG $p(e)=1$.
By the Helton-Vinnikov Corollary \ref{hbvcor}, we can write
\[p=\det(x_1I_d+x_2A+x_3B)\]
with hermitian $A,B\in\C^{d\times d}$ where $d:=\deg p$.
The hypothesis $a\in C(p,e)\setminus Z(p)$ now translates into $A\succ0$ by Proposition \ref{hbdet}.
From $\det(B-\la A)=0$ and Lemma \ref{genev}, we get $\la\in\R$.

To prove the second statement, fix $b\in\R^n$.
We show that \[(*)\qquad b\in C(p,e)\iff b\in C(p,a).\]
If $b$ is a linear combination of $e$ and $a$, this is a an exercise that we leave to the reader
(make a case distinction according to the signs of the coefficients in this linear combination).
From now on suppose that $e$, $a$ and $b$ are linearly independent.
Suppose therefore that $a$ and $b$ are linearly independent. After an affine transformation, we can
even assume that $e$, $a$ and $b$ are the first three unit vectors. Hence we can reduce to the case where the number of variables $n$ equals $3$.
By the Helton-Vinnikov Corollary \ref{hbvcor}, we can choose hermitian matrices $A,B\in\C^{d\times d}$ such that
\[p=\det(x_1I_d+x_2A+x_3B)\]
so that
\[C(p,e)=\{c\in\R^3\mid c_1I_d+c_2A+c_3B\succeq0\}\]
by Proposition \ref{hbdet}. Since $A\succ0$ as mentioned above, Proposition \ref{hbdet} furthermore
gives
\[C(p,a)=\{c\in\R^3\mid c_1I_d+c_2A+c_3B\succeq0\}.\]
But now of course $C(p,e)=C(p,a)$.
\end{proof}

Now we come to the analog of Theorem \ref{rc-is-c}

\begin{thm}[G\aa rding]\label{hb-is-c}
Let $p\in\R[x]$ be hyperbolic in direction $e$. Then both $(C(p,e)\setminus Z(p))\cup\{0\}$ and 
$C(p,e)$ are cones.
\end{thm}

\begin{proof}
We start with $C:=(C(p,e)\setminus Z(p))\cup\{0\}$. We have to show $0\in C$, $C+C\subseteq C$ and
$\R_{\ge0}C\subseteq C$. One checks immediately the first and the third property. For the second one,
let $a,b\in C$. By Theorem \ref{hbshift}, $a$ and $b$ are then hyperbolicity directions of $p$ and we
have $C(p,a)=C(p,e)=C(p,b)$. Since $a\in C(p,b)$, we have $p(a+b)=p(a-(-1)b)\ne0$. It remains to show
that $a+b\in C(p,e)$. Because of $C(p,e)=C(p,b)$, we can equivalently show that $a+b\in C(p,b)$.
To this end, let $p(a+b-\la b)=0$ . We have to show that $\la\ge0$. Because of $p(a-(\la-1)b)=0$ and
$a\in C(p,b)$, we have even $\la-1\ge0$.

To prove that $C(p,e)$ is also convex, we observe that
\begin{align*}
C(p,e)&=\bigcap_{\ep>0}\left\{a\in\R^n~\middle|~\forall\la\in\R:\left(p(a-\la e)=0\implies\la+\ep>0\right)\right\}\\
&=\bigcap_{\ep>0}\left\{a\in\R^n~\middle|~\forall\la\in\R:\left(p(a-(\la-\ep)e)=0\implies\la>0\right)\right\}\\
&=\bigcap_{\ep>0}\left(\left\{a\in\R^n~\middle|~\forall\la\in\R:\left(p(a-\la e)=0\implies\la>0\right)\right\}-\ep e\right)\\
&=\bigcap_{\ep>0}((C(p,e)\setminus Z(p))-\ep e)\\
\end{align*}
is an intersection of convex sets.
\end{proof}

\subsection{Relaxing conic hyperbolic programs}

\begin{df}\label{dfhomogeneouspencil}
Let $p\in\R[[x]]$ be a power series with $p(0)\ne0$ and $d\in\N_0$. Consider the symmetric matrices
$A_0,A_1,\ldots,A_n\in\R^{(n+1)\times(n+1)}$ from Definition \ref{dfpencil}. 
Then denote 
\[M_{p,d}^*:=x_0A_0+x_1A_1+\ldots+x_nA_n=x_0M_{p,d}\left(\frac x{x_0}\right)\in\R[x]^{(n+1)\times(n+1)}.\]
\end{df}

The following is just a slight generalization of Lemma \ref{pencileval}.

\begin{lem}\label{homogeneouspencileval}
Let $p\in\R[x]$ be a power series with $p(0)\ne0$, $d\in\N_0$,
$a=(a_0\,a_1\,\ldots\,a_n)^T\in\R^{n+1}$ and \[v=(v_0\,v_1\,\ldots\,v_n)^T\in\R^{n+1}.\]
Then $v^TM_{p,d}^*(a)v=L_{p,d}((v_0+v_1x_1+\ldots+v_nx_n)^2(a_0+a_1x_1+\ldots+a_nx_n))$.
\end{lem}

\begin{proof}
Completely analogous to the proof of Lemma \ref{pencileval}.
\end{proof}

The following is a homogeneous version of Lemma \ref{pencilrotate}(a).

\begin{lem}\label{homogeneouspencilrotate}
Suppose $U\in\R^{n\times n}$ is an orthogonal matrix and consider the orthogonal matrix
\[\widetilde U:=\begin{pmatrix}1&0\\0&U\end{pmatrix}\in\R^{(n+1)\times(n+1)}.\]
Denote by $\widetilde x$ the column vector of variables $x_0,\ldots,x_n$.
If $p\in\R[[x]]$ is a power series with $p(0)\ne0$ and $d\in\N_0$, then
\[M_{p(Ux),d}^*=\widetilde U^TM_{p,d}^*(\widetilde U\widetilde x)\widetilde U.\]
\end{lem}

\begin{proof}
This follows easily by homogenization from Lemma \ref{pencilrotate}(a).
\end{proof}

\begin{lem}\label{indofu}
Let $e\in\R^n$ and $p\in\R[x]$ such that $e\ne0$ and $p(e)\ne0$. Let $d\in\N_0$. Denote the first unit vector in $\R^n$ by $u$.
Let $U_1,U_2\in\R^{n\times n}$ be orthogonal matrices
such that $U_ie=\|e\|u$ for $i\in\{1,2\}$. Setting $q_i:=p(U_i^Tx)\in\R[x]$ and
$r_i:=q_i(1,x_2,\ldots,x_n)\in\R[x_2,\ldots,x_n]$, we have $r_i(0)\ne0$ for $i\in\{1,2\}$ and
\[U_1^TM_{r_1,d}^*(U_1x)U_1=U_2^TM_{r_2,d}^*(U_2x)U_2.\]
\end{lem}

\begin{proof}
It is clear that $r_i(0)=q_i(e_i)=p(e)\ne0$ for $i\in\{1,2\}$.
The matrix $\widetilde W:=U_1U_2^T\in\R^{n\times n}$ is
orthogonal and satisfies
$\widetilde Wu=u$. Hence $\widetilde W$ can be written in the form
\[\widetilde W=\begin{pmatrix}1&0\\0&W\end{pmatrix}\in\R^{n\times n}\]
for some orthogonal matrix $W\in\R^{(n-1)\times(n-1)}$. We
have $q_1(\widetilde Wx)=q_2$ and therefore $r_1(Wy)=r_2$ where $y$ is the column vector with entries $x_2,\ldots,x_n$.
By Lemma \ref{homogeneouspencilrotate}, we have
\[M_{r_2,d}^*=\widetilde W^TM_{r_1,d}^*(\widetilde Wx)\widetilde W.\]
Replacing here $x$ by $U_2x$, we get
 \[M_{r_2,d}^*(U_2x)=U_2U_1^TM_{r_1,d}^*(U_1x)U_1U_2^T\]
and thus
 \[U_2^TM_{r_2,d}^*(U_2x)U_2=U_1^TM_{r_1,d}^*(U_1x)U_1\]
 as desired.
\end{proof}

The following is the analog of Definition \ref{dfpencil}.

\begin{df}\label{hbdfpencil}
Denote the first unit vector in $\R^n$ by $u$.
Let $p\in\R[x]$ be hyperbolic in direction $e$ of degree $d$
and choose an orthogonal matrix
$U\in\R^{n\times n}$ such that $Ue=\|e\|u$.
Consider the polynomial $q:=p(U^Tx)$ which obviously is hyperbolic in direction $u$ and the polynomial
$r:=q(1,x_2,\ldots,x_n)\in\R[x_2,\ldots,x_n]$ which is a real zero polynomial by Proposition \ref{hbrz}.
Then the homogeneous linear matrix polynomial
\[M_{p,e}:=U^TM_{r,d}^*(Ux)U\in\R[x]^{n\times n}\]
does not depend on the choice of $U$ by Lemma \ref{indofu}. We call it the
the \emph{pencil associated to $p$ with respect to $e$}.
Moreover, we call the cone
\[S(p,e):=\{a\in\R^n\mid M_{p,e}(a)\succeq0\}=\{a\in\R^n\mid M_{r,d}^*(Ua)\succeq0\}\]
the \emph{spectrahedral cone associated to $p$ with respect to $e$}.
\end{df}

\begin{rem}\label{scalingremark}
Let $p\in\R[x]$ be hyperbolic in direction $e$. Then it is obvious that $\la p$ is hyperbolic in direction $\mu e$ and
\[S(p,e)=S(\la p,\mu e)\]
for all $\la,\mu\in\R$ with $\mu>0$.
\end{rem}

\begin{pro}\label{trlin}
Let $p\in\R[x]$ be hyperbolic in direction $e$.
\begin{enumerate}[(a)]
\item If $p$ is non-constant, then the directional derivative
\[D_ep=\frac d{dt}p(x+t e)|_{t=0}\]
of $p$ in direction $e$ is again hyperbolic in direction $e$.
\item For all $a\in\R^n$, we have
\[e^TM_{p,e}(a)e=\|e\|^2\tr_{p,e}(a).\]
\item The map
\[\R^n\to\R,\ a\mapsto\tr_{p,e}(a)\]
is linear.
\end{enumerate}
\end{pro}

\begin{proof}
(a) follows essentially from Definition \ref{dfhb} and from Rolle's theorem. To prove the other statements,
we can suppose WLOG $\|e\|=1$ and $p(e)=1$ by Remark \ref{scalingremark} and Definition \ref{eigtr}. Set
$d:=\deg p$ and denote the first unit vector in $\R^n$ by $u$.
Choose an orthogonal matrix
$U\in\R^{n\times n}$ such that $Ue=u$ and consider $q:=p(U^Tx)$ which is hyperbolic in direction $u$ and
$r:=q(1,x_2,\ldots,x_n)\in\R[x_2,\ldots,x_n]$ which is a real zero polynomial with $r(0)=q(u)=p(e)=1$.
Write $\trunc_1r=1+b_2x_2+\ldots+b_nx_n$ with $b_2,\ldots,b_n\in\R$
and set $c:=Ua\in\R^n$.
By the Definitions \ref{dfhomogeneouspencil}, \ref{dfpencil} and \ref{hbdfpencil} as well as
Example \ref{moments3}, we have
\begin{align*}
e^TM_{p,e}(a)e&=e^TU^TM_{r,d}^*(c)Ue=u^TM_{r,d}^*(c)u\\
&=L_{r,d}(1)c_1+L_{r,d}(X_2)c_2+\ldots+L_{r,d}(X_n)c_n\\
&=dc_1+b_2c_2+\ldots+b_nc_n.
\end{align*}
Now fix $a\in\R^n$ and write $p(a-te)=p(e)\prod_{i=1}^d(\la_i-t)$ with $\la_1,\ldots,\la_d\in\R$. The coefficient of $t^{d-1}$ in the univariate
polynomial $f:=q(Ua-tu)=p(a-te)\in\R[t]$ is $(-1)^{d-1}p(e)\tr_{p,e}(a)=(-1)^{d-1}\tr_{p,e}(a)$. Since the coefficients of the monomials
\[x_1^d,x_1^{d-1}x_2,\ldots,x_1^{d-1}x_n\] in the polynomial $q$ are $1,b_2,\ldots,x_n$, respectively, it is an easy exercise to see that the coefficient
of $t^{d-1}$ in $f$ is also given by
\[(-1)^{d-1}(dc_1+b_2c_2+\ldots+b_nc_n)=(-1)^de^TM_{p,e}(a)e.\]
It follows that $e^TM_{p,e}(a)e=\tr_{p,e}(a)$. Since $a\in\R^n$ was arbitrary and $c=Ua$ depends of course linearly on $a$, we get also (c).
\end{proof}

The following is a sharpening of Proposition \ref{rcsspecdet}:

\begin{pro}\label{hbspecdet}
Suppose $d\in\N_0$ and $A_1,\ldots,A_n\in\C^{d\times d}$,
\[U:=\{v_1A_1+\ldots+v_nA_n\mid v_1,\ldots,v_n\in\R\}\]
and $(U,V)$ is an admissible couple (in particular, each $A_i$ is hermitian). Set
\[p:=\det(x_1A_1+\ldots+x_nA_n)\in\R[x]\]
and let $e\in\R^n\setminus\{0\}$ such that
\[e_1A_1+\ldots+e_nA_n=I_d.\]
\begin{enumerate}[(a)]
\item We have
\begin{align*}
C(p,e)&=\{a\in\R^n\mid\forall M\in V:\tr(M^2(a_1A_1+\ldots+a_nA_n))\ge0\}\qquad\text{and}\\
S(p,e)&=\{a\in\R^n\mid\forall M\in U:\tr(M^2(a_1A_1+\ldots+a_nA_n))\ge0\}.
\end{align*}
\item $C(p)\subseteq S(p,e)$
\item If $U$ is perfect, then $C(p)=S(p,e)$.
\end{enumerate}
\end{pro}

\begin{proof}
The first claim in (a)
follows immediately from Proposition \ref{hbdet} together with Definition \ref{couple}. To prove the second statement in (a), 
we choose an orthogonal matrix $W\in\R^{n\times n}$ such that $We=\|e\|u$ where $u$ denotes the first unit vector in $\R^n$. Set $q:=p(W^Tx)$ and consider the real zero polynomial
$r:=q(1,x_2,\ldots,x_n)\in\R[x_2,\ldots,x_n]$. According to Definition \ref{hbdfpencil}, we now have
\[S(p,e)=\{a\in\R^n\mid M_{r,d}^*(Wa)\succeq0\}.\]
We now use a lot the notation explained in Remark \ref{modulematrix}. Write $A$ for the column vector with entries
$A_1,\ldots,A_n$. Then $B:=WA$ is again a column vector whose entries are matrices $B_1,\ldots,B_n\in\C^{n\times n}$.
Because of $A=W^TB$, we obviously have that
\[U=\{v^TA\mid v\in\R^n\}=\{v^TB\mid v\in\R^n\},\]
i.e., the $A_i$ generate the same real vector space as the $B_i$. Substituting $x$ by $W^Tx$ in the equation $p=\det(x^TA)$ that
defines $p$, we get
\[q=p(W^Tx)=\det((W^Tx)^TA)=\det(x^TWA)=\det(x^TB).\]
Observing that $B_1=u^TB=u^TWA=(W^Tu)^TA=e^TA=I_n$, we see that
\[r=q(1,x_2,\ldots,x_n)=\det(I_d+x_2B_2+\ldots+x_nB_n)\]
and hence
\[v^TM_{r,d}^*v=\tr((v^TB)^2(x^TB)).\]
for all $v\in\R^n$ due to Lemma \ref{homogeneouspencileval}.
Substituting here $Wa$ for $x$, we get
\[v^TM_{r,d}^*(Wa)v=\tr((v^TB)^2(Wa)^TB)=\tr((v^TB)^2a^TW^TB)=\tr((v^TB)^2a^TA)
\]
for all $a,v\in\R^n$ and thus
\[M_{r,d}^*(Wa)\succeq0\iff\forall M\in U:\tr(M^2a^TA)\ge0\]
for all $a\in\R^n$.
\end{proof}

We can now prove the homogeneous version of Theorem \ref{relaxation}. For polynomials that have a hermitian determinantal representation like in Proposition \ref{hbspecdet}, it follows immediately from that lemma. For other polynomials, we will again need the Helton-Vinnikov theorem, this time in its form of Corollary \ref{hbvcor}.

\begin{thm}\label{relaxationsharpening}
Let $p\in\R[x]$ be hyperbolic in direction $e$. Then $C(p,e)\subseteq S(p,e)$.
\end{thm}

\begin{proof}
For $n\le3$, the claim follows from Proposition \ref{hbspecdet}(b) where we use the
Helton-Vinnikov Corollary \ref{hbvcor}
for $n=3$. Note that for $n\le2$, the $A_i$ in Proposition \ref{hbspecdet} can obviously
be chosen to be diagonal and we do not need Helton-Vinnikov.

We now suppose $n>3$ and reduce it to the already proven case $n=3$.
Fix $a\in C(p,e)$ and $v\in\R^n$. We have to show \[(*)\qquad v^TM_{p,e}(a)v\ge0.\]
Denote again by $u$ the first unit vector in $\R^n$ and
choose an orthogonal matrix $U\in\R^{n\times n}$ such that
$Ue=u$, $w:=Uv\in\R^3\times\{0\}\subseteq\R^n$ and
$b:=Ua\in\R^3\times\{0\}\subseteq\R^n$.
Set $q:=p(U^Tx)$ and consider the real zero polynomial \[r:=q(1,x_2,\ldots,x_n)\in\R[x_2,\ldots,x_n].\]
By Definition \ref{hbdfpencil}, our Claim $(*)$ can now more concretely be formulated as
\[(**)\qquad w^TM_{r,d}^*(b)w\ge0\]
where $d:=\deg p$.
By means of Lemma \ref{homogeneouspencileval}, the claim becomes
\[(***)\qquad L_{r,d}((w_1+w_2x_2+w_3x_3)^2(b_1+b_2x_2+b_3x_3))\ge0\]
where we took into account that $w,b\in\R^3\times\{0\}\subseteq\R^n$. Accordingly, we
now consider the homogeneous polynomial
\[\tilde q:=q(x_1,x_2,x_3,0\ldots,0)\in\R[x_1,x_2,x_3]\]
of degree $d$ which is hyperbolic in direction $u$ and the
real zero polynomial \[\tilde r:=r(x_2,x_3,0,\ldots,0)=\tilde q(1,x_2,x_3)\in\R[x_2,x_3].\]
By Lemma \ref{restriction}(a) applied on $r$ and $\tilde r$, we can rewrite $(***)$ by
\[(****)\qquad L_{\tilde r,d}((w_1+w_2x_2+w_3x_3)^2(b_1+b_2x_2+b_3x_3))\ge0.\]
By Lemma \ref{homogeneouspencileval} we are done if we can show
$M_{\tilde r,d}^*(b_1,b_2,b_3)\succeq0$.
By Definition \ref{hbdfpencil} where one takes $I_3$ for the orthogonal matrix, this means that
$(b_1,b_2,b_3)\in S(\tilde q,u)$.
By the already treated case $n=3$, it suffices to show that $(b_1,b_2,b_3)\in C(\tilde q,u)$.
This is equivalent to $ b\in C(q,u)$ which is in turn equivalent to our hypothesis $a\in C(p,e)$.
\end{proof}

\begin{rem}\label{hom-pol-rel-rem}
\begin{enumerate}[(a)]
\item
Let $p\in\R[x]$ be hyperbolic in direction $e$. Theorem \ref{relaxationsharpening} means that $v^TM_{p,e}(a)v\ge0$ for all $a\in C(p,e)$ and
$v\in\R^n$. For $v=e$, this means by Proposition \ref{trlin}(b) just that each element of the hyperbolicity cone $C(p,e)$ has nonnegative trace (with respect to $p$ in direction
$e$) which is clear by Definition \ref{eigtr} since it has even all eigenvalues nonnegative.
\item Let $p\in\R[x]$ be a polynomial of degree at most $d$ and denote by
\[q:=x_0^dp\left(\frac x{x_0}\right)\in\R[x_0,x]\] its degree $d$ homogenization.
Inspecting the proof of Proposition \ref{trlin} (and using variables $x_0,\ldots,x_n$ instead of $x_1,\ldots,x_n$,
the following enlightening interpretation of the polyhedron $P_d(p)$ defined in Definition \ref{halfspace} becomes now obvious:
Its elements are those $b\in\R^n$ such that $(1,b)\in\R^{n+1}$ has nonnegative trace with respect to the hyperbolic polynomial $q$
in direction of the first unit vector (confer Proposition \ref{hbrz}). With this in mind,
Remark \ref{polyhedral-relaxation-remark}(d) can now be read as an instance of the fact that an element has nonnegative trace if all its
eigenvalues are nonnegative (with respect to an hyperbolic polynomial and an hyperbolicity direction). For the same reason,
Item (a) of this remark can be seen as a generalization of Remark \ref{polyhedral-relaxation-remark}(d).
\end{enumerate}
\end{rem}


\section{The determinant of the general symmetric matrix}

In this section, we fix $d\in\N_0$, set $n:=d+\frac{d^2-d}2=\frac{d(d+1)}2$,
\[\La:=\{(i,j)\in\{1,\ldots,d\}^2\mid i\le j\},\]
choose a bijection $\rh\colon\La\to\{1,\ldots,n\}$ and set $x_{ij}:=x_{\rh(i,j)}$ for $(i,j)\in\La$ so that
$\R[x]=\R[x_{ij}\mid1\le i\le j\le n]$.

Given any vector $a$ of length $n$, we write in the following
\[[a]:=\begin{pmatrix}
a_{\rh(1,1)}&a_{\rh(1,2)}&\ldots&a_{\rh(1,d)}\\a_{\rh(1,2)}&a_{\rh(2,2)}&\ldots&a_{\rh(2,d)}\\\vdots&\vdots&&\vdots\\\vdots&\vdots&&\vdots\\
a_{\rh(1,d)}&a_{\rh(2,d)}&\ldots&a_{\rh(d,d)}\end{pmatrix}
\]
for the symmetric $d\times d$ matrix whose upper triangular part contains the entries of the ``long vector'' $a$
(always with respect to the order prespecified by $\rh$). Moreover, if $A=[a]$, then we write $a=\vec A$, i.e.,
$\vec a$ is a ``long vector'' that stores the entries in the upper triangular part of $A$.
In the following we often identify $\R^n$ with the real vector space of symmetric $d\times d$ matrices by means of the vector space isomorphism
$a\mapsto[a]$.

\begin{df}
We call
\[X:=[x]=\begin{pmatrix}
x_{11}&x_{12}&\ldots&x_{1d}\\x_{12}&x_{22}&\ldots&x_{2d}\\\vdots&\vdots&&\vdots\\\vdots&\vdots&&\vdots\\x_{1d}&x_{2d}&\ldots&x_{dd}\end{pmatrix}\in
\R[x]^{d\times d}\]
the \emph{general symmetric matrix} of size $d$.
\end{df}

\begin{pro}
The determinant $\det X$ of the general symmetric matrix $X$ of size $d$ is a homogeneous polynomial of degree $d$ that is hyperbolic with respect to
the identity matrix $I_d$. We have
\[C(\det X,I_d)=S(\det X,I_d)=\{A\in\R^{d\times d}\mid A\succeq0\}.\]
\end{pro}

\begin{proof}
It is clear that $\det X$ is homogeneous of degree $d$.
For $(i,j)\in\La$, denote by $A_{ij}$ the unique symmetric matrix whose upper triangular part has zeros everywhere except for a one entry at position $(i,j)$.
Then \[\det X=\det(x_1A_1+\ldots+x_nA_n)\]
and $A_1,\ldots,A_n$ form a basis of the space of symmetric matrices which is perfect by Remark \ref{perfectex}(a).
Proposition \ref{hbspecdet}(c) now says that $C(\det X,I_d)=S(\det X,I_d)$.
Moreover, it follows easily from Proposition \ref{hbdet} that
\[C(\det X,I_d)=\{A\in\R^{d\times d}\mid A\succeq0\}.\]
\end{proof}

\subsection{Saunderson's representation of the derived cone}

\begin{thm}
Denote the standard unit vectors in $\R^n$ by $u_1,\ldots,u_n$ and set $p:=\det X\in\R[x]$.
As in Definition \ref{hbdfpencil}, let $U$ be an orthogonal matrix in $\R^{n\times n}$ such that $U\vec I_d=\|\vec I_d\|u_1$,
consider $q:=p(U^Tx)$ and the real zero polynomial $r:=q(1,x_2,\ldots,x_n)\in\R[x_2,\ldots,x_n]$. Consider the matrices
$B_i:=[U^Tu_i]\in\R^{d\times d}$ for $i\in\{1,\ldots,n\}$.
Then the following hold:
\begin{enumerate}[(a)]
\item \[M_{r,d}^*(Ux)=d\sqrt d\begin{pmatrix}\tr(B_1XB_1)&\ldots&\tr(B_1XB_n)\\\vdots&&\vdots\\\tr(B_nXB_1)&\ldots&\tr(B_nXB_n)\end{pmatrix}\in\R[x]^{n\times n}\]
\item Suppose $d\ge1$ so that $n\ge1$ and consider
the pencil $N$ that arises from $M_{r,d}^*(Ux)\in\R[x]^{n\times n}$ by deleting the first row and the first column.
The matrices $B_2,\ldots,B_n$ form a basis of the vector space of symmetric trace zero matrices in $\R^{d\times d}$
so that $N$ is by (a) essentially the pencil for which Saunderson \cite[Theorem 2]{sau} has shown that
the linear matrix inequality $N(x)\succeq0$ defines the ``first derivative relaxation of the cone of psd matrices'', i.e.,
\[\{[a]\mid a\in\R^n,N(a)\succeq0\}=C(D_{I_n}\det X,I_n)\]
where $D_{I_n}\det X$ is the directional derivative of $\det X$ in direction $\vec I_n$ which is hyperbolic in direction $I_d$ by
Proposition \ref{trlin}(a).
\end{enumerate}
\end{thm}

\begin{proof}
First note that
\begin{multline}\tag{$*$}
\sum_{i=1}^n(u_i^TUx)B_i=\sum_{i=1}^n(u_i^TUx)[U^Tu_i]=\left[\sum_{i=1}^n(u_i^TUx)U^Tu_i\right]\\
=\left[U^T\sum_{i=1}^n(u_i^TUx)u_i\right]=\left[U^TUx\right]=[x]=X.
\end{multline}
Substituting here $x$ by $U^Tx$, we get from this
\begin{equation}\tag{$**$}
\sum_{i=1}^nx_iB_i=[U^Tx].
\end{equation}

\medskip
(a) \quad We have $p=\det X=\det([x_1u_1+\ldots+x_nu_n])=\det([I_nx])$ where $x$ is the column vector of variables $x_i$
and thus
\begin{align*}
q&=p(U^Tx)=\det([U^Tx])\overset{(**)}=\det(x_1B_1+\ldots+x_nB_n).
\end{align*}
Consequently, $r=\det(B_1+x_2B_2+\ldots+x_nB_n)$. Because of $B_1=[U^Tu_1]=\frac{I_d}{\|\vec I_d\|}=\frac{I_d}{\sqrt d}$, we have
$(\sqrt d)^dr=\det(I_d+x_2\sqrt dB_2+\ldots+x_n\sqrt dB_n)$.
Now Corollary \ref{cubictraces} together with $\sqrt dB_1=I_d$ implies that $L_{r,d}(1)=d\sqrt d\tr(B_1^3)$, $L_{r,d}(x_i)=d\sqrt d\tr(B_1^2B_i)$,
$L_{r,d}(x_ix_j)=d\sqrt d\tr(B_1B_iB_j)$ and $L_{r,d}(x_ix_jx_k)=d\sqrt d\tr(B_iB_jB_k)$ for all $i,j,k\in\{1,\ldots,n\}$.
By Definition \ref{dfpencil}(a), the entry in row $i$ and column $j$ of $M_{r,d}^*(Ux)$ is thus
\[d\sqrt d\sum_{k=1}^n(u_k^TUx)\tr(B_iB_kB_j)\overset{(*)}=d\sqrt d\tr(B_iXB_j)\]
for all $i,j\in\{1,\ldots,d\}$.

\medskip
(b) \quad Since $U^T$ is invertible, the vectors $\vec B_1,\ldots,\vec B_{n-1}$ are linearly independent in $\R^n$. Since $U^T$ is orthogonal, they are moreover
orthogonal to $U^T\vec u_1=\vec I_d$ with respect to the standard scalar product on $\R^n$. Summing up,
$B_1,\ldots,B_{n-1}$ are $n-1$ linearly independent symmetric matrices of trace zero and thus form a basis of the
vector space of symmetric trace zero matrices in $\R^{d\times d}$ which has dimension $n-1$.
\end{proof}

\section{The generalized Lax conjecture}\label{sec:glc}

The generalized Lax conjecture (GLC) has been stated by Helton and Vinnikov \cite[Section 6.1]{hv},
see also \cite[Page 63]{ren}.

\begin{conjecture}[generalized Lax conjecture, GLC]\label{glc} Each rigidly convex set is a spectrahedron.
\end{conjecture}

The following has been proven in \cite[Lemma 2.1]{hv} with heavy machinery from real algebraic geometry. We found an elementary approach which we present here.

\begin{lem}\label{sylvester}
Let $p,q\in\R[x]$ be non-constant real zero polynomials such that $C(p)=C(q)$. Then $p$ and $q$ have a common non-constant factor in $\R[x]$.
\end{lem}

\begin{proof}
For the duration of this proof, we call a line $L$ through the origin in $\R^n$ \emph{exceptional} if one of the polynomials $p$ and $q$ is constant on $L$.
Then actually both are constant on $L$ since if one of them, say $p$, is constant on $L$, then $L\subseteq C(p)=C(q)$ and so
$q$, being a real zero polynomial, is also constant on $L$.
If $L$ is a non-exceptional line then $p$ and $q$ thus both vanish somewhere on it and because of $C(p)=C(q)$ they must even have a common
root on $L$.

Now write $p=\sum_{i=0}^dp_i$ and $q=\sum_{j=0}^eq_j$ with $d:=\deg p\ge1$ and $e:=\deg q\ge1$ where $p_i\in\R[x]$ is homogeneous of degree $i$ and
$q_j\in\R[x]$
is homogeneous of degree $j$ for all $i$ and $j$. Consider now a kind of Sylvester matrix, namely
\[
S:=\begin{pmatrix}
p_d&p_{d-1}& \hdotsfor{2} & p_0\\
&p_d&p_{d-1}& \hdotsfor{2} & p_0\\
&& \ddots &&&& \ddots \\
&&&p_d&p_{d-1}& \hdotsfor{2} & p_0\\
q_e&q_{e-1}& \hdotsfor{2} & q_0\\
&q_e&q_{e-1}& \hdotsfor{2} & q_0\\
&& \ddots &&&& \ddots \\
&&&q_e&q_{e-1}& \hdotsfor{2} & q_0\\
\end{pmatrix}\in\R[x]^{(d+e)\times(d+e)}
\]
whose first $e$ rows are consecutive shifts of $(p_d,\ldots,p_0)$ and whose last $d$ rows are consecutive shifts of $(q_e,\ldots,q_0)$
where the empty space is filled up with zeros.

We claim that $S(a)$ is singular for all $a\in\R^n$. This is clear for $a=0$ or if $a$ spans an exceptional line since then the first column of $S(a)$
is zero. To prove the claim, it suffices to consider the case where $a$ spans a non-exceptional line $L$. But on such a line $p$ and $q$ have a
root in common so that there exists $\la\in\R$ with $p(\la a)=q(\la a)=0$. Then $\sum_{i=0}^dp_i(a)\la^{i+k}=\la^kp(a)=0$ and
$q=\sum_{j=0}^eq_j(a)\la^{j+k}=\la^kq(a)=0$ for all $k\in\N_0$ so that the transpose of the non-zero vector
\[\begin{pmatrix}\la^{d+e-1}&\la^{d+e-2}&\hdots&\la^2&\la&1\end{pmatrix}\]
lies in the kernel of $S(a)$. In particular, $S(a)$ is again singular.

By the claim just proven, we have $\det(S(a))=0$ for all $a\in\R^n$. This implies the polynomial identity $\det S=0$. Hence $S$ is singular as a matrix
over the field $K(x)=K(x_1,\ldots,x_n)$ of rational functions in the variables $x_1,\ldots,x_n$. Over this field, there is thus a non-trivial
linear dependance of its rows. Clearing denominators, this means that there exist a non-zero vector
\[v:=\begin{pmatrix}f_1&\hdots&f_e&g_1&\hdots&g_{d}\end{pmatrix}^T\in\R[x]^{d+e}\] such that $v^TS=0$. Denoting by $x_0$
an additional variable, we multiply this from the right with the vector
\[w:=\begin{pmatrix}1&x_0&x_0^2&\hdots&x_0^{d+e-1}\end{pmatrix}^T\in\R[x_0]^{d+e}\]
and obtain
\[\underbrace{\sum_{i=1}^ef_ix_0^{i-1}}_{=:f\in\R[x_0,x]}p^*+\underbrace{\sum_{j=1}^dg_jx_0^{j-1}}_{=:g\in\R[x_0,x]}q^*=v^TSw=0\]
where $p^*$ and $q^*$ are the homogenizations of $p$ and $q$, respectively, as introduced in Definition \ref{defhomogeneous}.
For every polynomial $h\in\R[x_0,x]$ we denote by $\deg_{x_0}h$ its degree
with respect to $x_0$, i.e., the degree of $h$ when it is seen as a polynomial in $x_0$ with coefficients from $K[x]$.
We have $\deg_{x_0}g\le d-1<d=\deg_{x_0}p^*$ so that $p^*$ cannot divide $g$ in $\R[x_0,x]$.
Therefore, when we look at the prime factorization of
\[fp^*=-gq^*\]
in the factorial ring $\R[x_0,x]$, we find an irreducible factor $r$ of $p^*$ that divides $q^*$ in $\R[x_0,x]$.
Then $r(1,x)$ divides $q=q^*(1,x)$ in $\R[x]$. It remains to show that $r(1,x)$ cannot be constant. Indeed, the only way this could happen would be that
$x_0$ divides $r$ which is impossible since $r(0,x)$ divides $p^*(0,x)=p_d\ne0$.
\end{proof}

\begin{lem}\label{monicrep}
Let $A_0,A_1,\ldots,A_n\in\R^{d\times d}$ be symmetric such that the origin is an interior point of the spectrahedron
\[S:=\{a\in\R^n\mid A_0+a_1A_1+\ldots+a_nA_n\succeq0\}.\]
Then there exists an invertible matrix $Q\in\R^{d\times d}$, $e\in\{0,\ldots,d\}$ and symmetric matrices $B_1,\ldots,B_n\in\R^{e\times e}$ such that
\begin{align*}
Q^TA_0Q&=\begin{pmatrix}I_e&0\\0&0\end{pmatrix}\qquad\text{and}\\
Q^TA_iQ&=\begin{pmatrix}B_i&0\\0&0\end{pmatrix}\text{ for }i\in\{1,\ldots,n\}.
\end{align*}
Consequently,
\[S=\{a\in\R^n\mid I_e+a_1B_1+\ldots+a_nB_n\succeq0\}.\]
\end{lem}

\begin{proof}
Choose an orthogonal matrix $U\in\R^{d\times d}$ such that $U^TA_0U$ is diagonal. Then choose a permutation matrix $P\in\R^{d\times d}$ such
that the diagonal entries of (the diagonal) matrix $P^TU^TA_0UP$ are $\la_1,\ldots,\la_e$ followed by $d-e$ zeros. Since $0\in S$, we have 
$A_0\succeq0$ so that $\la_1,\ldots,\la_e$ are nonnegative. Let $D\in\R^{d\times d}$ be a diagonal matrix with diagonal entries 
$\sqrt{\la_1},\ldots,\sqrt{\la_e}$ followed by $d-e$ arbitrary non-zero entries and set $Q:=UPD^{-1}\in\R^{d\times d}$. Then $Q$ is invertible and
\[Q^TA_0Q=\begin{pmatrix}I_e&0\\0&0\end{pmatrix}.\]
For each $i\in\{1,\ldots,n\}$, $Q^TA_iQ$ is symmetric and can therefore be written as
\[Q^TA_iQ=\begin{pmatrix}B_i&C_i\\C_i^T&D_i\end{pmatrix}\]
with $B_i\in\R^{e\times e}$, $C_i\in\R^{e\times(d-e)}$ and $D_i\in\R^{(d-e)\times(d-e)}$ such that $B_i$ and $D_i$ are symmetric.
Since $0$ is in the interior of $S$, there exists $\ep>0$ such that for all $\la\in\R$ with $-\ep<\la<\ep$ and all $i\in\{1,\ldots,n\}$, we have
$A_0+\la A_i\succeq0$ and thus $Q^TA_0Q+\la Q^TA_iQ\succeq0$. It follows that $D_i\succeq0$ and $-D_i\succeq0$ and thus $D_i=0$
for all $i\in\{1,\ldots,n\}$. Now it is an easy exercise to show that $C_i=0$ for all $i\in\{1,\ldots,n\}$.
\end{proof}

\begin{pro}\label{geoalg}
Let $p\in\R[x]$ be a real zero polynomial. Then the following are equivalent:
\begin{enumerate}[(a)]
\item For each irreducible factor $f$ of $p$ in $\R[x]$, the rigidly convex set $C(f)\subseteq\R^n$ is a spectrahedron.
\item There exist $q\in\R[x]$, $d\in\N_0$ and symmetric matrices $A_1,\ldots,A_n\in\R^{d\times d}$ such that
\[pq=\det(I_d+x_1A_1+\ldots+x_nA_n)\]
and $C(p)\subseteq C(q)$.
\end{enumerate}
In this case, $C(p)$ is a spectrahedron.
\end{pro}

\begin{proof}
$C(p)$ is of course the intersection over the $C(f)$ where $f$ runs over its irreducible factors. Using block diagonal matrices, one reduces therefore
easily to the case where $p$ is irreducible.

To show (a)$\implies$(b), we suppose that (a) holds. By Lemma \ref{monicrep}, we find $d\in\N_0$ and symmetric matrices $A_1,\ldots,A_n\in\R^{d\times d}$ such that \[C(p)=\{a\in\R^n\mid I_d+a_1A_1+\ldots+a_nA_n\succeq0\}.\]
Setting $r:=\det(I_d+x_1A_1+\ldots+x_nA_n)$, we have $C(r)=C(p)$ by Proposition \ref{rcsdet}. By the irreducibility of $p$ and
Lemma \ref{sylvester}, this implies that $p$ divides $r$ in $\R[x]$. Choose $q\in\R[x]$ such that $pq=r$. Since $r$ is a real zero polynomial
by Proposition \ref{detexample}, $q$ is also a real zero polynomial. Moreover it is obvious that $C(p)\cap C(q)=C(r)$. Together with
$C(r)=C(p)$, we obtain $C(p)\subseteq C(q)$.

To prove (b)$\implies$(a), let $q\in\R[x]$, $d\in\N_0$ and hermitian
matrices $A_1,\ldots,A_n\in\C^{d\times d}$ be given such that $pq=r:=\det(I_d+x_1A_1+\ldots+x_nA_n)$
and $C(p)\subseteq C(q)$. Then $C(p)=C(p)\cap C(q)=C(r)$ is a spectrahedron by Proposition \ref{detexample}
\end{proof}

\begin{rem}\label{lineality}
If $C\subseteq\R^n$ is a cone, then $-C:=\{-a\mid a\in C\}$ is again a cone and $C\cap-C$ is a subspace of $\R^n$ which is called the
\emph{lineality space} of $C$.
\end{rem}

The lemma can also easily be deduced from \cite[Fact 2.9]{bgls} proved by G\aa rding \cite[Theorem 3]{gar}.
Here we give a very short proof based on the hermitian version of the result of Helton and Vinnikov.

\begin{lem}\label{independentlineality}
Let $p\in\R[x]$ be hyperbolic in direction of the first unit vector $u$ of $\R^n$. Suppose that $m\in\{1,\ldots,n\}$ and
$C(p,u)\cap-C(p,u)=\{0\}\times\R^{n-m}\subseteq\R^n$. Then
$p\in\R[x_1,\ldots,x_m]$.
\end{lem}

\begin{proof}
We have to show $p=p(x_1,\ldots,x_m,0,\ldots,0)$. It suffices to show $p(x_1,a,b)=p(x_1,a,0)$ for all $a\in\R^{m-1}$ and $b\in\R^{n-m}$. Fix $a\in\R^{m-1}$ and $b\in\R^{n-m}$ and consider the polynomial $q:=p(x_1,x_2a,x_3b)\in\R[x_1,x_2,x_3]$ of degree $d$
which is obviously hyperbolic in direction of the first unit vector of $\R^3$.
Because of  the Helton-Vinnikov Corollary \ref{hbvinnikov}, we find hermitian matrices $A_2,A_3\in\C^{d\times d}$ such that $q=\det(x_1I_d+x_2A_2+x_3A_3)$. Because of $(0,0,b)\in C(p,u)\cap-C(p,u)$, we have that all roots of
the univariate polynomial $\det(A_3-tI_d)=q(-t,0,1)=p(-t,0,b)\in\R[t]$ are nonnegative and nonpositive and therefore zero.
So all eigenvalues of the hermitian matrix $A_3$ are zero and therefore $A_3=0$. Consequently, $q=\det(x_1I_d+x_2A_2)\in\R[x_1,x_2]$.
Hence $p(x_1,a,b)=q(x_1,1,1)=q(x_1,1,0)=p(x_1,a,0)$ as desired.
\end{proof}

A stronger version of the following lemma is folklore and can be found for example in \cite[Theorem 2.5.1]{web}. For convenience of the reader,
we present the version that we need with the corresponding simplified proof.

\begin{lem}\label{ray}
Let $S$ be a closed unbounded convex set in $\R^n$ that contains the origin. Then $S$ contains a ray, i.e., there exists $a\in\R^n\setminus\{0\}$ such
that $\{\la a\mid\la\in\R_{\ge0}\}\subseteq S$.
\end{lem}

\begin{proof}
Choose a sequence $(a_i)_{i\in\N}$ in $S\setminus\{0\}$ such that $\lim_{i\to\infty}\|a_i\|=\infty$. Consider the sequence
\[\left(\frac{a_i}{\|a_i\|}\right)_{i\in\N}\] of points on the unit sphere in $\R^n$. Since this sphere is compact, we may suppose that it converges to some
point $a\in\R^n$ with $\|a\|=1$. We claim that \[\{\la a\mid\la\in\R_{\ge0}\}\subseteq S.\] To this purpose, we fix $\la\in\R_{\ge0}$ and choose
$k\in\N$ such that
\[\frac\la{\|a_i\|}\le1\qquad
\text{and thus}\qquad
\frac\la{\|a_i\|}a_i\in S\]
for all $i\ge k$ because $S$ is convex and contains the origin. It follows that
\[\la a=\la\lim_{i\to\infty}\frac{a_i}{\|a_i\|}=\lim_{i\to\infty}\frac\la{\|a_i\|}a_i\in S\]
since $S$ is closed.
\end{proof}

\begin{thm}[equivalent formulations of GLC]
The following are equivalent:
\begin{enumerate}[(a)]
\item Each rigidly convex set is a spectrahedron, i.e., Conjecture \ref{glc} (GLC) holds.
\item For each real zero polynomial $p\in\R[x]$, there exist a polynomial $q\in\R[x]$, some $d\in\N_0$ and symmetric
matrices $A_1,\ldots,A_n\in\R^{d\times d}$ such that
\[pq=\det(I_d+x_1A_1+\ldots+x_nA_n)\]
and $C(p)\subseteq C(q)$.
\item Each hyperbolicity cone is spectrahedral.
\item Each compact rigidly convex set is a spectrahedron.
\end{enumerate}
\end{thm}

\begin{proof}
(a)$\iff$(b) follows directly from Proposition \ref{geoalg}, (a)$\implies$(d) is trivial.
and (c)$\implies$(a) follows immediately from Proposition \ref{hbrz}.

\medskip
It remains to show (d)$\implies$(c). Suppose $(d)$ holds and let $p\in\R[x]$ be hyperbolic in direction $e$. WLOG we suppose
that $p$ has positive degree as otherwise its hyperbolicity cone $C(p,e)$ is all of $\R^n$.
By Remark \ref{lineality}, the lineality space \[L:=C(p,e)\cap-C(p,e)\subseteq\R^n\] of the hyperbolicity cone $C(p,e)$ is a linear subspace of $\R^n$.
It consists of all elements of $\R^n$ all of whose eigenvalues (with respect to $p$ in direction $e$) in the sense of Definition \ref{eigtr}(a) are zero.
Since $C(p,e)$ consists of those elements all of whose eigenvalues are nonnegative, we see that
\[L=C(p,e)\cap H=-C(p,e)\cap H.\]
Thus $L$ is of course contained in \[H:=\{a\in\R^n\mid\tr_{p,e}(a)=0\}\subseteq\R^n\] by Definition \ref{eigtr}(b) which is a hyperplane by
Proposition \ref{trlin}(c) with
\[e\notin H\supseteq L\]
due to the positive degree of $p$. In particular, $\dim L=n-m$ for some $m\in\{1,\ldots,n\}$.

\smallskip
\textbf{Claim 1.} We can reduce to the case where
\begin{itemize}
\item $e$ is the first unit vector of $\R^n$,
\item $H=\{0\}\times\R^{n-1}\subseteq\R^n$ and
\item $L=\{0\}\times\R^{n-m}\subseteq\R^n$.
\end{itemize}

\smallskip
\textbf{Justification.}
Choose an invertible matrix $A\in\R^{n\times n}$ whose first column is $e$, whose remaining columns span $H$ and
whose last $n-m$ columns span $L$, i.e., $A$ maps
\begin{itemize}
\item the first unit vector $u$ of $\R^n$ to $e$,
\item the subspace $H':=\{0\}\times\R^{n-1}\subseteq\R^n$ onto $H$ and
\item the subspace $L':=\{0\}\times\R^{n-m}\subseteq\R^n$ onto $L$.
\end{itemize}
Setting $q:=p(Ax)\in\R[x]$, we have for each $a\in A$ the univariate polynomial identity 
$q(a-tu)=p(Aa-te)\in\R[t]$ which shows that $p$ is
hyperbolic in direction $u$ with \[C(p,e)=\{Aa\mid a\in C(q,u)\}\]
and $\tr_{q,u}(a)=\tr_{p,e}(Aa)$ for each $a\in A$. It follows that $H'=\{a\in\R^n\mid\tr_{q,u}(a)=0\}$ and
$L'=\{A^{-1}a\mid a\in L\}=C(q,u)\cap-C(q,u)$. This proves Claim 1.

\smallskip
\textbf{Claim 2.} We can further reduce to the case where
\begin{itemize}
\item $e$ is the first unit vector of $\R^n$,
\item $H=\{0\}\times\R^{n-1}\subseteq\R^n$ and
\item $L=\{0\}\subseteq\R^n$.
\end{itemize}

\smallskip
\textbf{Justification.} Suppose we are already in the situation described in Claim 1.
By Lemma \ref{independentlineality}, we then have that
\[p\in\R[x_1,\ldots,x_m].\]
Viewed as a polynomial in the variables $x_1,\ldots,x_m$, $p$ is clearly again hyperbolic with respect to the first unit vector $u$ of $\R^m$ and
we have obviously
\[C(p,e)=C(p,u)\times\R^{n-m}\]
as well as $\tr_{p,e}(a,0)=\tr_{p,u}(a)$ for all $a\in\R^m$.
It is therefore enough to show that the hyperbolicity cone $C(p,u)\subseteq\R^m$ is spectrahedral.
Finally, we have that
\begin{align*}
H'&:=\{0\}\times\R^{m-1}=\{a\in\R^m\mid(a,0)\in H\}=\{a\in\R^m\mid\tr_{p,e}(a,0)=0\}\\
&=\{a\in\R^m\mid\tr_{p,u}(a)=0\}
\end{align*}
and $L':=C(p,u)\cap-C(p,u)=\{a\in\R^m\mid(a,0)\in L\}=\{0\}$. This proves Claim 2.

\smallskip
\textbf{Claim 3.} Suppose we are in the situation of Claim 2 and consider
\[q:=p(1,x_2,\ldots,x_n)\in\R[x_2,\ldots,x_n].\]
Then $q$ is a real zero polynomial and its associated rigidly convex set
\[C(q)=\{(a_2,\ldots,a_n)\in\R^{n-1}\mid (1,a_2,\ldots,a_n)\in C(p,e)\}\subseteq\R^{n-1}\]
is compact.

\smallskip
\textbf{Justification.} By Proposition \ref{hbrz}(a), we only need to show that $C(q)$ is compact. Certainly, it is closed since, for example, it is an
intersection of closed half-spaces by Theorem \ref{intersection}. By Lemma \ref{ray}, it is enough to show that $C(q)$ does not contain a ray.
By Proposition \ref{hbrz}(b) this is equivalent to showing that $H\cap C(p,e)=\{0\}$ which is true by Claim 2 since $L=H\cap C(p,e)$.
This proves Claim 3.

\smallskip
\textbf{Claim 4.} Suppose we are in the situation of Claim 2. Then
\[C(p,e)=\{0\}\cup\left\{a\in\R^n\mid a_1>0,\frac{(a_2,\ldots,a_n)}{a_1}\in C(q)\right\}.\]

\smallskip
\textbf{Justification.} The inclusion from right to left follows easily from Claim 3 and the fact that $C(p,e)$ is a cone by Theorem \ref{hb-is-c}.
For the other inclusion, let $a\in C(p,e)\setminus\{0\}$. By Claim 3, it suffices to show that $a_1>0$. Writing $a=a_1e+b$ with $b\in H$, we see that
$0\le\tr_{p,e}(a)=a_1(\deg p)+0$. Since $p$ has positive degree, it follows that $a_1\ge0$. Moreover, if we had $a_1=0$, it would follow that
$\tr_{p,e}(a)=0$ and hence $a\in C(p,e)\cap H=L=\{0\}$ by Claim 2 which contradicts $a\ne0$. This proves Claim 4.

\smallskip
Now we can finally \textbf{conclude the proof} of (d)$\implies$(c). The rigidly convex set $C(q)\subseteq\R^{n-1}$
is compact by Claim 3 is thus a spectrahedron by hypothesis (d). Accordingly, we can choose some $e\in\N_0$ and symmetric matrices
$A_1,\ldots,A_n\in\R^{e\times e}$ such that
\[C(q)=\{(a_2,\ldots,a_n)\in\R^{n-1}\mid A_1+a_2A_2+\ldots+a_nA_n\succeq0\}.\]
We claim that
\[C(p,e)=\{a\in\R^n\mid a_1\ge0,\,a_1A_1+a_2A_2+\ldots+a_nA_n\succeq0\}\]
so that the hyperbolicity cone $C(p,e)$ can be defined by a linear matrix inequality of size $e+1$ (in block diagonal form with a block of size $1$
and a block of size $e$). The inclusion from left to right is immediately by the corresponding inclusion from Claim 4. The other inclusion follows
from the other inclusion in Claim 4 if we can exclude that there exists $(a_2,\ldots,a_n)\in\R^{n-1}\setminus\{0\}$ with
$a_2A_2+\ldots+a_nA_n\succeq0$. But this follows from the boundedness of $C(q)$ proved in Claim 3.
\end{proof}

\subsection{The real zero amalgamation conjectures}

In this subsection, we consider three tuples $x=(x_1,\ldots,x_\ell)$, $y=(y_1,\ldots,y_m)$ and $z=(z_1,\ldots,z_n)$ of $\ell+m+n$ distinct variables for some $\ell,m,n\in\N_0$.
The following question \cite[Problem 2.13]{ss} has been motivated by the implications it would have to GLC by what follows in Subsection \ref{subs:wrap} below.
Its study has been initiated by \cite{ss} but only very partial results have been obtained so far.

\begin{problem}[Sawall and Schweighofer, real zero amalgamation problem, RZAP]\label{rzap}
Suppose that $p\in\R[x,y]$ and $q\in\R[x,z]$ are real zero polynomials
with \[p(x,0)=q(x,0).\] When does there exist a real zero polynomial $r\in\R[x,y,z]$ (called an \emph{amalgam}) such that \[r(x,y,0)=p\qquad\text{and}\qquad r(x,0,z)=q\qquad?\]
\end{problem}

In the situation of RZAP, we call $x_1,\ldots,x_\ell$ the \emph{shared variables} and $r$ the \emph{amalgamation polynomial}.

In \cite[Example 6.1]{ss}, an example of two real zero polynomials is provided with $\ell=6$ shared variables and $m=n=1$ individual variables that do not possess an amalgam
in the sense above. For $\ell=2$, a counterexample of the same nature cannot exist \cite[Section 7]{ss}. In \cite[Section 7]{ss} even more motivation is given for the following
conjecture \cite[Conjecture 7.6]{ss}:

\begin{conjecture}[Sawall and Schweighofer, Weak real zero amalgamation conjecture, WRZAC]\label{wrzac} 
In the case of $\ell=2$ two joint variables, the real zero amalgamation problem \ref{rzap} is always solvable.
\end{conjecture}

A stronger form of this conjecture is the following \cite[Conjecture 7.7]{ss}.

\begin{conjecture}[Sawall and Schweighofer, Strong real zero amalgamation conjecture, SRZAC]\label{srzac} 
Let $\ell=2$, i.e., $x=(x_1,x_2)$, and $d\in\N_0$. Suppose
$p\in\R[x,y]$ and $q\in\R[x,z]$ are real zero polynomials of degree at most $d$ such that $p(x,0)=q(x,0)$. Then there exists a real zero polynomial $r\in\R[x,y,z]$ of degree at most $d$ such that \[p=r(x,y,0)\qquad\text{and}\qquad q=r(x,0,z).\]
\end{conjecture}

\subsection{Wrapping rigidly convex sets into spectrahedra and tying them with a cord}\label{subs:wrap}

\begin{lem}\label{plane-x1-x2}
Suppose that Conjecture \ref{srzac} (SRZAC) holds. Let $p\in\R[x_1,x_2,y_1,\ldots,y_m]=\R[x,y]$ be a real zero polynomial of degree $d\ge2$.
Set \[n:=\left(\frac{d(d+1)}2-3\right)\in\N_0.\] Then there exists a real zero polynomial $q\in\R[x,y,z_1,\ldots,z_n]=\R[x,y,z]$ of degree $d$ such that
$q(x,y,0)=p$,
\begin{align*}
C(p)&\subseteq\{(a,b)\in\R^2\times\R^m\mid (a,b,0)\in S(q)\}\qquad\text{ and}\\
 \{a\in\R^2\mid(a,0)\in C(p)\}&=\{a\in\R^2\mid (a,0,0)\in S(q)\}.
\end{align*}
In particular, $C(p)$ is contained in a spectrahedron that agrees with $C(p)$ on the plane spanned by the first two unit vectors.
\end{lem}

\begin{proof} WLOG $p(0)=1$. By the Helton-Vinnikov theorem \ref{vinnikov}, we find symmetric $A_1,A_2\in\R^{d\times d}$ such that
$p(x,0)=\det(I_d+x_1A_1+x_2A_2)$.

We now claim that we can find real symmetric matrices $B_1,\ldots,B_n\in\R^{d\times d}$ such that they span together with
$I_d$, $A_1$ and $A_2$
a subspace of the space of real symmetric matrices of size $d$ that is perfect in the sense of Definition \ref{couple}.
To this end, we distinguish three different cases.
In each of these cases, we will
use one of the subspaces that are perfect according to \ref{perfectex}(a).

If $A_1$ and $A_2$ are scalar multiples of the identity matrix then we can for example set $B_i:=0$ for $i\in\{1,\ldots,n\}$ and get the perfect subspace generated by $I_d$.

If the span of $I_d$, $A_1$ and $A_2$ is two-dimensional, then we can
jointly diagonalize $A_1$ and $A_2$ by conjugating them with a suitable orthogonal matrix and therefore can assume them to be diagonal.
Then the co-dimension of the span of $I_d$, $A_1$ and $A_2$ inside the space of real diagonal
matrices is $d-2=(d+1)-3$ which is at most $n$ because of $d\ge2$. Hence we find diagonal matrices $B_1,\ldots,B_n\in\R^{d\times d}$
such that the span of $I_d,A_1,A_2,B_1,\ldots,B_n$ is the space of all diagonal matrices which is again perfect.

The last case is where $I_d$, $A_1$ and $A_2$ are linearly independent. Then we complete them to a basis
$I_d,A_1,A_2,B_1,\ldots,B_n$ of the perfect space of real symmetric matrices of size $d$.

The claim is now proven and we choose  $B_1,\ldots,B_n\in\R^{d\times d}$ according to it.
By Conjecture~\ref{srzac} (SRZAC), $p\in\R[x,y]$ and
\[r:=\det(I_d+x_1A_1+x_2A_2+z_1B_1+\ldots+z_nB_n)\in\R[x,z]\] can be
amalgamated into a real zero polynomial $q\in\R[x,y,z]$ of degree $d$ such that $q(x,y,0)=p$ and $q(x,0,z)=r$.
By Theorem \ref{relaxation} applied to $q$, we have $C(q)\subseteq S(q)$. Together with
Lemma \ref{restriction}(b), this yields our first statement
\[C(p)=\{(a,b)\in\R^{2+m}\mid(a,b,0)\in C(q)\}\subseteq\{(a,b)\in\R^{2+m}\mid(a,b,0)\in S(q)\}.\]
We have $C(r)=S(r)$ by Proposition \ref{rcsspecdet}(c).
Together with Lemma \ref{restriction}, we get
\begin{multline*}
 \{a\in\R^2\mid(a,0)\in C(p)\}=\{a\in\R^2\mid(a,0,0)\in C(q)\}\\
=\{a\in\R^2\mid(a,0)\in C(r)\}=\{a\in\R^2\mid(a,0)\in S(r)\}\\
 \supseteq\{a\in\R^2\mid(a,0,0)\in S(q)\}\supseteq\{a\in\R^2\mid(a,0)\in C(p)\}
\end{multline*}
where the last inclusion follows from the already proven part of the lemma.
\end{proof}

\begin{lem}\label{plane2dim}
Suppose that Conjecture \ref{srzac} (SRZAC) holds. Let $p\in\R[x]=\R[x_1,\ldots,x_\ell]$ be a real zero polynomial of degree $d\ge2$.
Set \[m:=\left(\frac{d(d+1)}2-3\right)\in\N_0.\]
Let $U$ be a two-dimensional subspace of $\R^\ell$.
Then there exists a real zero polynomial $q\in\R[x_1,\ldots,x_\ell,y_1,\ldots,y_m]=\R[x,y]$
of degree $d$ such that $q(x,0)=p$ and the spectrahedron $S:=\{a\in\R^{\ell}\mid (a,0)\in S(q)\}\subseteq\R^\ell$
satisfies $C(p)\subseteq S$ and $U\cap C(p)=U\cap S$.
\end{lem}

\begin{proof}
Choose an orthogonal matrix $Q\in\R^{\ell\times\ell}$ such that $\{Qx\mid x\in U\}$ equals the span of the first two unit vectors $U'$ in $\R^\ell$.
Applying Lemma \ref{plane-x1-x2} to the real zero polynomial $p':=p(Qx)$, we obtain a real zero polynomial $q'\in\R[x,y]$ of degree $d$
with $q'(x,0)=p'$ such that the spectrahedron $S':=\{a\in\R^\ell\mid (a,0)\in S(q')\}$ satisfies $C(p')\subseteq S'$ and $U'\cap C(p')=U'\cap S'$.
Consider now the orthogonal matrix
\[Q':=\begin{pmatrix}Q&0\\0&I_m\end{pmatrix}\in\R^{(\ell+m)\times(\ell+m)}\]
and the real zero polynomial \[q:=q'\left(Q'^T\begin{pmatrix}x\\y\end{pmatrix}\right)=q'(Q^Tx,y)\in\R[x,y]\]
of degree $d$. Then $q(x,0)=q'(Q^{-1}x,0)=p'(Q^{-1}x)=p$. Moreover, we have
$C(p)=\{Q^Ta\mid a\in C(p')\}$,
\[S(q)=\left\{Q'^T\begin{pmatrix}a\\b\end{pmatrix}\mid(a,b)\in S(q')\right\}=\{(Q^Ta,b)\mid(a,b)\in S(q')\}\]
and thus $S:=\{a\in\R^{\ell}\mid (a,0)\in S(q)\}=\{Q^Ta\mid a\in S'\}$
by Proposition \ref{rotatespectrahedron}. Hence $C(p')\subseteq S'$ and $U'\cap C(p')=U'\cap S'$ easily translate into
the desired conditions $C(p)\subseteq S$ and $U\cap C(p)=U\cap S$.
\end{proof}

The following gives a very weak form of the generalized Lax conjecture (GLC) under the hypothesis of the strong real zero amalgamation conjecture (SRZAC).

\begin{thm}\label{wrapping}
Suppose that Conjecture \ref{srzac} (SRZAC) holds. Suppose $k,\ell\in\N_0$.
Fix a real zero polynomial $p\in\R[x]=\R[x_1,\ldots,x_\ell]$ polynomial of degree $d\ge2$ and a union $W$ of $k$ many
two-dimensional subspaces of $\R^\ell$. Then $C(p)$ is contained in a spectrahedron that agrees with $C(p)$ on $W$ and which is defined by a linear matrix inequality of size
\[s:=k\left(\ell+\frac{d(d+1)}2-2\right).\]
\end{thm}

\begin{proof}
One easily reduces to the case $k=1$. Then the claim follows easily from Lemma \ref{plane2dim}.
\end{proof}

If one supposes only the weak zero amalgamation conjecture (WRZAC) instead of SRZAC, we will get the almost same result in Theorem \ref{weak-wrapping} below:
Only the bound on the size of the linear matrix inequality will be slightly worse. To this end, we need the analogue of Lemma \ref{plane-x1-x2} above.

\begin{lem}\label{weak-plane-x1-x2}
Suppose that Conjecture \ref{wrzac} (WRZAC) holds. Let $p\in\R[x_1,x_2,y_1,\ldots,y_m]=\R[x,y]$ be a real zero polynomial of degree
$d\ge2$. Set \[n:=\left(\frac{d(d+1)}2-2\right)\in\N_0.\] Then there exists a real zero polynomial $q\in\R[x,y,z_1,\ldots,z_n]=\R[x,y,z]$ such that
$q(x,y,0)=p$,
\begin{align*}
C(p)&\subseteq\{(a,b)\in\R^{2+m}\mid (a,b,0)\in S_\infty(q)\}\qquad\text{ and}\\
 \{a\in\R^2\mid(a,0)\in C(p)\}&=\{a\in\R^2\mid (a,0,0)\in S_\infty(q)\}.
\end{align*}
In particular, $C(p)$ is contained in a spectrahedron that agrees with $C(p)$ on the plane spanned by the first two unit vectors.
\end{lem}

\begin{proof}
WLOG $p(0)=1$. By the Helton-Vinnikov theorem \ref{vinnikov}, we find symmetric $A_1,A_2\in\R^{d\times d}$ such that
$p(x,0)=\det(I_d+x_1A_1+x_2A_2)$.

We now claim that we can find real symmetric matrices $B_1,\ldots,B_n\in\R^{d\times d}$ such that they span together with
$A_1$ and $A_2$ a subspace of the space of real symmetric matrices of size $d$ that is perfect in the sense of Definition \ref{couple}.
To this end, we distinguish three different cases.

If $A_1=A_2=0$, then we simply set $B_i:=0$ for all $i\in\{1,\ldots,n\}$ since $\{0\}\subseteq\R^{d\times d}$ is trivially perfect.

If the real span of $A_1$ and $A_2$ is one-dimensional, then we can
jointly diagonalize $A_1$ and $A_2$ by conjugating them with a suitable orthogonal matrix and therefore can assume them to be diagonal.
Then the co-dimension of the span of $A_1$ and $A_2$ inside the space of real diagonal
matrices is $d-1=(d+1)-2$ which is at most $n$ because of $d\ge2$. Hence we find diagonal matrices $B_1,\ldots,B_n\in\R^{d\times d}$
such that the span of $A_1,A_2,B_1,\ldots,B_n$ is the space of all diagonal matrices which is again perfect.

The last case is where $A_1$ and $A_2$ are linearly independent. Then we complete them to a basis
$A_1,A_2,B_1,\ldots,B_n$ of the perfect space of real symmetric matrices of size $d$.

The claim is now proven and we choose  $B_1,\ldots,B_n\in\R^{d\times d}$ according to it.
By Conjecture~\ref{wrzac} (WRZAC), we find for
\[r:=\det(I_d+x_1A_1+x_2A_2+z_1B_1+\ldots+z_nB_n)\in\R[x,z]\]
a real zero polynomial $q\in\R[x,y,z]$ such that \[q(x,y,0)=p\qquad\text{and}\qquad\trunc_3q(x,0,z)=\trunc_3r.\]
By Theorem \ref{relaxation} and Remark \ref{whenvirtualdegreerises}, we have $C(q)\subseteq S(q)\subseteq S_\infty(q)$. Together with
Lemma \ref{restriction}(b), this yields our first statement
\[C(p)=\{(a,b)\in\R^{2+m}\mid(a,b,0)\in C(q)\}\subseteq\{(a,b)\in\R^{2+m}\mid(a,b,0)\in S_\infty(q)\}.\]
We have $C(r)=S_\infty(r)=S_\infty(q(x,0,z))$ by Proposition \ref{rcsspecdet}(d) and
Lemma \ref{dependsonlyoncubicpart}.
Together with Lemma \ref{restriction}, we have the chain of inclusions
\begin{multline*}
 \{a\in\R^2\mid(a,0)\in C(p)\}= \{a\in\R^2\mid a\in C(p(x,0))\}\\
= \{a\in\R^2\mid a\in C(r(x,0))\}=\{a\in\R^2\mid(a,0)\in C(r)\}\\
=\{a\in\R^2\mid(a,0)\in S_\infty(q(x,0,z))\}\supseteq\{a\in\R^2\mid(a,0,0)\in S_\infty(q)\}\\
\supseteq\{a\in\R^2\mid(a,0)\in C(p)\}
\end{multline*}
where the last inclusion follows from the already proven part of the lemma.
\end{proof}

\begin{thm}\label{weak-wrapping}
Suppose that Conjecture \ref{wrzac} (WRZAC) holds. Suppose $k,\ell\in\N_0$.
Fix a real zero polynomial $p\in\R[x]=\R[x_1,\ldots,x_\ell]$ polynomial of degree $d\ge2$ and a union $W$ of $k$ many
two-dimensional subspaces of $\R^\ell$. Then $C(p)$ is contained in a spectrahedron that agrees with $C(p)$ on $W$ and which is defined by a linear matrix inequality of size
\[s:=k\left(\ell+\frac{d(d+1)}2-1\right).\]
\end{thm}

\begin{proof}
Completely analogous to the proof of Theorem \ref{wrapping} where Lemma \ref{plane-x1-x2} is exchanged by Lemma \ref{weak-plane-x1-x2}.
\end{proof}

\subsection{Tying with a ribbon instead of a cord in the case of cubic real zero polynomials}

In the last two subsections, we showed that we can wrap rigidly convex sets into a spectrahedron and
tie them with finitely many cords provided Conjecture \ref{srzac} (SRZAC) or at least Conjecture
\ref{wrzac} (WRZAC) holds true. For rigidly convex sets defined by \emph{cubic} real zero polynomials,
we will be able to improve this. Namely, we can even tie by a two-dimensional version of cords, say
by ribbons. This means that we can make the spectrahedron agree on finitely many three-dimensional
(instead of two-dimensional) subspaces with the rigidly convex set. The technique for the proof is
almost literally the same except that we use instead of the Helton-Vinnikov Theorem \ref{vinnikov} now
the following complex version which is a version of the Helton-Vinnikov Corollary \ref{vcor}
which allows for one more variable:

\begin{thm}[Buckley and Košir]\label{buckleykosir}
If $p\in\R[x_1,x_2,x_3]$ is a cubic real zero polynomial with $p(0)=1$, then there exist hermitian matrices $A_1,A_2,A_3\in\C^{3\times 3}$ such that
\[p=\det(I_3+x_1A_1+x_2A_2+x_3A_3).\]
\end{thm}

\begin{proof} Under a certain smoothness assumption, this follows from \cite[Theorem 6.4]{bk}.
With a perturbation and limit argument, this smoothness assumption can be removed by standard
techniques. This is explained in the proof of \cite[Proposition 8]{kum1}.
\end{proof}

The details will be provided in future versions of this article.

\section*{Acknowledgments}
The author's doctoral students Alejandro González Nevado and David Sawall gave many hints on how to improve this article.
The author enjoyed discussion with Mario Kummer in certain details of this article. 
This work has been supported by the DFG grant SCHW 1723/1-1.

\end{document}